\documentclass[a4paper, 12pt]{amsart}
\usepackage{amsmath,mathrsfs,amsthm}
\usepackage{txfonts,rotating,tikz}
\usepackage[all]{xy}
\usepackage{color} 
\usepackage{hyperref}

\newcommand{\doi}[1]{\href{https://doi.org/#1}{\texttt{doi:#1}}}
\newcommand{\arxiv}[1]{\href{https://arxiv.org/abs/#1}{\texttt{arXiv:#1}}}

\usepackage[all]{xy}
\DeclareMathOperator{\End}{\mathrm{End}}
\DeclareMathOperator{\Iim}{\mathrm{Im}}
\DeclareMathOperator{\Hom}{\mathrm{Hom}}
\DeclareMathOperator{\Ker}{\mathrm{Ker}}
\DeclareMathOperator{\ad}{\mathrm{ad}}
\DeclareMathOperator{\Res}{\mathrm{Res}}

\newtheorem{theorem}[subsection]{Theorem}
\newtheorem{lemma}[subsection]{Lemma}
\newtheorem{remark}[subsection]{Remark}
\theoremstyle{definition}

\newtheorem{coro}[subsection]{Corollary}

\newtheorem{definition}[subsection]{Definition}
\newtheorem{proposition}[subsection]{Proposition}
\DeclareMathOperator{\Vir}{\mathrm{Vir}}
\DeclareMathOperator{\Id}{\mathrm{Id}}

\newtheorem*{theoremA}{Theorem A}
\newtheorem*{theoremB}{Theorem B}
\newtheorem*{theoremC}{Theorem C}
\newtheorem*{theoremD}{Theorem D}
\newtheorem*{theoremE}{Theorem E}

\DeclareMathOperator{\scrAd}{{\mathscr{A}d}}

\DeclareMathOperator{\Spec}{\mathrm{Spec}}
\DeclareMathOperator{\Lie}{\mathrm{Lie}}
\DeclareMathOperator{\Exp}{\mathrm{Exp}}
\newcommand{\fh}{\mathfrak{h}}
 \newcommand{\fl}{\mathfrak{l}}
 
 \newcommand{\fn}{\mathfrak{n}}
 \newcommand{\fp}{\mathfrak{p}}
 
 \newcommand{\fs}{\mathfrak{s}}

\newcommand{\fb}{\mathfrak{b}}
\newcommand{\fg}{\mathfrak{g}}

\DeclareMathOperator{\red}{\mathrm{red}}

\DeclareMathOperator{\sAlg}{\mathscr{A}lg}
\DeclareMathOperator{\Set}{\mathscr{S}et}

\DeclareMathOperator{\Ad}{\mathrm{Ad}}
\def\fprod#1#2#3{{#1}{\displaystyle\mathop{\times}_{#2}}{#3}}

\DeclareMathOperator{\Mor}{\mathrm{Mor}}
\DeclareMathOperator{\Rep}{\mathrm{Rep}}
\DeclareMathOperator{\Alg}{\mathrm{Alg}}

\DeclareMathOperator{\an}{\mathrm{an}}

\DeclareMathOperator{\Iid}{\mathrm{Id}}

\DeclareMathOperator{\SL}{SL}

\def\dprod#1#2#3{{#1}{\displaystyle\mathop{\otimes}_{#2}}{#3}}

\newcommand{\bc}{\mathbb{C}}

\newcommand{\bz}{\mathbb{Z}}
\DeclareMathOperator{\fppf}{\mathrm{fppf}}
\newcommand{\beqn}{\begin{equation}}
\newcommand{\eeqn}{\end{equation}}
\begin{document}

\title{Conformal blocks for Galois covers of algebraic curves}

\author{Jiuzu Hong}
\address{\textrm{Jiuzu Hong} \newline \indent 
Department of Mathematics, University of North Carolina at Chapel Hill, Chapel Hill, NC 27599-3250, U.S.A.}
\email{  jiuzu@email.unc.edu}
\author{Shrawan Kumar}
\address{ \textrm{Shrawan Kumar} \newline \indent 
Department of Mathematics, University of North Carolina at Chapel Hill, Chapel Hill, NC 27599-3250, U.S.A.}
\email{ shrawan@email.unc.edu}
\keywords{Twisted affine Kac-Moody Lie algebras, conformal blocks, propagation of vacua, factorization theorem, Hurwitz stack, flat projective connection, parahoric Bruhat-Tits group schemes, moduli stack of bundles}
\makeatletter
\@namedef{subjclassname@2020}{\textup{2020} Mathematics Subject Classification}
\makeatother
\subjclass[2020]{17B67, 17B68, 14H81, 17B81, 14D21, 14H60, 81R10}
\maketitle
\begin{abstract}
We study the spaces of  twisted conformal blocks attached to a $\Gamma$-curve $\Sigma$ with marked $\Gamma$-orbits and an action of $\Gamma$ on a simple Lie algebra $\fg$, where $\Gamma$ is a finite group. We prove that if $\Gamma$ stabilizes a Borel subalgebra of $\fg$, then Propagation Theorem and Factorization Theorem hold. We endow a flat projective connection  on the sheaf of twisted conformal blocks attached to a  smooth family of pointed  $\Gamma$-curves; in particular, it is locally free.  We also prove that  the sheaf of twisted conformal blocks on the stable compactification of Hurwitz stack is locally free.  

Let $\mathscr{G}$ be the parahoric Bruhat-Tits group scheme on the quotient curve $\Sigma/\Gamma$ obtained via the $\Gamma$-invariance of Weil restriction associated to  $\Sigma$ and the simply-connected simple algebraic group $G$ with Lie algebra $\fg$.  We prove that the space of twisted conformal blocks can be identified with the space of generalized theta functions on the moduli stack of quasi-parabolic $\mathscr{G}$-torsors on $\Sigma/\Gamma$ when the level $c$ is divisible by $|\Gamma|$ (establishing a conjecture due to Pappas-Rapoport).
\end{abstract}

\tableofcontents

\section{Introduction}
The Wess-Zumino-Witten model is a type of two dimensional conformal field theory, which associates to an algebraic curve  with marked points and integrable highest weight modules of an affine Kac-Moody Lie algebra associated to the points,  a finite dimensional vector space consisting of conformal blocks. The space of conformal blocks has many important  properties including Propagation of Vacua and Factorization.  Deforming  the pointed algebraic curves in  a family, we get a sheaf of conformal blocks. This  sheaf admits a flat projective   connection when the family of pointed curves is a smooth family.  The mathematical theory of conformal blocks was first established in a pioneering work by Tsuchiya-Ueno-Yamada \cite{TUY} where all these properties were obtained. All the above  properties are important ingredients in the proof of the celebrated Verlinde formula for the dimension of the space of conformal blocks (cf. \cite{Be,Fa, Ku2, So1, V}).This theory has  a geometric counterpart  in  the theory of  moduli spaces of principal bundles over algebraic curves and also the moduli of curves and its  stable compactification.  

In this paper we study a twisted theory of conformal blocks on Galois covers of algebraic curves.  More precisely, we consider an algebraic curve $\Sigma$ with an action of a finite group $\Gamma$. Moreover, we take  a group homomorphism  $\phi: \Gamma\to {\rm Aut}(\fg)$ of $\Gamma$ acting on a simple Lie algebra $\fg$. Given any smooth point $q\in \Sigma$, we  attach an affine Lie algebra $\hat{L}(\fg, \Gamma_q)$ defined below Lemma \ref{evaluation_lem}
(in general a twisted affine Lie algebra), where $\Gamma_q$ is the stabilizer group of $\Gamma$ at $q$.  The integrable highest weight representations of $\hat{L}(\fg, \Gamma_q)$ of level $c$ (where $c$ is a positive integer) are  parametrized by certain finite set $D_{c,q}$ of dominant weights of the reductive Lie algebra $\fg^{\Gamma_q}$, i.e., for any $\lambda\in D_{c,q}$ we attach an integrable highest weight representation $\mathscr{H}(\lambda)$ of $\hat{L}(\fg, \Gamma_q)$ of level $c$ and conversely
(cf.\,Section \ref{Kac_Moody_Section}).   Given a collection $\vec{q}:=(q_1,\cdots, q_s)$ of smooth points in $\Sigma$ such that their $\Gamma$-orbits are disjoint and  a collection of weights $\vec{\lambda} = (\lambda_1, \dots, \lambda_s)$  with $\lambda_i\in D_{c,q_i}$, we consider the representation $\mathscr{H} (\vec{\lambda}):=
\mathscr{H}(\lambda_1) \otimes \dots \otimes \mathscr{H}(\lambda_s)$ (cf. Definition \ref{def1.2}). Now, 
 define the associated space of {\it twisted covacua} (or {\it twisted dual conformal blocks}) as follows:
\[  \mathscr{V}_{\Sigma, \Gamma, \phi}(\vec{q}, \vec{\lambda}):=\frac{\mathscr{H}(\vec{\lambda}) }{\fg[\Sigma\backslash \Gamma\cdot \vec{q} ]^\Gamma \cdot \mathscr{H}(\vec{\lambda}) }, \]
where  $\fg[\Sigma\backslash \Gamma\cdot \vec{q} ]^\Gamma $ is the Lie algebra of $\Gamma$-equivariant regular  functions from $\Sigma\backslash \Gamma\cdot \vec{q}$ to $\fg$ acting on the $i$-th factor $\mathscr{H}(\lambda_i)$ of $\mathscr{H}(\vec{\lambda})$ via its Laurent series expansion at $q_i$.  In this paper we often work with a more intrinsic but equivalent definition of the space of twisted covacua (see Definition \ref{def1.2}), where we work with marked $\Gamma$-orbits.   

The following {\it Propagation of Vacua} is the first main result of the paper (cf.\,Corollary \ref{coro2.2.3} (a)). 
\begin{theoremA}
\label{theoremA}
{\it Assume that $\Gamma$ stabilizes a Borel subalgebra of $\fg$. Let $q$ be a smooth point of $\Sigma$ such that $q$ is not $\Gamma$-conjugate to any point $\vec{q}$. 
Assume further that $0\in D_{c, q}$ (cf. Corollary \ref{newcoroweight0}). Then, we have the following isomorphism of spaces of twisted covacua:}
\[  \mathscr{V}_{\Sigma, \Gamma, \phi}(\vec{q}, \vec{\lambda} )\simeq \mathscr{V}_{\Sigma, \Gamma, \phi}\left((\vec{q},q) ,  (\vec{\lambda}, 0 )\right) .  \]
\end{theoremA}
In fact, a  stronger version of Propagation Theorem is  proved (cf.\,Theorem \ref{Propagation_thm} and Corollary  \ref{coro2.2.3} (b)).   Even though, we generally   follow the argument given in  \cite[Proposition 2.3]{Be}, in our  equivariant setting we need to generalize some important ingredients.
 For example,  the fact that
\[  \text{ `` The endomorphism   $X_{-\theta}\otimes f$ of $\mathscr{H}$ is locally nilpotent for all $f\in \mathscr{O}(U)$ " }   \]
 in the proof of Proposition 2.3 of \cite{Be}, can not easily be generalized to the  twisted case. To prove an analogous result, we need to assume that $\Gamma$ stabilizes a Borel subalgebra of $\fg$, and use Lemma \ref{lemma 1.3} crucially. It will be  interesting to see if  this assumption can be removed.

Let $q$ be a nodal point in $\Sigma$. Assume that the action of $\Gamma$ at $q$ is stable (see Definition \ref{stable_action}) and  the stabilizer group $\Gamma_q$ does not exchange the two formal branches around $q$.  Let $\Sigma'$ be the normalization of $\Sigma$ at the points  $\Gamma\cdot q$, and let $q',q''$  be the two smooth points in $\Sigma'$ over  $q$.  The following {\it Factorization Theorem} is our second main result (cf.\,Theorem 
\ref{thm3.1.2}).
\begin{theoremB} {\it Assume that $\Gamma$ stabilizes a Borel subalgebra of $\fg$. Then,
there exists a natural isomorphism:
\[\mathscr{V}_{\Sigma, \Gamma, \phi}(\vec{q}, \vec{\lambda}  )\simeq  \bigoplus_{\mu\in D_{c,q''}} \mathscr{V}_{\Sigma', \Gamma, \phi}\left((\vec{q}, q', q''), (\vec{\lambda}, \mu^*, \mu )\right),   \]
where $\mu^*$ is the dominant weight of $\fg^{\Gamma_{q'}}$ such that $V(\mu^*)$ is the dual representation $V(\mu)^*$ of $\fg^{\Gamma_q}=\fg^{\Gamma_{q'}}=\fg^{\Gamma_{q''}}$.}
\end{theoremB}
The formulation of the Factorization Theorem in the twisted case is a bit more delicate, since  the parameter sets $D_{c, q'}$ and $D_{c,q''}$ attached to the points  $q', q''$ are different in general; nevertheless they are related by the dual of representations under the assumption  that the action of $\Gamma$ at the simple node $q$ is stable and the stabilizer group $\Gamma_q$ does not exchange the branches
(cf.\,Lemma \ref{lemma 4.2}). Its proof  requires additional care (from that of the   untwisted case) at several places.  The assumption that  $\Gamma$ stabilizes a Borel subalgebra of $\fg$ also appears in this theorem as  we use the  Propagation Theorem in its proof.  

As proved in Lemma \ref{lem2.1.3}, the space of twisted covacua is finite dimensional. We sheafify the notion of twisted covacua associated to a family of  $s$-pointed $\Gamma$-curves as in Definition  \ref{def_sheaf_conformal_block} and show that
given a family $(\Sigma_T, \vec{q})$ of $s$-pointed $\Gamma$-curves over an irreducible scheme $T$ and weights $\vec{\lambda}=(\lambda_1, \dots, \lambda_s)$ with $\lambda_i\in D_{c, q_i}$ as above, one can functorially attach a coherent sheaf $\mathscr{V}_{\Sigma_T, \Gamma, \phi}(\vec{q}, \vec{\lambda}) $ of twisted covacua over the base $T$ (cf. Theorem \ref{covacua_coherent_basechange}).  As explained below, we generalize the construction to define a coherent sheaf of twisted covacua over the Hurwitz stack $\overline{\mathscr{H}M}_{g, \Gamma,\eta}$. 

   We prove the following stronger theorem (cf.\,Theorems \ref{Locally_free_smooth_base} and \ref{thm3.4.1}).
\begin{theoremC} {\it Assume that the family $\Sigma_T \to T$ is a smooth  family over a smooth base $T$. 
Then, the  sheaf $\mathscr{V}_{\Sigma_T, \Gamma, \phi}(\vec{q}, \vec{\lambda}) $  is  locally free of finite rank over  $T$.  In fact,  there exists a flat projective connection on  $\mathscr{V}_{\Sigma_T, \Gamma, \phi}(\vec{q}, \vec{\lambda}) $. } 
\end{theoremC}
This theorem  relies mainly on the Sugawara construction for the  twisted affine Kac-Moody algebras. In the untwisted case, this construction is quite  well-known (cf.\,\cite[\S 12.8]{Ka}). In the twisted case, the construction can be found in \cite{KW,W},  where the formulae are written in terms of the  abstract Kac-Moody presentation of $\hat{L}(\fg, \sigma)$, where $\sigma$ is a finite order automorphism of $\fg$.  For our application, we require the formulae in terms of the affine realization of $\hat{L}(\fg, \sigma)$ as  a central extension of the twisted loop algebra $\fg((t))^{\sigma}$. We present such a formula in (\ref{Virasoro_I}, \ref{Virasoro_II}) in Section \ref{Sugawara_section}, which might be new (to our knowledge).

Let $\overline{\mathscr{H}M}_{g, \Gamma,\eta}$ be the Hurwitz stack of stable $s$-pointed $\Gamma$-curves of genus $g$ with marking data $\eta$ at the marked points such that the set of $\Gamma$-orbits of the marked points contains the full ramification divisor (cf.\,Definition \ref{ramification_datum}). Then,  $\overline{\mathscr{H}M}_{g, \Gamma,\eta}$ is a smooth and proper Deligne-Mumford stack of finite type (cf.\,Theorem \ref{thm8.8}). We can attach a collection $\vec{\lambda}$ of dominant weights to the marking data $\eta$, and associate a coherent sheaf  $\mathscr{V}_{g, \Gamma, \phi}(\eta,  \vec{\lambda})$ of twisted covacua over the Hurwitz stack $\overline{\mathscr{H}M}_{g, \Gamma,\eta}$. {\it The presence of the Hurwitz stack is a new phenomenon in the twisted theory.} We prove the following theorem (cf.\,Theorem \ref{Hurwitz_locally_free}).
\begin{theoremD} {\it Assume that $\Gamma$ stabilizes a Borel subalgebra of $\fg$. Then,
the sheaf $\mathscr{V}_{g, \Gamma, \phi}(\eta,  \vec{\lambda})$ is locally free over the stack $\overline{\mathscr{H}M}_{g, \Gamma,\eta}$ .} 
\end{theoremD}

Our  proof of this theorem follows closely the work of Looijenga \cite{L} in the non-equivariant setting; in particular,  we use the canonical smoothing deformation of nodal curves (Lemma \ref{smoothing_construction}) and gluing tensor elements (Lemma \ref{Gluing_tensor_lemma} and the construction before that).  The Factorization Theorem also plays a crucial role in the proof.  In the case $\Gamma$ is cyclic,  Theorem D together with the Factorization Theorem allows us to reduce the computation of the dimension of the space of  twisted covacua to the case of cyclic covers of projective line with three marked points (see Remark \ref{remark8.11} (1)).  
\vspace{0.5em}

When $\Gamma$ is of prime order and the marked points are unramified,  the space $\mathscr{V}_{\Sigma, \Gamma, \phi}(\vec{q}, \vec{\lambda} )$ was studied earlier by Damiolini \cite{D}, where she proved the results described above in this case under some more constraints.  Our work is a vast generalization of her work, since we do not need to put any restrictions on the $\Gamma$-orbits, and the only restriction on $\Gamma$ is that $\Gamma$ stabilizes a Borel subalgebra of $\fg$ (when $\Gamma$ is a cyclic group it automatically holds).
 In particular, when $\Gamma$ has nontrivial stabilizers at the marked points, general twisted affine Kac-Moody Lie algebras and their representations occur naturally in this twisted theory of conformal blocks. Damiolini's work dealt  with the untwisted affine Lie algebras since only the unramified points are marked in her setting.  In our work, Kac's theory of twisted affine Lie algebras associated to finite order automorphisms and related Sugawara operators in the twisted setting are extensively employed. These new features bring considerably more Lie theoretic complexity for  the results stated above, which enriches the twisted theory in a most natural way.  Notably, the proof of Theorem A (or Theorem \ref{Propagation_thm}) is highly technical, where we have to introduce the technical condition that the finite group $\Gamma$ stabilizes a Borel subalgebra in $\mathfrak{g}$.  Furthermore, the Hurwitz stack of  $\Gamma$-curves with only unramified points marked is in general not proper, i.e., such  a pointed smooth $\Gamma$-curve may degenerate to a $\Gamma$-curve with non-free nodal $\Gamma$-orbits. Accordingly, it is desirable to have factorization theorem (Theorem B) for the $\Gamma$-curves with general nodal $\Gamma$-orbits, which naturally involves the twisted conformal blocks with ramified points marked.  Our more general theory of twisted conformal blocks fits perfectly with the compactification of Hurwitz stacks, and marking ramified points is very crucial towards a Verlinde type formula for twisted conformal blocks of any kind.

There were also some earlier works related to  the twisted theory of conformal blocks. For example Frenkel-Szczesny \cite{FS} studied the twisted modules over Vertex algebras on algebraic curves, and Kuroki-Takebe \cite{KT} studied a twisted Wess-Zumino-Witten model on elliptic curves. We also learnt from S.\,Mukhopadhyay that he  obtained certain results (unpublished) in this direction in the setting of diagram automorphisms.
\vspace{0.5em}

In the usual (untwisted) theory of conformal blocks, the space of conformal blocks has a  beautiful  geometric interpretation in that  it can be identified with the space of generalized theta functions on the moduli space of parabolic $G$-bundles over the algebraic curve, where $G$ is the simply-connected simple algebraic group associated to $\fg$ (cf. Beauville-Laszlo \cite{BL}, Faltings \cite{Fa},  Kumar-Narasimhan-Ramanathan   \cite{KNR}, Laszlo-Sorger \cite{LS} and Pauly \cite{P}).

From a $\Gamma$-curve $\Sigma$ and an action of $\Gamma$ on  $G$, the $\Gamma$-invariants of Weil restriction produces a parahoric Bruhat-Tits group scheme $\mathscr{G}$ on $\bar{\Sigma} =\Sigma/\Gamma$. Recently, the geometry of the moduli stack $\mathscr{B}un_{\mathscr{G}}$ of $\mathscr{G}$-torsors over $\bar{\Sigma}$  has extensively been studied by Pappas-Rapoport \cite{PR1, PR2}, Heinloth \cite{He},  Zhu \cite{Zh} and
Balaji-Seshadri \cite{BS}. A connection between generalized theta functions on $\mathscr{B}un_{\mathscr{G}}$ and twisted conformal blocks associated to the Lie algebra of $\mathscr{G}$ was  conjectured by Pappas-Rapoport \cite{PR2}.    Along this direction,  some results have recently been obtained by Zelaci [Z] when $\Gamma$ is of order $2$ acting on $\fg=sl_{n}$ by certain involutions and very special weights.

 We  study this connection in full generality in the setting of $\Gamma$-curves $\Sigma$.
Let $G$ be the simply-connected simple algebraic group with the action of $\Gamma$ corresponding to $\phi:\Gamma\to {\rm Aut}(\fg)$. We assume that $\Sigma$ is a smooth irreducible projective curve with a collection $\vec{q}=(q_1,\cdots, q_s)$ of marked points such that their $\Gamma$-orbits are disjoint. To this, we attach a collection $\vec{\lambda}=(\lambda_1,\cdots,\lambda_s)$ of weights with $\lambda_i\in D_{c,q_i}$ as before. Assume that $c$ is divisible by $|\Gamma|$. Then, the irreducible representation $V(\lambda_i)$ of $\fg^{\Gamma_{q_i}}$ of highest weight $\lambda_i$ integrates to an algebraic representation of $G^{\Gamma_{q_i}}$ (cf.\,Proposition \ref{Prop10.8}), where $G^{\Gamma_{q_i}}$ is the fixed subgroup of $\Gamma_{q_i}$ in $G$.  Let $P_i^{q_i}$ be the stabilizer in $G^{\Gamma_{q_i}}$ of the highest weight line $\ell_{\lambda_i}\subset V(\lambda_i)$.   Let $\mathscr{G}$ be the parahoric Bruhat-Tits group scheme over $\bar{\Sigma}:=\Sigma/\Gamma$ obtained from the $\Gamma$-invariants of the Weil restriction via $\pi: \Sigma\to \bar{\Sigma}$ from the constant group scheme $G\times \Sigma\to \Sigma$ over $\Sigma$ (cf. Definition 
\ref{defi11.1}).  
One can attach the moduli stack $\mathscr{P}arbun_{\mathscr{G}}(\vec{P})$ of quasi-parabolic $\mathscr{G}$-torsors with parabolic subgroups $\vec{P}=(P_i^{q_i})$ attached to $q_i$ for each $i$ (cf. Definition \ref{def11.2}).  With the assumption that $c$ is divisible by $|\Gamma|$, we can define a line bundle $\mathfrak{L}(c;\vec{\lambda})$ on  $\mathscr{P}arbun_{\mathscr{G}}(\vec{P})$ (cf.\,Definition \ref{def11.6}).   The following is our last main theorem (cf. Theorem \ref{thm12.1}) confirming a conjecture of Pappas-Rapoport for 
$\mathscr{G}$. 
\begin{theoremE}\label{thmE}
{\it Assume that $\Gamma$ stabilizes a Borel subalgebra of $\fg$
and that $c$ is divisible by $|\Gamma|$.  Then, there exists a canonical isomorphism:
\[  H^0( \mathscr{P}arbun_{\mathscr{G}}(\vec{P}), \mathfrak{L}(c, \vec{\lambda})   ) \simeq   \mathscr{V}_{\Sigma, \Gamma, \phi}(\vec{q}, \vec{\lambda} )^{\dagger},       \]
where $H^0( \mathscr{P}arbun_{\mathscr{G}}(\vec{P}), \mathfrak{L}(c, \vec{\lambda})   )$ denotes the space of  global sections of the line bundle $\mathfrak{L}(c,\vec{\lambda})$ and $\mathscr{V}_{\Sigma, \Gamma, \phi}(\vec{q}, \vec{\lambda} )^{\dagger}$ denotes the space of twisted conformal blocks, i.e., the dual space of $\mathscr{V}_{\Sigma, \Gamma, \phi}(\vec{q}, \vec{\lambda} ) $.}
\end{theoremE}

One of the main ingredients in the proof of this theorem is the connectedness of the ind-group ${\rm Mor}_{\Gamma}(\Sigma^*, G)$ consisting of $\Gamma$-equivariant morphisms from $\Sigma^*$ to $G$ (cf.\,Theorem \ref{thm8.1.1}), where $\Sigma^*$ is a $\Gamma$-stable affine open subset of  $\Sigma$.  Another important ingredient is the Uniformization Theorem for the stack of $\mathscr{G}$-torsors on the parahoric Bruhat-Tits group scheme $\mathscr{G}$  due to Heinloth \cite{He}; in fact, its parabolic analogue (cf. Theorem \ref{Thm11.3}).  Finally, yet another ingredient is the splitting of the central extension of the twisted loop group $G(\mathbb{D}_q^\times)^{\Gamma_q}$ over 
$\Xi = {\rm Mor}_{\Gamma}(\Sigma\backslash \Gamma\cdot q, G)$ and the reducedness and the irreducibility of $\Xi$ (cf. Theorem \ref{Thm10.7} and Corollary \ref{coro11.5}), where $q$ is a point in $\Sigma$ and $\mathbb{D}_q^\times$ (resp. $\mathbb{D}_q$) is the punctured formal disc (resp. formal disc) around $q$ in $\Sigma$.  

In spite of the parallels with the classical case, there are some important essential differences in the twisted case. First of all the constant group scheme is to be replaced by the 
parahoric Bruhat-Tits group scheme $\mathscr{G}$.   Further, the group $\Xi$ could have nontrivial characters resulting in the splitting over $\Xi$ non-unique. (It might be mentioned that in the special case considered by Zelaci [Z, Proposition 5.1] mentioned above, $\Xi$ has only trivial character.)  To overcome this difficulty,  we need to introduce a {\it canonical}  splitting over $\Xi$ of the central extension of the twisted loop group $G(\mathbb{D}^\times_q )^{\Gamma_q}$ (cf.\,Theorem \ref{Thm10.7}). We are able to do it when $c$ is divisible by $|\Gamma|$ (cf. Remark \ref{remark12.2} (b)).   

It is interesting to remark that Zhu \cite{Zh} proved that for any line bundle on the moduli stack $\mathscr{B}un_{\mathscr{G}}$ for a `reasonably good' parahoric Bruhat-Tits group scheme $\mathscr{G}$ over a curve $\bar{\Sigma}$, the pull-back of the line bundle to the twisted affine Grassmannian at every point of $\bar{\Sigma}$ is of the same central charge.  It matches the way we define the space of covacua, i.e., we attach  integrable highest weight representations of twisted affine Lie algebras of the {\it same} central charge at every point.  \\

Our work was initially  motivated by a conjectural connection predicted by Fuchs-Schweigert \cite{FSc} between  the trace of diagram automorphism on the space of conformal blocks and certain conformal field theory related to twisted affine Lie algebras.  A Verlinde type formula for the trace of diagram automorphism on the space of conformal blocks  has been proved recently by the first author \cite{Ho1, Ho2}, where  the formula involves the twisted affine Kac-Moody algebras mysteriously.  

Assuming a twisted analogue of Teleman's vanishing theorem of Lie algebra homology, in a recent paper \cite{HK}, we  derive an analogue of the Kac-Walton formula and the Verlinde formula for general $\Gamma$-curves (with mild restrictions on ramification types, but not requiring that $\Gamma$ stabilizes a Borel subalgebra). In particular, if the Lie algebra $\fg$ is not of type $D_4$, there are no restrictions on ramification types. Using the machinery of crossed modular categories, under the assumption that  $\Gamma$ stabilizes a Borel subalgebra of $\fg$, Deshpande-Mukhopadhyay \cite{DM}  deduced a Verlinde type formula for the dimension of twisted conformal blocks,  which is expressed in terms of S-matrices. 
\vskip1ex

In the following we recall  the structure of this paper. 

In Section \ref{Kac_Moody_Section}, we introduce the twisted affine Lie algebra $\hat{L}(\fg,\sigma)$ attached to a finite order automorphism $\sigma$ of $\fg$ following \cite[Chap. 8]{Ka}. We  prove some preparatory lemmas which is used later in Section \ref{Propagation_section}.

In Section \ref{conformal_block_section},  we  define the space of twisted covacua attached to  a Galois cover of an algebraic curve. We  prove that  this space  is finite dimensional under the assumption given in Definition \ref{def1.2}.

Section \ref{Propagation_section}  is devoted to proving the Propagation Theorem.

Section \ref{Factorization_section} is devoted to proving the Factorization Theorem.   

In Section \ref{Sugawara_section}, we prove the independence of parameters for  integrable highest weight  representations of twisted affine Kac-Moody algebras over a base. We also prove that the Sugawara operators  acting on the integrable highest weight representations of twisted affine Kac-Moody algebras are independent of the parameters up to  scalars.  This section is preparatory for Section \ref{Projective_Connection_section}.

In Section \ref{Projective_Connection_section}, we define the sheaf of twisted covacua for  a family $\Sigma_T$ of $s$-pointed $\Gamma$-curves. We further show that this sheaf is locally free of finite rank for a smooth family $\Sigma_T$ over a smooth base $T$. In fact, it admits  a flat projective connection. 

In Section \ref{Hurwitz_Section}, we consider stable families of  $s$-pointed $\Gamma$-curves and we show that the sheaf of twisted covacua over the  stable compactificaiton of Hurwitz stack is locally free. 

In Section \ref{thetafuctions}, we prove the connectedness of the ind-group ${\rm Mor}_{\Gamma}(\Sigma^*, G)$,  following an argument by Drinfeld in the non-equivariant case. In particular, we show that the twisted Grassmannian $X^q= G(\mathbb{D}^*_q )^{\Gamma_q}/ G(\mathbb{D}_q )^{\Gamma_q} $ is irreducible.

In  Section \ref{section10}, we construct the central extensions of the twisted loop group $G(\mathbb{D}^*_q )^{\Gamma_q}$ and  prove the existence of its canonical splitting over $\Xi :={\rm Mor}_{\Gamma}(\Sigma\setminus \Gamma\cdot q, G)$.

  In Section \ref{section11}, we  introduce the moduli stack $\mathscr{P}arbun_{\mathscr{G}}$ of quasi-parabolic $\mathscr{G}$-torsors over $\bar{\Sigma}$, where   $\mathscr{G}$ is  the  parahoric Bruhat-Tits group scheme. We further    
  recall its uniformization theorem essentially due to Heinloth and construct the line bundles over 
  $\mathscr{P}arbun_{\mathscr{G}}$.
  
  In Section \ref{section12}, we establish the identification of twisted conformal blocks and generalized theta functions on the moduli stack $\mathscr{P}arbun_{\mathscr{G}}$.  
  
\vskip2ex

\noindent {\bf Acknowledgements}:   We would like to thank  Prakash Belkale,  Joseph Bernstein, Chiara Damiolini,  Zhiwei Yun and Xinwen Zhu for some helpful conversations. We  also would like to thank Matthieu Romagny for answering some questions on stable compactification of Hurwitz stacks. We thank the referee for carefully reading the manuscript and providing various suggestions for improvement. J.\,Hong is partially supported by the Simons Foundation Collaboration Grant 524406 and NSF grant DMS-2001365;  S. Kumar is partially supported by the NSF grant DMS-1501094.

\section{Twisted affine Kac-Moody algebras}
\label{Kac_Moody_Section}
This section is devoted to recalling the definition of twisted affine Kac-Moody Lie algebras and their basic properties (we need).

Let $\sigma$ be an operator of finite order $m$  acting on two vector spaces $V$ and $W$ over $\mathbb{C}$. Consider the diagonal action of $\sigma$ on $V\otimes W$. We have the following decomposition of the $\sigma$-invariant subspace in $V\otimes W$,
\[ (V\otimes W)^\sigma=\oplus_{\xi} V_{\xi}\otimes W_{\xi^{-1}} ,  \]
where the summation is over $m$-th roots of unity  and $V_{\xi}$ (resp. $W_{\xi^{-1}}$) is the $\xi$-eigenspace of $V$ (resp. $\xi^{-1}$-eigenspace of $W$).  We say $v\otimes w$ is pure or more precisely $\xi$-pure if $v\otimes w\in V_{\xi}\otimes W_{\xi^{-1}}$.  {\it Throughout this paper, if we write $v\otimes w\in (V\otimes W)^\sigma$, we mean $v\otimes w$ is pure. }

 Let $\fl$ be a Lie algebra over $\mathbb{C}$ and let $A$ be a commutative  algebra over $\mathbb{C}$. Let $\sigma$ act on $\fl$ (resp.  $A$)  as Lie algebra
(resp. algebra)  automorphism of finite orders. For any $x\otimes a\in \fl\otimes A$, we denote it by $x[a]$ for brevity. There is a Lie algebra structure on $\fl\otimes A$ with the Lie bracket given by 
\[ [x[a],y[b]]:=[x,y][ab], \quad \text{ for any elements } x[a],y[b]\in \fl\otimes A.   \]
Then, $(\fl\otimes A)^\sigma$ is a Lie subalgebra.

Let $\mathfrak{g}$ be a simple Lie algebra over $\mathbb{C}$  with a Cartan subalgebra $\mathfrak{h}$ and let $\sigma$ be an automorphism of $\mathfrak{g}$ such that $\sigma^m=1$ ($\sigma$ is not necessarily of order $m$). 
Let
$\langle\cdot ,\cdot\rangle$ be the invariant (symmetric, nondegenerate)
bilinear form on $\fg$ normalized so that the induced form on the
dual space $\fh^*$ satisfies $\langle\theta ,\theta\rangle =2$ for the
highest root $\theta$ of $\fg$.  The bilinear form $\langle\cdot,\cdot \rangle$ is $\sigma$-invariant since $\sigma$ is a Lie algebra automorphism of $\fg$.

Let $\mathcal{K} =\bc((t)):= \bc [[t]][t^{-1}]$ be the field of Laurent power series, and let $\mathcal{O}$
be the ring of formal power seires $\mathbb{C}[[t]]$ with the maximal ideal $\mathfrak{m}=t\mathcal{O}$.  We fix a $\bold{primitive}$ $m$-th root of unity $\epsilon = \epsilon_m= e^{\frac{2\pi i}{m}}$ throughout the paper.    We define an action of $\sigma$ on $\mathcal{K}$ as field  automorphism by setting 
$$\sigma (t)=\epsilon^{-1}t\,\,\,\text{and $\sigma$ acting trivially on $\bc$}.$$
It gives rise to  an action of $\sigma$ on the loop algebra $L(\mathfrak{g}):=\mathfrak{g}\otimes_{\mathbb{C}}\mathcal{K} $. Under this action,
$$L(\fg)^\sigma = \oplus_{j=0}^{m-1}\,\left(\fg_j\otimes \mathcal{K}_j\right),$$
where
\begin{equation}  \label{eq1.1.1.3}\fg_j:=\{x\in \fg: \sigma (x)=\epsilon^j x\},\,\,\,\text{and}\,\,\mathcal{K}_j=\{P\in \mathcal{K}: \sigma (P)=\epsilon^{-j} P\}.
\end{equation}

We now define a central extension  $\hat{L}(\fg,\sigma):=L(\fg)^\sigma\oplus \mathbb{C}C$ of $L(\fg)^\sigma$ under the bracket
  \begin{equation}  \label{eq1.1.1.4}
[x[P]+z C, x'[P'] +z' C] =
[x,x' ][PP'] +m^{-1}\Res_{t=0} \,\bigl(({dP})
P'\bigr) \langle x,x'\rangle C,  
 \end{equation}
for  $x[P],x'[P']\in L(\fg)^\sigma$, $z, z'\in\bc$; where
$\Res_{t=0}$ denotes the coefficient of $t^{-1}dt$.  Let  $\hat{L}(\fg,\sigma)^{\geq 0}$ denote the subalgebra 
\[  \hat{L}(\fg,\sigma)^{\geq 0}:= \oplus_{j= 0}^{m-1}   \fg_j\otimes  \mathcal{O}_j \oplus  \mathbb{C} C, \]
where $\mathcal{O}_j=\mathcal{K}_j\cap \mathcal{O}$.
We also denote 
\[  \hat{L}(\fg,\sigma)^+:=  \oplus_{j  = 0}^{m-1}   \fg_j\otimes  \mathfrak{m}_{j}  , \quad \text{and }    \hat{L}(\fg,\sigma)^-:= \oplus_{j< 0}   \fg_j\otimes  t^{j}  ,\]
where $\mathfrak{m}_j=\mathfrak{m}\cap \mathcal{O}_j$.
Then,   $\hat{L}(\fg,\sigma)^+$ is an ideal of  $\hat{L}(\fg,\sigma)^{\geq 0}$ and the quotient  $\hat{L}(\fg,\sigma)^{\geq 0}/  \hat{L}(\fg,\sigma)^+$ is isomorphic to $ \fg_0\oplus \mathbb{C}C$.   Note that $\fg_0$ is the Lie algebra $\fg^\sigma$ of $\sigma$-fixed points in $\fg$.    As  vector spaces we have 
\[  \hat{L}(\fg, \sigma) = \hat{L}(\fg,\sigma)^{\geq 0} \oplus    \hat{L}(\fg,\sigma)^- . \]

By the classification theorem of finite order automorphisms of simple Lie algebras (cf. \cite[ Proposition 8.1, Theorems 8.5, 8.6]{Ka}),
there exists a  `compatible' Cartan subalgebra $\fh$ and a `compatible' Borel subalgebra $\fb \supset \fh$ of $\fg$  both stable under the action of $\sigma$ such that 
\begin{equation}  \label{eq1.1.1.0} \sigma=\tau \epsilon^{{\rm ad} h}, 
\end{equation}
where $\tau$ is a diagram automorphism of $\fg$ of order $r$ preserving $\fh$ and $\fb$, and $\epsilon^{{\rm ad} h}$ is the inner automorphism of $\fg$ such that for any root $\alpha$ of $\fg$, $\epsilon^{{\rm ad} h}$ acts on the root space $\fg_\alpha$ by the multiplication $\epsilon^{\alpha(h)}$, and $\epsilon^{{\rm ad} h}$ acts on $\fh$ by the identity. {\it We consider $\tau = \Id$ also as a diagram automorphism.} Here  $h$ is an element in $\fh^\tau$. In particular,  $\tau$ and $\epsilon^{{\rm ad}h}$ commute.  Moreover,  $\alpha(h)  \in \mathbb{Z}^{\geq 0}$  for  any simple root  $\alpha$  of  $\fg^\tau$, $\beta(h)  \in \mathbb{Z}$  for  any simple root  $\beta$  of  $\fg$ and $\theta_0(h)\leq \frac{m}{r}$ where $\theta_0\in (\fh^\tau)^*$ denotes the following weight of $\fg^\tau$
\[ \theta_0=\begin{cases}  
\text{ highest root of } \fg,   \text{ if } r=1\\
\text{ highest short root  of  } \fg^\tau, \text{ if } r>1 \text{ and }(\fg, r)\neq (A_{2n},2)\\
  2\cdot \text{highest short root} \text{ of } \fg^\tau, \,\text{ if } (\fg,r)=(A_{2n},2).  \end{cases}  \]

 Observe that $r$ divides $m$, and $r$ can only be $1,2, 3$.   Note that $\fg^\sigma$ and $\fg^{\tau}$ share the common Cartan subalgebra $\fh^{\sigma}=\fh^{\tau}$.

Let $I(\fg^\tau)$ denote the set of vertices of the Dynkin diagram of $\fg^\tau$.  Let $\alpha_i$ denote the simple root associated to $i\in I(\fg^\tau)$.  Let  $\hat{I}(\fg,\sigma)$ denote the set $I(\fg^\tau)\sqcup \{o\}$, where $o$ is just a symbol. (Observe that $\tau$ is determined from $\sigma$.)
  Set 
\[   s_i=\begin{cases} \alpha_i(h)  \quad \text{if } i\in I(\fg^\tau) \\
\frac{m}{r}-\theta_0(h)  \quad \text{if } i=o .
    \end{cases}\]

Then,  $s=\{s_i \,|\, i\in  \hat{I}(\fg,\sigma) \}$ is a tuple of non-negative integers. 
Let  $\hat{L}(\fg, \tau)$ denote the Lie algebra with the construction similar to  $\hat{L}(\fg, \sigma)$ where $\sigma$ is replaced by $\tau$, $m$ is replaced by $r$ and $\epsilon$ is replaced by $\epsilon^{\frac{m}{r}}$. There exists an isomorphism of Lie algebras (cf. \cite[Theorem 8.5]{Ka}):
  \begin{equation} \label{neweqn2.3.1} \phi_\sigma:   \hat{ L}(\fg, \tau)\simeq   \hat{L}(\fg, \sigma)  
\end{equation}
given by $C\mapsto C$ and
 $x[t^j]\mapsto    x[t^{\frac{m}{r}j+k  } ]$, for any $x$ an  $\epsilon^{\frac{m}{r}j}$-eigenvector of $\tau$, and $x$ also a $k$-eigenvector of ${\rm ad } \,h$.  We remark that in the case $ (\fg,r)=(A_{2n},2)$, our labelling for $i=o$ is the same as $i=n$ in [Ka, Chapter 8].  It is well-known that  $\hat{L}(\fg,\tau)$ is an affine Lie algebra, more precisely  $\hat{L}(\fg, \tau)$ is untwisted if $r=1$ and twisted if $r>1$.

By Theorem 8.7 in \cite{Ka}, there exists a $sl_2$-triple $x_i,y_i,h_i \in \fg$ for each $i\in  \hat{I}(\fg,\sigma)$ where 
\begin{itemize}
\item $x_i\in (\fg^\tau)_{\alpha_i}$, $y_i\in (\fg^\tau)_{-\alpha_i}$ when $i\in I(\fg^\tau)$;
\item  $x_{o} $ (resp. $y_{o}$) is a $(-\theta_0)$(resp. $\theta_0$)-weight vector with respect to the adjoint action of  $\fh^\tau$ on $\fg$, and is also  an $\epsilon^{\frac{m}{r}}$ (resp. $\epsilon^{-\frac{m}{r}}$)-eigenvector of $\tau$ ;
\item $x_i\in \mathfrak{n}$ for $i\in I(\fg^\tau)$ and $x_o \in  \mathfrak{n}^-$, where  $\mathfrak{n}$ (resp.  $\mathfrak{n}^-$) is the nil-radical of $ \mathfrak{b}$ (resp. the opposite Borel subalgebra $ \mathfrak{b}^-$). Similarly,  $y_i\in \mathfrak{n^-}$ for $i\in I(\fg^\tau)$ and $y_o \in  \mathfrak{n}$,
\end{itemize}
 (see explicit construction of $x_i, y_i, i\in \hat{I}(\fg,\sigma)$ in \cite[\S 7.4, \S 8.3]{Ka}), such that 
  \[  x_i[t^{s_i}], y_i[t^{-s_i}],  [x_i[t^{s_i} ],  y_i[t^{-s_i}]], \, i\in \hat{I}(\fg,\sigma),\]
   are Chevalley generators of  $\hat{L}(\fg,\sigma)$ and $\{x_i, y_i, [x_i, y_i]\}_{i\in  \hat{I}(\fg,\sigma): s_i=0}$ are Chevalley generators of $[\mathfrak{g}^\sigma, \mathfrak{g}^\sigma]$.  We set 
\[ \tilde{x}_i:=x_i[t^{s_i}],  \tilde{y}_i:=y_i[t^{-s_i}], \text{ and } \tilde{h}_i:=[\tilde{x}_i, \tilde{y}_i],\quad \text{for any } i\in \hat{I}(\fg,\sigma) . \]
Via the isomorphism $\phi_\sigma$, we have 
\[ \phi_\sigma(x_i)= \tilde{x_i}, \phi_\sigma(y_i)=\tilde{y}_i, \text{ for any } i\in I(\fg^\tau),  \]
and 
\[\phi_\sigma(x_o[t])=\tilde{x}_o,    \quad  \phi_\sigma(y_o[t^{-1}])=\tilde{y}_o . \]
Thus,  $\deg \tilde{x}_i=s_i$ and $\deg \tilde{y}_i=-s_i$. The Lie algebra $\hat{L}(\fg,\sigma)$ is called an $(s,r)$-realization of the associated affine Lie algebra $\hat{L}(\fg,\tau)$. 

From the above discussion, for any $i\in \hat{I}(\fg, \sigma)$,  we have
\begin{equation}
\label{eigenvalue_sigma}
\sigma(x_i)=\epsilon^{s_i} x_i, \,\text{ and } \sigma(y_i)=\epsilon^{-s_i}y_i  \,.
\end{equation}


We fix a positive integer $c$ called the {\em level} or {\em central charge}. Let $\Rep_c$ be the set of isomorphism classes of integrable highest weight (in particular, irreducible)  $\hat{L}(\fg,\sigma)$-modules with central charge $c$, where in our realization $C$ acts by $c$,  the standard Borel subalgebra  of $\hat{L}(\fg,\sigma)$ is generated by $\{\tilde{x}_i, \tilde{h}_i\}_{i\in  \hat{I}(\fg, \sigma)}$ and $\hat{L}(\fg,\sigma)^-$ is generated by 
\[  \{\tilde{y}_i\}_{\{i\in  \hat{I}(\fg, \sigma): s_i> 0\}},  (\text{cf. [Ka, Theorem 8.7] }).\]
  Thus,   $\hat{L}(\fg,\sigma)^{\geq 0}$ is a standard parabolic subalgebra of  $\hat{L}(\fg,\sigma)$. For any $\mathscr{H}\in \Rep_c$, let $\mathscr{H}^0$ be the subspace of $\mathscr{H}$ annihilated by  $\hat{L}(\fg,\sigma)^+$. Then, $\mathscr{H}^0$ is an irreducible finite dimensional  $\fg^\sigma$-submodule of $\mathscr{H}$ with highest weight (say) $\lambda(\mathscr{H})\in (\fh^\sigma)^*=(\fh^\tau)^*$ for the choice of the Borel subalgebra of $\fg^\sigma$  generated by $\fh^\sigma$ and $\{x_i: s_i=0\}$.  The correspondence $\mathscr{H}\mapsto \lambda(\mathscr{H})$ sets up an injective map $\Rep_c \to (\fh^\tau)^*$. Let $D_c$ be its image.
For $\lambda \in D_c$, let $\mathscr{H}(\lambda)$ be the corresponding integrable highest weight  $\hat{L}(\fg,\sigma)$-module with central charge $c$. 

  For any $\lambda\in D_c$ and $i\in \hat{I}(\fg,\sigma)=I(\fg^\tau)\sqcup \{o\}$, we associate an integer $n_{\lambda, i}$ as follows.  
Set 
\begin{equation}
\label{integer_n_mu} 
 n_{\lambda, i} =  
 \lambda ([x_i, y_i]) + \langle x_i,y_i \rangle  \frac{s_i c }{ m}.
   \end{equation}
   


For $\sigma = \tau$ a diagram automorphism of $\mathfrak{g}$  (including $\tau =\Id$), by definition $s_i= 0$ for 
$i\in I(\fg^\tau)$ and $s_o = 1$. 

For any diagram automorphism $\tau$ of order $r$ (including $r = 1$), we follow the concrete realization of $x_i, y_i, i\in {I}(\fg^\tau)\sqcup\{o\}  $ in \cite[\S 8.3]{Ka}. We emphasize that in the case $(\fg,r)=(A_{2n, 2})$, our labelling ``$o$" corresponds to $i=n$ in \cite[\S 8.3]{Ka}.  When $(\fg, r)\not= (A_{2n}, 2)$, we have 
\begin{equation}   \langle x_i, y_i\rangle =\begin{cases}
1  \quad \text{if $\alpha_i$ is a long root for $i\in I(\fg^\tau)$  } \\
 r   \quad \text{if  $i=o$,  or $\alpha_i$ is a short root for $i\in I(\fg^\tau)$}
    \end{cases},\end{equation}
and when $(\fg, r)= (A_{2n}, 2)$,
\begin{equation}  \langle x_i, y_i\rangle =\begin{cases}  1  \quad \text{if  $i=o$}   \\
2   \quad \text{if $\alpha_i$ is a long root for $i\in I(\fg^\tau)$ }\\
4  \quad \text{ if $\alpha_i$ is a short root for $i\in I(\fg^\tau)$}
    \end{cases}  .\end{equation}

\begin{lemma}
\label{finite_set_weight_lem}
The set $D_c$ can be described as follows:
\[ D_c=\{  \lambda \in (\fh^\tau)^*   \,| \,   n_{\lambda,i}\in \mathbb{Z}_{\geq 0}  \quad \text{for any } i\in \hat{I}(\fg,\sigma)   \}  . \]
\end{lemma}
\begin{proof}
 The lemma follows from the fact that the irreducible highest weight   $\hat{L}(\fg,\sigma)$-module $\mathscr{H}(\lambda)$ with highest weight $\lambda$  is integrable if and only if  the eigenvalues of $\tilde{h}_i,i\in \hat{I}(\fg, \sigma)$ on the highest weight vector in $\mathscr{H}(\lambda)$ are non-negative integers. 
\end{proof}

Define
$$\bar{s}_i= \langle x_i, y_i \rangle s_i, \,\,\,\text{for any $i\in \hat{I}(\fg, \sigma)$ }$$
and let
$$\bar{s} := \text{gcd}\,\{ \bar{s}_i: i\in \hat{I}(\fg,\sigma)\}.$$
As an immediate consequence of Lemma \ref{finite_set_weight_lem} and the identity \eqref{integer_n_mu}, we get the following:
\begin{coro} \label{newcoroweight0} For any integer $c\geq 1$, $0\in D_c$ if and only if $m$ divides $\bar{s}c$. 

In particular,  $0\in D_c$ if $m$ divides $c$.

Also, for a diagram automorphism $\sigma = \tau$, $0\in D_c$  for all $c$ if $ (\fg, r) \neq (A_{2n}, 2)$. If  $(\fg, r) = (A_{2n}, 2)$,  $0\in D_c$ if and only if $c$ is even.
\end{coro}
We recall the following well known result.
\begin{lemma} For any automorphism $\sigma$ and any $c \geq 1$, $D_c \neq \emptyset$.
\end{lemma} 
\begin{proof} By the isomorphism $\phi_\sigma$ (as in equation \eqref{neweqn2.3.1}), it suffices to prove the lemma for the diagram automorphisms $\tau$ (including $\tau = \Id$). By Corollary \ref{newcoroweight0},  $0\in D_c$   if $ (\fg, r) \neq (A_{2n}, 2)$. If  $ (\fg, r) = (A_{2n}, 2)$, take $\lambda = \omega_n$: the $n$-th fundamental weight of type $B_n$ (following the Bourbaki convention of indexing as in [Bo, Planche II]). Then, $\omega_n \in D_c$ for odd values of $c$ and 
$0\in D_c$ for even values of $c$.
\end{proof}

 Let $V(\lambda)$ be the irreducible $\fg^\sigma$-module with highest weight $\lambda$ and highest weight vector $v_+$.
Let  $\hat{M}(V(\lambda),c)$ be the generalized Verma module $U( \hat{L}(\fg,\sigma)) \otimes_{U(  \hat{L}(\fg,\sigma)^{\geq 0} )  } V(\lambda)$ with highest weight vector $v_+
=1\otimes v_+$, where the action of  $\hat{L}(\fg,\sigma)^{\geq 0} $ on $V(\lambda)$ factors through the projection map  $\hat{L}(\fg,\sigma)^{\geq 0}\to  \fg^\sigma\oplus \mathbb{C} C$ and the center $C$ acts by $c$.   If $\lambda \in D_c$, then  the unique irreducible quotient of  $\hat{M}(V(\lambda),c)$  is the integrable representation $\mathscr{H}(\lambda)$.  
 Let $K_\lambda$ be  the kernel of $ \hat{M}(V(\lambda),c)\to \mathscr{H}(\lambda)$.  
Set
\[ \hat{I}(\fg,\sigma)^+:=\{i\in \hat{I}(\fg,\sigma)\,|\,    s_i>0  \} . \]
Then, as $U( \hat{L}(\fg,\sigma) )$-module, $K_\lambda$ is generated by
\begin{equation} \label{eqn1.2.0} \{ \tilde{y}_i^{n_{\lambda,i} +1}  \cdot v_+ \,| \, i\in \hat{I}(\fg,\sigma)^+\}\,\,\,\text{ (cf. [Ku, $\S$ 2.1])}.
\end{equation}
Moreover, these elements are highest weight vectors. 

\begin{lemma}\label{lem_kac}
Fix $i\in \hat{I}(\fg,\sigma)$. Let $f\in \mathcal{K}$ be such that $\sigma(f)=\epsilon^{- s_i}f$ and $f\equiv t^{s_i}\mod{t^{{s_i}+1}}$. Put $X=x_i[f]$ and $Y=\tilde{y}_i=y_i[t^{-s_i}]$. For any  $p> n_{\lambda,i}$ and $q> 0$, there exists $\alpha \neq 0 $ such that 
\begin{equation*}
Y^{p}\cdot v_{+}=\alpha X^{q}  Y^{p+q}\cdot v_{+}
\end{equation*}
in the generalized Verma module  $\hat{M}(V(\lambda),c)$.
\end{lemma}
\begin{proof}
Let $H:=[X,Y]=h_i[t^{-s_i}f]+\frac{\bar{s_i }}{m} C$. Then, $[H, Y]=-2y_i[t^{-2s_i}f]$ commutes with $Y$. Note that $H\cdot v_+= n_{\lambda, i}v_+$. Then, one can check that, for  $d\geq 0$,  
\[ HY ^d\cdot v_+= (n_{\lambda,i}-2d)Y^d\cdot v_+ , \]
and,  for  $p\geq 0$,  
\[ XY^{p+1}\cdot v_+=  (p+1)(n_{\lambda, i} -p ) Y^p\cdot v_+. \]
By induction on $q$, the lemma follows.
\end{proof}

\begin{lemma} \label{lemma 1.3} Let $\fg$ and $\sigma$ be as above and let $\mathfrak{b}$ be a $\sigma$-stable  Borel subalgebra with $\sigma$-stable  Cartan subalgebra $ \mathfrak{h}\subset \mathfrak{b}$ of $\fg$.  Then, any element $x$  of 
 $\left(\mathfrak{n} \otimes \mathcal{K}\right)^\sigma$ acts locally nilpotently  on any integrable highest weight  $\hat{L}(\fg, \sigma)$-module  $\mathscr{H}(\lambda)$, where $\mathfrak{n}$ is the nil-radical of  $\mathfrak{b}$.

Replacing the Borel subalgebra  $\mathfrak{b}$ by the opposite Borel subalgebra  $\mathfrak{b}^-$, the lemma holds for any  $x\in \left(\mathfrak{n}^- \otimes \mathcal{K}\right)^\sigma$ as well,   where $\mathfrak{n}^-$ is the nil-radical of  $\mathfrak{b}^-$.
\end{lemma}
\begin{proof} Take a basis $\{y_\beta \}_\beta$ of  $\mathfrak{n}$ consisting of common eigenvectors under the action of $\sigma$ as well as 
$\mathfrak{h}^\sigma$ (which is possible since the adjoint action of  $\mathfrak{h}^\sigma$ on $\fg$ commutes with the action of $\sigma$)
and write $x=\sum_\beta \, y_\beta[P_\beta]$ for some $P_\beta \in \mathcal{K}$. Since $x$ is $\sigma$-invariant and each $y_\beta$ is an eigenvector, each $y_\beta[P_\beta]$ is  $\sigma$-invariant. Let   $\hat{L}(\fg, \sigma)_x$ be the Lie subalgebra of     $\hat{L}(\fg, \sigma)$ generated by the elements $\{ y_\beta[P_\beta]\}_\beta$. Then, since  $\mathfrak{n}$ is nilpotent (in particular, $N$-bracket of elements from  $\mathfrak{n}$ is zero, for some large enough $N$),    $\hat{L}(\fg, \sigma)_x$ is a finite dimensional nilpotent Lie algebra. (Observe that  $\mathfrak{n}$ being nilpotent, for any two elements $s_1, s_2 \in \mathfrak{n}, \langle s_1, s_2\rangle =0$.) 
Take any element  $v\in  \mathscr{H}(\lambda )$ and let $\mathscr{H}(x, v)$ be the    $\hat{L}(\fg, \sigma)_x$ -submodule of $\mathscr{H}(\lambda)$ generated by $v$. Since $\sigma$ stabilizes the pair $(\fb, \fh)$, the centralizer $Z_{\mathfrak{g}}(\mathfrak{h}^\sigma)$ 
 of $\mathfrak{h}^\sigma$ in $\mathfrak{g}$ equals $ \mathfrak{h}$. To prove this, we can write $\sigma= \tau\circ \Ad(t)$, for a diagram automorphism  (possibly identity) $\tau$ of $\fg$ associated to the pair $(\fb, \fh)$ and $t\in T$ with Lie algebra  $\Lie T =\fh$. Thus, $\fh^\sigma=\fh^\tau$ and hence $Z_{\mathfrak{g}}(\mathfrak{h}^\sigma) = Z_{\mathfrak{g}}(\mathfrak{h}^\tau) = \mathfrak{h}$.  
Since $y_\beta$ is an eigenvector for the adjoint action of $\mathfrak{h}^\sigma$ with non-trivial action,
 any $y_\beta[P_\beta]$ can be written as a finite sum of commuting real root vectors for   $\hat{L}(\fg, \sigma)$ and a $\sigma$-invariant element of the form 
$y_\beta[P^+_\beta]$ with $P^+_\beta \in t\mathbb{C}[[t]]$ (cf. [Ka, Exercises 8.1 and 8.2, $\S$8.8]). Thus, $y_\beta[P_\beta]$ acts locally nilpotently on $\mathscr{H}(\lambda)$ (in particular, on  $\mathscr{H}(x, v)$).  Now, using [Ku, Lemma 1.3.3 (c$_2$)], we get that   $\mathscr{H}(x,v)$  is  finite dimensional. Using Lie's theorem, the lemma follows.
\end{proof}

\section{Twisted Analogue of Conformal Blocks}\label{conformal_block_section}

{\it By a scheme we mean a quasi-compact and separated scheme over $\mathbb{C}$. By an algebraic curve, we mean a projective, reduced but not necessarily connected curve.} 

In this section we  define the space of twisted covacua attached to  a Galois cover of an algebraic curve. We  prove that  this space  is finite dimensional.

For a smooth point $p$ in an algebraic curve $\bar{\Sigma}$ over $\mathbb{C}$, let $\mathcal{K}_p$ denote the quotient field of the completed local ring $\hat{\mathscr{O}}_p$ of $\bar{\Sigma}$ at $p$.  We denote by $\mathbb{D}_p$ (resp. $\mathbb{D}^\times_p$) the formal disc ${\rm Spec} \, \hat{\mathscr{O}}_p$ (resp. the punctured formal disc ${\rm Spec}\, \mathcal{K}_p$).  

\begin{definition}
\label{Gamma_curve}
A morphism $\pi: {\Sigma}\to \bar{\Sigma}$ of projective curves is said to be a {\it Galois cover with finite Galois group} $\Gamma$ (for short $\Gamma$-{\it cover}) if   the group $\Gamma$ acts on ${\Sigma}$ as algebraic automorphisms  and ${\Sigma}/ \Gamma\simeq \bar{\Sigma}$ and no nontrivial element of $\Gamma$ fixes pointwise any irreducible component of $\Sigma$.
\end{definition} 

 For any smooth point $q\in {\Sigma}$, the stabilizer group $\Gamma_q$ of $\Gamma$ at $q$ is always cyclic.  The order $e_q:=|\Gamma_q|$ is called  the {\it ramification index of } $q$.  Thus, $q$ is unramified if and only if $e_q=1$.  Denote $p=\pi(q)$.  We can also write $e_p=e_q$, since $e_{q}=e_{q'}$ for any $q,q'\in \pi^{-1}(p)$. We  also say that $e_p$ is the ramification index of $p$.  Denote by $d_p$ the cardinality of the fiber $\pi^{-1}(p)$. Then $|\Gamma|=e_p\cdot  d_p$.

The action of $\Gamma_q$ on  the tangent space $T_q{\Sigma}$ induces a primitive character $\chi_q: \Gamma_q\to \mathbb{C}^\times$, i.e.,  $\chi_q(\sigma_q)$ is a primitive $e_p$-th root of unity for any generator $\sigma_q$ in $\Gamma_q$.  From now on we shall fix $\sigma_q\in \Gamma_q$ so that 
\[\chi_q(\sigma_q) =e^{2\pi i/e_p}.\]
 For any two smooth points $q, q'\in {\Sigma}$, if $\pi(q)=\pi(q')$  then
 \[ \Gamma_{q'}=\gamma\Gamma_q\gamma^{-1}, \text{ for any element }  \gamma\in \Gamma \text{ such that } q'=\gamma\cdot q. \] 
Moreover,
\[  \chi_{q'}(\gamma\sigma \gamma^{-1}) =\chi_q(\sigma), \, \text{ for any } \sigma\in \Gamma_q.  \]

Given a smooth point $p\in \bar{\Sigma}$ such that $\pi^{-1}(p)$ consists of smooth points in ${\Sigma}$,  let $\pi^{-1}(\mathbb{D}_p)$ (resp. $\pi^{-1}(\mathbb{D}^\times_p)$) denote the 
fiber product of ${\Sigma}$ and $\mathbb{D}_p$ (resp.  $\mathbb{D}_p^\times$) over $\bar{\Sigma}$. Then,
\[  \pi^{-1}(\mathbb{D}_p)\simeq\sqcup_{q\in \pi^{-1}(p) } {\mathbb{D}}_q,  \quad \text{and }   \quad  \pi^{-1}(\mathbb{D}^\times_p)\simeq \sqcup_{q\in \pi^{-1}(p)  } {\mathbb{D}}^\times_q   ,  \]     
where ${\mathbb{D}}_q$ (resp. ${\mathbb{D}}_q^\times$) denotes the formal disc (resp. formal punctured  disc) in ${\Sigma}$ around $q$.  
 \vskip1ex
  {\it Let the finite group $\Gamma$ also act on $\mathfrak{g}$ as Lie algebra automorphisms.}
\vskip1ex
  Let $\mathfrak{g}[\pi^{-1}(\mathbb{D}^\times_p) ]^\Gamma$ be the Lie algebra consisting of $\Gamma$-equivariant regular maps from $\pi^{-1}(\mathbb{D}^\times_p )$ to $\mathfrak{g}$. There is a natural isomorphism $\mathfrak{g}[\pi^{-1}(\mathbb{D}^\times_p) ]^\Gamma\simeq ( \fg\otimes \mathbb{C}[\pi^{-1}(\mathbb{D}^\times_p)])^\Gamma$.
  Let 
\begin{equation} \label{2.1.1}\hat{ \mathfrak{g}}_p :=\mathfrak{g}[\pi^{-1}(\mathbb{D}^\times_p) ]^\Gamma\oplus \mathbb{C}C
\end{equation}
 be the central extension of $\mathfrak{g}[\pi^{-1}(\mathbb{D}^\times_p) ]^\Gamma$ defined as follows:
  \begin{equation}  [X, Y] =[X,Y]_0+  \frac{1}{|\Gamma|}   \sum_{q\in \pi^{-1}(p)  } {\rm Res}_{q} \langle d X, Y \rangle    C,  
 \label{pointwise-affine-Lie}
  \end{equation}
 for any $X,Y \in \mathfrak{g}[\pi^{-1}(\mathbb{D}^\times_p) ]^\Gamma$, where $[,]_0$ denotes the point-wise Lie bracket induced from the bracket on $\fg$.  We set the subalgebra 
\begin{equation} \label{2.1.2}\hat{ \mathfrak{p}}_p :=\mathfrak{g}[\pi^{-1}(\mathbb{D}_p) ]^\Gamma\oplus \mathbb{C}C
\end{equation}
  and 
\begin{equation} \label{2.1.3}\hat{ \mathfrak{g}}_p^+ :=\Ker ( \mathfrak{g}[\pi^{-1}(\mathbb{D}_p) ]^\Gamma \to \mathfrak{g}[\pi^{-1}(p)]^\Gamma )
\end{equation} 
  obtained by the restriction map $\mathbb{C}[\pi^{-1}(\mathbb{D}_p)] \to \mathbb{C}[\pi^{-1}(p)]$, where $\fg[\pi^{-1}(p) ]^\Gamma$ denotes the Lie algebra consisting of $\Gamma$-equivariant maps  $x:\pi^{-1}(p)\to \fg$.    Let $\fg_p$ denote  $\fg[ \pi^{-1}(p) ]^\Gamma$.  
  The following lemma is obvious.
  \begin{lemma}
  \label{evaluation_lem}
  The evaluation map ${\rm ev}_q: \fg_p\to  \fg^{\Gamma_q}$ given by 
  \[ x\mapsto x(q) \]
   for any $x\in \fg_p$ and  $q\in \pi^{-1}(p)$ is an isomorphism of Lie algebras. 
  \end{lemma}

Let $\sigma_q$ be the generator of $\Gamma_q$ such that $\chi_q(\sigma_q) = e^{\frac{2\pi i}{e_p}}$. Let $\hat{L}(\fg,\sigma_q)$ denote the affine Lie algebra  associated to $\fg$ and  $\sigma_q$ as defined in Section \ref{Kac_Moody_Section}. We  denote this algebra by $\hat{L}(\fg, \Gamma_q,\chi_q)$  or $\hat{L}(\fg, \Gamma_q )$ in short.

\begin{lemma}\label{lemma2.3}  
The restriction map ${\rm res}_q:  \hat{\fg}_p\to  \hat{L}(\fg, \Gamma_q)$ given by 
\[  X\mapsto  X_q,  \,\text{ and } C\mapsto C,  \]
  is an isomorphism of Lie algebras, where $X\in \fg[\pi^{-1}(\mathbb{D}^\times_p ) ]^\Gamma$ and $X_q$ is the restriction of $X$ to $\mathbb{D}_q^\times$.  Moreover, 
\[   {\rm res}_q(\hat{\mathfrak{p}}_p )= \hat{L}(\fg,\Gamma_q)^{\geq 0}, \,\text{ and }\, {\rm res}_q(\hat{\fg}^+_p)= \hat{L}(\fg,\Gamma_q)^+.  \]

\end{lemma}
\begin{proof} 
For any $X, Y\in \fg[\pi^{-1}( \mathbb{D}^\times_p ) ]^\Gamma$,  the restriction of $[X,Y]_0$ to $\mathbb{D}_q^\times$ is equal to $[X_q,Y_q]_0$. 
Note that for any $\gamma\in \Gamma$ and $x,y\in \fg$, we have $\langle \gamma(x), \gamma(y) \rangle= \langle x,y  \rangle$, which  follows from the Killing form realization of $\langle ,\rangle$ on $\fg$.
Since $X, Y$ are $\Gamma$-equivariant, for any $q, q'\in \pi^{-1}(p)$ we have
\[ {\rm Res}_q\langle dX, Y\rangle=   {\rm Res}_{q'}\langle dX, Y\rangle. \]
 It is now easy to see that ${\rm res}_q:  \hat{\fg}_p\to  \hat{L}(\fg, \Gamma_q)$ is an isomorphism of Lie algebras, and 
  \[   {\rm res}_q(\hat{\mathfrak{p}}_p )= \hat{L}(\fg,\Gamma_q)^{\geq 0}, \,\text{ and }\, {\rm res}_q(\hat{\fg}^+_p)= \hat{L}(\fg,\Gamma_q)^+.  \]


\end{proof}

By the above lemma, we have a faithful functor ${\rm Rep}_c(\hat{\fg}_p)\to {\rm Rep}(\fg_p)$ from the category of integrable highest weight representations of $\hat{\fg}_p$ of level $c$ to the category of finite dimensional  representations of $\fg_p$.
 We denote by $D_{c,p}$ the parameter set of (irreducible) integrable highest weight representations of $\hat{\mathfrak{g}}_p$ of level $c$ obtained as the   subset of the set of dominant integral weights of $\mathfrak{g}_p$ under the above faithful functor.   Let $D_{c,q}$ denote the parameter set of (irreducible) integrable highest weight representations of $\hat{L}(\fg, \Gamma_q)$ as in Section \ref{Kac_Moody_Section}. Then, we can identify $D_{c,p}$ and $D_{c,q}$ via the restriction isomorphism ${\rm res}_q:  \hat{\fg}_p\to  \hat{L}(\fg, \Gamma_q)$ as  in Lemma \ref{lemma2.3}.

 \begin{definition}\label{defi2.1.1}
For any $s\geq 1$, by an {\it  $s$-pointed curve}, we mean the pair $(\bar{\Sigma}, \vec{p}=(p_1, \dots, p_s))$ consisting of distinct and smooth points $\{p_1, \dots, p_s\}$ of $\bar{\Sigma}$, such that the following condition is satisfied.

\vskip1ex

(*) Each irreducible component of $\bar{\Sigma}$ contains at least one point $p_i$. 
\vskip1ex

Similarly, by an  {\it  $s$-pointed $\Gamma$-curve}, we mean the pair $(\Sigma, \vec{q}=(q_1, \dots, q_s))$ consisting of  smooth points $\{q_1, \dots, q_s\}$ of $\Sigma$ such that  $(\bar{\Sigma}, \pi(\vec{q})=(\pi(q_1), \dots, \pi(q_s)))$ is a $s$-pointed curve.

\end{definition}
\vskip1ex

From now on we fix an $s$-pointed curve  $(\bar{\Sigma}, \vec{p})$  (for any $s\geq 1$), where $\vec{p}=(p_1, \dots, p_s)$,  and a Galois cover $\pi: {{\Sigma}} \to \bar{\Sigma}$  with the finite  Galois group $\Gamma$ such that the fiber $\pi^{-1}(p_i)$ consists of smooth points for any $i=1,2,\cdots, s$.  We also fix a simple Lie algebra $\mathfrak{g}$ and a group homomorphism $\phi: \Gamma\to  {\rm Aut}(\fg)$, where $ {\rm Aut}(\fg)$ is the group of automorphisms of $\fg$.

We now fix an $s$-tuple $\vec{\lambda}=(\lambda_{1},\ldots,\lambda_{s})$
of weights with  $\lambda_{i}\in D_{c, p_i}$ `attached' to the point $p_i$.  To this data, there is associated the
{\em space of  (twisted) vacua} 
$\mathscr{V}_{{{\Sigma}},\Gamma, \phi}(\vec{p},\vec{\lambda})^{\dagger}$ (or the space of {\em   (twisted)  conformal
blocks})
and its dual space
$\mathscr{V}_{{{\Sigma}},\Gamma,\phi }(\vec{p},\vec{\lambda})$
called the {\em space of  (twisted) covacua} (or the space of {\em   (twisted) dual conformal
blocks}) defined as follows:

\begin{definition}
\label{def1.2}
  Let $\mathfrak{g}[{{\Sigma}}\backslash
  \pi^{-1}( \vec{p}) ]^\Gamma$ denote the space of $\Gamma$-equivariant regular maps
$f:{{\Sigma}}\backslash \pi^{-1}(\vec{p}) \to \mathfrak{g}$.
 Then,
$\mathfrak{g}[{{\Sigma}}\backslash \pi^{-1}(\vec{p}) ]^\Gamma$ is a Lie algebra
under the pointwise bracket.  

Set
\begin{equation}
\mathscr{H}(\vec{\lambda}):=\mathscr{H}(\lambda_{1})\otimes\cdots \otimes\mathscr{H}(\lambda_{s}),   \label{eq4}
\end{equation}
where $\mathscr{H}(\lambda_i)$ is the integrable highest weight  representation of $\hat{\mathfrak{g}}_{p_i}$ of level $c$ with highest weight $\lambda_i\in D_{c, p_i}$. 

Define an action of the Lie algebra $\mathfrak{g}[{{\Sigma}}\backslash
 \pi^{-1} ( \vec{p} ) ]^\Gamma$ on $\mathscr{H}(\vec{\lambda})$ as
follows:
\begin{equation}
\label{componentwise_action}
X\cdot (v_{1}\otimes\cdots\otimes v_{s}) =
  \sum^{s}_{i=1}v_{1}\otimes\cdots\otimes X_{p_i}\cdot
  v_{i}\otimes \cdots\otimes v_{s}, \,\,\text{for} \,X\in \fg [{\Sigma}\backslash \pi^{-1}(\vec{p})  ]^\Gamma,   \,\text{and} \,v_{i}\in \mathscr{H}(\lambda_{i}),
\end{equation}
where $X_{p_{i}}$ denotes the restriction of $X$ to
$\pi^{-1}(\mathbb{D}^\times_{p_i})$, hence $X_{p_i}$ is an element in $\hat{\fg}_{p_i}$.

By the residue theorem [H, Theorem 7.14.2, Chap. III],
\begin{equation}
\sum_{q\in \pi^{-1}( \vec{p}) } \text{~Res}_{q}  \langle dX,Y  \rangle   =0,\quad \text{for any}\quad
X,Y \in  \fg[ {{\Sigma}} \backslash \pi^{-1} ( \vec{p}) ]^{\Gamma}.  \label{eq6}
\end{equation}

Thus, the action \eqref{componentwise_action} indeed is an action of the Lie algebra
$\mathfrak{g}[{{\Sigma}}\backslash \pi^{-1}( \vec{p}) ]^\Gamma$ on
$\mathscr{H}(\vec{\lambda})$. 

Finally, we are ready to define the {\it space of  (twisted) vacua}
\begin{equation}
{\mathscr{V}_{{{\Sigma}},\Gamma,\phi  }(\vec{p},\vec{\lambda})^{\dagger}} :=\Hom_{\mathfrak{g}[{\Sigma}\backslash \pi^{-1}(\vec{p})]^\Gamma}(\mathscr{H} (\vec{\lambda}),\mathbb{C}),\label{eq7}
\end{equation}
and  the {\it space of  (twisted) covacua}
\begin{equation}
\mathscr{V}_{{{\Sigma}},\Gamma,\phi  }(\vec{p},\vec{\lambda}):=[\mathscr{H}(\vec{\lambda})]_{\mathfrak{g}[{{\Sigma}}\backslash \pi^{-1} (\vec{p}) ]^\Gamma},\label{eq8}
\end{equation}
where $\mathbb{C}$ is considered as the trivial module under the
action of $\mathfrak{g}[{{\Sigma}}\backslash \pi^{-1} (\vec{p})]^\Gamma$, and \\
$[\mathscr{H} (\vec{\lambda})]_{\mathfrak{g}[{\Sigma}\backslash
   \pi^{-1} ( \vec{p} ) ]^\Gamma }$ denotes the space of covariants
$\mathscr{H} (\vec{\lambda})/\bigl(\mathfrak{g}[{{\Sigma}}\backslash
 \pi^{-1}( \vec{p} ) ]^\Gamma  \cdot \mathscr{H} (\vec{\lambda})\bigr)$.
Clearly,
\begin{equation}
{\mathscr{V}_{{{\Sigma}},\Gamma,\phi}(\vec{p},\vec{\lambda})^{\dagger}} \simeq \mathscr{V}_{{{\Sigma}},\Gamma,\phi}(\vec{p},\vec{\lambda})^{*}.\label{eq9}
\end{equation}
\end{definition}
\begin{remark}
\label{notation_warning} {\rm Fix any $q_i\in \pi^{-1}(p_i)$. 
If we choose $\vec{\lambda}=(\lambda_1,\cdots, \lambda_s)$ to be a set of weights,  where for each $i$, $\lambda_i$ is a dominant weight of $\fg^{\Gamma_{q_i}}$ in $D_{c,q_i}$, we can transfer each $\lambda_i$ to an element in $D_{c,p_i}$ through the restriction isomorphism ${\fg}_{p_i}\simeq  \fg^{\Gamma_{q_i}}$ via Lemma \ref{evaluation_lem}. Accordingly, we denote the associated space of covacua by $\mathscr{V}_{{{\Sigma}},\Gamma,\phi  }(\vec{q},\vec{\lambda})$.  This terminology will  often be used interchangeably. }
\end{remark}

\begin{lemma}\label{lem2.1.3}
With the notation and assumptions as in Definition \ref{def1.2}, the
space of covacua
$\mathscr{V}_{{\Sigma},\Gamma, \phi}(\vec{p},\vec{\lambda})$ is
finite dimensional and hence by equation  \eqref{eq9}, so is the space of vacua $\mathscr{V}_{{\Sigma},\Gamma,\phi} (\vec{p},\vec{\lambda})^{\dagger}$.
\end{lemma}

\begin{proof}
Let $\fg[\pi^{-1}(\mathbb{D}^\times_{\vec{p}}  ) ]^\Gamma$ be the space of $\Gamma$-equivariant maps from the disjoint union of  formal punctured  discs $\sqcup_{q\in \pi^{-1}(\vec{p}) } \mathbb{D}^\times_{q}$ to $\fg$. 
Define a Lie algebra bracket on 
 \begin{equation}\label{Lie_algebra_multiple_points}
   \hat{\fg}_{\vec{p}} := \fg[\pi^{-1}(\mathbb{D}^\times_{\vec{p}}  ) ]^\Gamma \oplus \mathbb{C}C , \end{equation}
 by declaring $C$ to be the central element and the Lie bracket is defined in the similar way as in (\ref{pointwise-affine-Lie}).
 
Now, define an embedding of Lie algebras:
$$
\beta:\mathfrak{g}[{\Sigma}\backslash \pi^{-1} ( \vec{p} ) ]^\Gamma \to
\hat{\fg}_{\vec{p}}, \ X\mapsto
 X_{\vec{p}}  
 $$
where $ X_{\vec{p}}  $   is the restriction of   $X$  to $\pi^{-1}(\mathbb{D}^\times_{\vec{p}} )$ .

By the Residue Theorem,  $\beta$ is
indeed a Lie algebra homomorphism. Moreover, by Riemann-Roch theorem, $\Iim \beta + \fg [\pi^{-1} (\mathbb{D}_{\vec{p}} )]^\Gamma  $ has finite codimension in
${\hat{\mathfrak{g}}}_{\vec{p}} $, where $\pi^{-1}(\mathbb{D}_{\vec{p}} )$ is the disjoint union $\sqcup_{q\in \pi^{-1}(\vec{p})}\mathbb{D}_q$. Further, define the following
surjective Lie algebra homomorphism from the direct sum Lie algebra
$$
\bigoplus\limits^{s}_{i=1}  \hat{\fg}_{p_i}   \to
{\hat{\mathfrak{g}}}_{\vec{p}} \,,\,\,\,   \sum\limits^{s}_{i=1} X_i
\mapsto \sum\limits^{s}_{i=1} \tilde{X}_i, \ C_{i}\to C,\,\,\,
$$here $C_{i}$ is the center $C$ of
$\hat{\fg}_{p_{i}}$, and the map $X_i\in \fg [  \pi^{-1} ( \mathbb{D}^\times_{p_i})  ]^\Gamma $ naturally extends to $\tilde{X}_i \in  \fg[ \pi^{-1}(\mathbb{D}^\times_{\vec{p}} )]^\Gamma$ by taking $\pi^{-1}(\mathbb{D}^\times_{p_j} )$   to $0$  if $j\neq i$.  

\smallskip

Now, the lemma follows from [Ku, Lemma 10.2.2].

\end{proof}
\section{Propagation of twisted vacua}\label{Propagation_section}
We prove the Propagation Theorem in this section, which  asserts that adding marked points and attaching weight 0 to those points does not alter the space of twisted vacua. 

  Let ${\Sigma}\to \bar{\Sigma}$ be  a  $\Gamma$-cover (cf. Definition \ref{Gamma_curve}). Moreover,  $\phi:\Gamma\to {\rm Aut}(\fg)$ is a group homomorphism.

\begin{definition}\label{defi2.2.1}
Let $\vec{o}=(o_{1},\ldots,o_{s})$ and
$\vec{p}=(p_{1},\ldots,p_{a})$ be two disjoint non-empty
sets of smooth and distinct points in $\bar{\Sigma}$ such that $(\bar{\Sigma}, \vec{o})$ is a  $s$-pointed curve and let
$\vec{\lambda}=(\lambda_{1},\ldots,\lambda_{s})$,
$\vec{\mu}=(\mu_{1},\ldots,\mu_{a})$ be tuples of dominant weights such that $\lambda_i\in D_{c, o_i}$ and $\mu_j\in D_{c,p_j}$ for each $1\leq i\leq s, 1\leq j\leq a$.

We assume  that $\pi^{-1}(o_i)$ and $\pi^{-1}(p_j)$ consist of smooth points.

Denote the tensor product 
\begin{equation}
V(\vec{\mu}):=V(\mu_{1})\otimes\cdots\otimes V(\mu_{a}),
\end{equation}
where $V(\mu_{k})$ is the irreducible $\fg_{p_k}$-module with highest weight $\mu_k$. 

Define a $\mathfrak{g}[{\Sigma}\backslash \pi^{-1}( \vec{o}) ]^\Gamma$-module
structure on $V(\vec{\mu})$ as follows:
\begin{equation}
X \cdot (v_{1}\otimes\cdots\otimes
  v_{a})=\sum^{a}_{k=1}v_{1}\otimes\cdots\otimes X|_{p_k}\cdot
  v_{k}\otimes\cdots\otimes v_{a},
\end{equation}
for
$
v_{k}\in V(\mu_{k}), X \in \mathfrak{g}[{\Sigma}\backslash \pi^{-1}( \vec{o}) ]^\Gamma $, and $X|_{p_k}$ denotes the restriction $X|_{\pi^{-1}(p_k)} \in \fg_{p_k}$.   
This gives rise to the tensor product   $\mathfrak{g}[\Sigma\backslash
\pi^{-1} ( \vec{o} )  ]^\Gamma$-module structure on
$\mathscr{H} (\vec{\lambda})\otimes V(\vec{\mu})$.
 \end{definition}

The proof of the following lemma was communicated to us by J.\,Bernstein.
\begin{lemma}
\label{quasi-split-lemma}
Assume that $\Gamma$ stabilizes a Borel subalgebra $\fb\subset \fg$. Then, there exist a Cartan subalgebra $\fh\subset \fb$ such that $\Gamma$ stabilizes $\fh$. 
\end{lemma}
\begin{proof}
Let $G$ be the  simply-connected simple algebraic group associated to $\fg$, and let $B$ be the Borel subgroup associated to $\fb$. Let $N$ be the unipotent radical of $B$. Then, $\Gamma$ acts on $N$. 
  It is known that the space of all Cartan subalgebras in $\fb$ is a $N$-torsor (it follows easily from the conjugacy theorem of Cartan subalgebras).  
Let $\fh_o$ be any fixed Cartan subalgebra in $\fb$.  It defines a function $\psi: \Gamma \to N$ given by $\gamma\mapsto u_\gamma$, where $u_\gamma$ is the unique element in $N$ such that ${\rm Ad}\,{ u_{\gamma}}(\fh_o)= \gamma(\fh_o)$. It is easy to check that $\psi$ is a 1-cocycle of $\Gamma$ with values in $N$.  
Note that the group cohomology $H^1(\Gamma, N)=0$ since $\Gamma$ is a finite group and $N$ is unipotent. It follows that $\psi$ is a $1$-coboundary, i.e., there exists $u_o\in N$ such that $\psi(\gamma)=\gamma( u_o)^{-1}u_o$ for any $\gamma\in \Gamma$.  Set $\fh={\rm Ad}\,u_o(\fh_o)$.  It is now easy to verify that $\fh$ is $\Gamma$-stable.
\end{proof} 

\begin{theorem}\label{Propagation_thm}
With the notation and assumptions as in Definition \ref{defi2.2.1},  assume further that $\Gamma$ stabilizes a Borel subalgebra of $\fg$.  
Then,  the
canonical map
$$
\theta:\left[\mathscr{H} (\vec{\lambda})\otimes
  V(\vec{\mu})\right]_{\mathfrak{g}[\Sigma\backslash
  \pi^{-1} ( \vec{o})  ]^\Gamma}\to \mathscr{V}_{{\Sigma},\Gamma,\phi }\left((\vec{o},\vec{p}),(\vec{\lambda},\vec{\mu})\right)
$$
is an isomorphism, where $\mathscr{V}_{\Sigma,\Gamma,\phi}$ is the space of
covacua  and the map $\theta$ is induced from the
$\mathfrak{g}[{\Sigma}\backslash \pi^{-1}( \vec{o})]^\Gamma$-module embedding
$$
\mathscr{H} (\vec{\lambda})\otimes
V(\vec{\mu})\hookrightarrow \mathscr{H} (\vec{\lambda},\vec{\mu}),
$$
with $V(\mu_{j})$ identified as a $\mathfrak{g}_{p_j}$-submodule of
$\mathscr{H} (\mu_{j})$ annihilated by $\hat{\fg}_{p_j}^+$. (Observe
that since the subspace $V(\mu_{j})\subset \mathscr{H} (\mu_{j})$ is annihilated
by $\hat{\fg}_{p_j}^+$,  the embedding
$V(\mu_{j})\subset \mathscr{H} (\mu_{j})$ is indeed a
$\mathfrak{g}[{\Sigma}\backslash \pi^{-1}( \vec{o} ) ]^\Gamma$-module embedding.)
\end{theorem}

\begin{proof}
By Lemma \ref{quasi-split-lemma}, we may assume that $\Gamma$ stabilizes a Borel subalgebra $\fb$ and a Cartan subalgebra $\fh$ contained in $\fb$. From now on we fix such a $\fb$ and $\fh$.

Let $\mathscr{H} :=\mathscr{H} (\vec{\lambda})\otimes
V(\mu_{1})\otimes \cdots\otimes V(\mu_{a-1})$. By induction on $a$, it
suffices to show that the inclusion 
$
V(\mu_{a})\hookrightarrow \mathscr{H} (\mu_{a})
$
induces an isomorphism (abbreviating $\mu_{a}$ by $\mu$ and $p_{a}$ by
$p$)
\begin{equation}\label{neweqn25}
[\mathscr{H}\otimes
  V(\mu)]_{\mathfrak{g}[{\Sigma}^o]^\Gamma}\xrightarrow{\sim}[\mathscr{H}\otimes
  \mathscr{H} (\mu)]_{\mathfrak{g}[{\Sigma}^o\backslash \pi^{-1}(p)]^\Gamma},
\end{equation}
where ${\Sigma}^o := {\Sigma}\backslash  \pi^{-1} (\vec{o})$.

We first prove \eqref{neweqn25} replacing $\mathscr{H} (\mu)$ by the generalized
Verma module $\hat{M}(V(\mu),c)$ for $\hat\fg_p$ and the parabolic subalgebra $\hat\fp_p$, i.e.,
\begin{equation}\label{neweqn26}
[\mathscr{H}\otimes
  V(\mu)]_{\mathfrak{g}[{\Sigma}^o]^\Gamma}\xrightarrow{\sim}[\mathscr{H}\otimes
  \hat{M}(V(\mu),c)]_{\mathfrak{g}[{\Sigma}^o\backslash \pi^{-1}(p)]^\Gamma}. 
\end{equation}

Consider the Lie algebra
\begin{equation}
\mathfrak{s}_p := \mathfrak{g}[{\Sigma}^o\backslash \pi^{-1}(p) ]^\Gamma \oplus \mathbb{C}C,
\end{equation}
where $C$ is central in $\mathfrak{s}_p$ and

\begin{equation}
\label{affine_curve_center_ext}
[X, Y]=[X,Y]_0+  \frac{1}{|\Gamma|} \sum_{q\in \pi^{-1}(p)  } \Res_{q}  \langle dX,Y  \rangle  \,C,\text{~ for~ } X, Y\in \fg[{\Sigma}^o \backslash \pi^{-1}(p)  ]^\Gamma,
\end{equation}
where $[X,Y]_0$ is the point-wise Lie bracket. 

Let $\mathfrak{s}_p^{\geq 0}$ be the subalgebra of $\mathfrak{s}_p$:
\[  \mathfrak{s}_p^{\geq 0}:=\fg[{\Sigma}^o ]^\Gamma\oplus  \mathbb{C}C . \]

Fix a point $q\in \pi^{-1}(p)$ and a  generator $\sigma_q$ of $\Gamma_q$ such  that $\sigma_q$  acts on $T_q{\Sigma}$ by $\epsilon_q :=e^{\frac{2\pi i}{e_p}}$
(which is a primitive $e_p$-th root of unity).  By the Riemann-Roch theorem there exists a formal parameter $z_q$ around $q$ such that  $z_q^{-1}$ is a regular function on ${\Sigma}^o\backslash \{q\}$. Moreover, we require $z_q^{-1}$ to vanish at any other point $q'$ in $\pi^{-1}(p)$.  Replacing $z_q^{-1}$ by 
$$\sum_{j=1}^{e_p}\, \epsilon_q^{-j}\sigma_q^j(z_q^{-1}),$$
we can (and will) assume that 
\begin{equation} \label{3.3.new} \sigma_q\cdot z_q^{-1} = \epsilon_qz_q^{-1}.
\end{equation} 
Recall the Lie algebras  $\hat{L}(\fg,\Gamma_q)$ and $\hat{L}(\fg,\Gamma_q)^-=(z_q^{-1}\fg[z_q^{-1}])^{\Gamma_q}$ from $\S$2. Since $z_q$ is a formal parameter 
at $q$ with $ \sigma_q\cdot z_q= \epsilon_q^{-1}z_q$, we have 
\begin{equation} \label{decomposition_2}
\hat{L}(\fg, \Gamma_q)=\hat{L}(\fg,\Gamma_q)^{\geq 0}\oplus    (z_q^{-1}\fg[z_q^{-1}])^{\Gamma_q}.\end{equation}
Define, for any $x\in \mathfrak{g}$ and $k\geq 1$,
\[A(x[z_q^{-k}]) := \frac{1}{|\Gamma_q|} \sum_{\gamma\in \Gamma}\,\gamma\cdot (x[z_q^{-k}]) \in \mathfrak{s}_p,\]
and let $V\subset \mathfrak{s}_p$ be the span of $\{A(x[z_q^{-k}])\}_{x\in \mathfrak{g}, k \geq 1}$. 
 It is easy to check that 
\begin{equation}    \mathfrak{s}_p=\mathfrak{s}_p^{\geq 0}\oplus   V.   \label{decomposition_1}
\end{equation}
By Lemmas \ref{evaluation_lem} and  \ref{lemma2.3}, we can view $\hat{M}(V(\mu),c)$ as a generalized Verma module over $\hat{L}(\fg,\Gamma_q)$ induced from $V(\mu)$ as $\hat{L}(\fg, \Gamma_q)^{\geq 0}$-module. 

Consider the embedding of the Lie algebra
$$
\mathfrak{s}_p\hookrightarrow \hat{L}(\fg,\Gamma_q)
$$
by taking $C\mapsto C$ and any $X\mapsto X_q$.
We assert  that the above embedding $ \mathfrak{s}_p \hookrightarrow \hat{L}(\fg, \Gamma_q)$ induces a vector space isomorphism 
\begin{equation} \label{eqn3.3.27}
\gamma:  \mathfrak{s}_p/ \mathfrak{s}_p^{\geq 0} \simeq \hat{L}(\fg, \Gamma_q)/\hat{L}(\fg, \Gamma_q)^{\geq 0}.
\end{equation}
To prove the above isomorphism, observe first that $\gamma$ is injective: For $\alpha\in \mathfrak{g}[{\Sigma}^o\setminus \pi^{-1}(p)]^\Gamma$, if $\gamma(\alpha)\in \hat{L}(\fg, \Gamma_q)^{\geq 0}$, then $\alpha\in  \mathfrak{g}[({\Sigma}^o\setminus \pi^{-1}(p))\cup \{q\}]$. The $\Gamma$-invariance of $\alpha$ forces $\alpha\in  \mathfrak{g}[{\Sigma}^o]$, proving the injectivity of $\gamma$. To prove the surjectivity of $\gamma$, take a $\Gamma_q$-invariant $\alpha=x[z_q^{-k}]$ for $k \geq 1$. Thus, $\sigma_q(x)= \epsilon_q^{-k}x$. By the definition, since $z_q^{-1}$ vanishes at any point $q'\in \pi^{-1}(p)$ different from $q$,
\[\gamma(A(\alpha))=\alpha + \hat{L}(\fg, \Gamma_q)^{\geq 0}.\]
This proves the surjectivity of $\gamma$.  Thus, by the  PBW theorem,  as $\mathfrak{s}_p$-modules 
\begin{equation}
\label{Verma_iso}
\hat{M}(V(\mu),c)\simeq  U(\mathfrak{s}_p)\otimes_{U(\mathfrak{s}_p^{\geq 0} )}V(\mu).  
\end{equation}

Let $\mathfrak{g}[{\Sigma}^o\backslash \pi^{-1}(p)]^\Gamma$ act on $\mathscr{H}$ as follows:
\begin{align*}&X \cdot (v_1\otimes \dots \otimes v_s\otimes w_1\otimes \dots \otimes w_{a-1}) \\
&= \sum_{i=1}^s
v_1\otimes \dots \otimes X_{o_i}\cdot v_i\otimes \dots \otimes v_s\otimes w_1\otimes \dots \otimes w_{a-1} \\
&+
\sum_{j=1}^{a-1} v_1\otimes \dots \otimes v_s\otimes w_1\otimes \dots \otimes X|_{p_j}\cdot w_j\otimes \dots \otimes w_{a-1}\\
&\,\,\,\text{for $X\in 
\fg[{\Sigma}^o\backslash \pi^{-1}(p)]^\Gamma, v_i\in \mathscr{H}(\lambda_i)$ and $w_j \in V(\mu_j)$},
\end{align*}
and let $C$ act on
$\mathscr{H}$ by the scalar $-c$. By the Residue Theorem, these actions combine to make
$\mathscr{H}$ into an $\mathfrak{s}_p$-module.
Thus, the action of $C$ on the tensor product $\mathscr{H}\otimes
\hat{M}(V(\mu),c)$ is trivial.

Now, by the isomorphism (\ref{Verma_iso}) (in the following,
$\mathfrak{g}[{\Sigma}^o]^\Gamma$ acts on $V(\mu)$ via its restriction on $\pi^{-1}(p)$ and $C$
acts via the scalar $c$)
\begin{align*}
\left[ \mathscr{H}\otimes
  \hat{M}(V(\mu),c)\right]_{\mathfrak{g}[{\Sigma}^o\backslash \pi^{-1}(p)]^\Gamma} &=
\left[\mathscr{H}\otimes
  \hat{M}(V(\mu),c)\right]_{\mathfrak{s}_p },
\,\,\,\text{since $C$ acts trivially}\\
&\simeq \mathscr{H}\otimes_{U(\mathfrak{s}_p )}\hat{M}(V(\mu),c)\\
&\simeq
\mathscr{H}\otimes_{U(\mathfrak{s}_p  )}\left(U(\mathfrak{s}_p )\otimes_{U(\mathfrak{g}[{\Sigma}^o]^\Gamma \oplus
  \mathbb{C}C)}V(\mu)\right)\\
&\simeq \mathscr{H}\otimes_{U(\mathfrak{g}[{\Sigma}^o]^\Gamma \oplus
    \mathbb{C}C)} V(\mu)\\
&= \mathscr{H}\otimes_{U(\mathfrak{g}[{\Sigma}^o]^\Gamma )}V(\mu)\\
&= [\mathscr{H}\otimes V(\mu)]_{\mathfrak{g}[{\Sigma}^o]^\Gamma}.
\end{align*}
This proves \eqref{neweqn26}.

Now, we come to the proof of \eqref{neweqn25}:

Let $K(\mu)$ be the kernel of the canonical projection
$\hat{M}(V(\mu),c)\twoheadrightarrow \mathscr{H} (\mu)$.
In view of \eqref{neweqn26}, to prove \eqref{neweqn25}, it suffices to show that the image of
$$
\iota: \left[\mathscr{H}\otimes K(\mu)\right]_{\mathfrak{g}[{\Sigma}^o\backslash \pi^{-1}(p)]^\Gamma}\to
\left[\mathscr{H}\otimes
  \hat{M}(V(\mu),c)\right]_{\mathfrak{g}[{\Sigma}^o\backslash \pi^{-1}(p)]^\Gamma}
$$
is zero:
From the isomorphism \eqref{eqn3.3.27}, we get 
$$
\hat{L}(\fg, \Gamma_q) =\mathfrak{s}_p+\hat{L}(\fg,\Gamma_q)^{\geq 0}.
$$
Moreover, write 
$$\hat{L}(\fg, \Gamma_q)^{\geq 0}=\hat{L}(\fg, \Gamma_q)^+ + \fg^{\Gamma_q}+ \mathbb{C} C,$$
and observe that any element of $\fg^{\Gamma_q}$ can be (uniquely) extended to an element of $\fg_p:=\fg[\pi^{-1}(p)]^\Gamma$ (cf. Lemma
 \ref{evaluation_lem}). Further, ${\Sigma}^o$ being affine, the restriction map $\mathfrak{g}[{\Sigma}^o]^\Gamma \to \fg_p$ is surjective, and, of course,  $\mathfrak{g}[{\Sigma}^o]^\Gamma \subset \mathfrak{s}_p^{\geq 0}$. Thus, we get the decomposition:
$$
\hat{L}(\fg, \Gamma_q) =\mathfrak{s}_p+\hat{L}(\fg,\Gamma_q)^+,
$$
and hence, by the Poincar\'e-Birkhoff-Witt theorem,
$U(\hat{L}(\fg,\Gamma_q))$ is the span of elements of
the form
$$
Y_{1}\ldots Y_{m}\cdot X_{1}\ldots X_{n}, \text{~~ for~~ } Y_{i}\in
\mathfrak{s}_p, X_{j}\in\hat{L}(\fg,\Gamma_q)^+  \text{~~
  and~~ } m, n\geq 0.
$$
Consider the decomposition  \eqref{eq1.1.1.0}  for $\sigma_q$: $\sigma_q=\tau_q \epsilon_q^{{\rm ad} h}$, under a choice of $\sigma_q$-stable Borel subalgebra $\fb_q$ containing the same Cartan subalgebra $\fh$ in the sense of Section \ref{Kac_Moody_Section}. 
(Since $\Gamma$, in particular $\sigma_q$,  stabilizes the pair $(\fb, \fh)$,  as in the proof of Lemma \ref{lemma 1.3}, $\fh^{\sigma_q}=\fh^{\tau'_q}$  for some diagram automorphism  $\tau'_q$ of $\fg$ associated to the pair $(\fb, \fh)$. In particular, the centralizer $Z_{\mathfrak{g}}(\mathfrak{h}^\sigma)$ of $\mathfrak{h}^\sigma$ in $\mathfrak{g}$ equals $ \mathfrak{h}$ and hence we can take  $\fh_q=\fh$.)
Under such a choice,  there exist $sl_2$-triples $x_i,y_i,h_i \in \fg$ for each $i\in \hat{I}(\fg, \sigma_q)$  such that $\tilde{x}_i:=x_i[z_q^{s_i}], \tilde{y}_i:=y_{i}[z_q^{-s_i}], i\in \hat{I}(\fg,\sigma_q)$ are Chevalley generators of $\hat{L}(\fg, \sigma_q)$. Moreover, $x_i,y_i$ satisfy
\begin{equation}\label{eqn28new} \sigma_q(x_i)= \epsilon_q^{s_i} x_i, \,\text{ and } \sigma_q(y_i)=\epsilon_q^{-s_i} y_i .
\end{equation}

Let $v_+$ be the highest weight vector of $\hat{M}(V(\mu),c)$. 
  Recall (cf. \eqref{eqn1.2.0})  that $K(\mu)$ is generated by $\tilde{y}^{n_{\mu,i}+1}_i\cdot v_+$, for  $i\in \hat{I}(\fg,\sigma_q)^+$ consisting of $i\in \hat{I}(\fg, \sigma_q)$   such that $s_i>0$. 

Thus, to prove the vanishing of the map $\iota$, 
it suffices to show that  for any $i\in \hat{I}(\fg,\sigma_q)^+$
\begin{equation}\label{neweqn28}
\iota\left(h\otimes (X_{1}\ldots X_{n}\cdot \tilde{y}_i^{n_{\mu,i}+ 1}\cdot v_{+})\right)=0,
\end{equation}
for $h\in \mathscr{H}$, any $n\geq 0$ and $X_{j}\in
\hat{L}(\fg, \Gamma_q)^+$.
But, $\tilde{y}_i^{n_{\mu,i}+ 1}\cdot v_{+}$ being a highest weight vector, 
$$
\hat{L}(\fg, \Gamma_q)^+\cdot (\tilde{y}_i^{n_{\mu,i}+ 1}\cdot v_{+})=0.
$$

Thus, to prove \eqref{neweqn28}, it suffices to show that for any $i\in \hat{I}(\fg,\sigma_q)^+$
\begin{equation}\label{neweqn29}
\iota (h\otimes ( \tilde{y}_i^{n_{\mu,i}+ 1}\cdot v_{+}))=0,\quad\text{for any}\quad h\in \mathscr{H}. 
\end{equation}

Fix $i\in \hat{I}(\fg, \sigma_q)^+$. Take $f\in \mathbb{C}[{\Sigma}^o]$ such that
$$f_{q}\equiv z_q^{s_i}\,(\text{mod~ } z_q^{s_i+1}),
$$
and the order of vanishing of $f$ at any $q'\neq q\in \pi^{-1}(p)$ is at least $(n_{\mu,i}+3)s_i$. Moreover, replacing $f$ by 
$\frac{1}{|\Gamma_q|}\sum_{j=1}^{|\Gamma_q|}\, \epsilon_q^{s_ij}\sigma_q^j\cdot f$, we can (and will) assume that 
$$\sigma_q\cdot f=  \epsilon_q^{-s_i}f . 
$$
Now, take 
$$Z=\sum_{\gamma \in \Gamma/\Gamma_q} \, \gamma \cdot (x_i[f]).$$
Then, writing $Y=\tilde{y}_i$, 
\begin{align} \label{eqn29new}Z^N Y^{n_{\mu,i}+ N+1}\cdot v_{+}&=\left(\sum_{\gamma \in \Gamma/\Gamma_q} \,\left( \gamma \cdot (x_i[f])\right)_q\right)^NY^{n_{\mu,i}+N+ 1}\cdot v_{+}\notag \\
&=(x_i[f_q])^NY^{n_{\mu,i}+ N+1}\cdot v_{+}.
\end{align}
To prove the last equality, observe that $ \left(\gamma_1 \cdot (x_i[f])\right)_q \dots  \left(\gamma_N \cdot (x_i[f])\right)_q$ has  zero of order at least $(n_{\mu,i}+3)s_i+(N-1)s_i$ unless each $\gamma_j\cdot \Gamma_q=\Gamma_q$. But, $Y^{n_{\mu,i}+ N+1}$ has order of pole equal to $(n_{\mu,i}+N+1)s_i$. Since $(n_{\mu,i}+3)s_i+(N-1)s_i >(n_{\mu,i}+N+1)s_i $, we get the last equality. 
Thus,
 by Lemma \ref{lem_kac} for $X=x_i[f_q]$ and $Y=\tilde{y}_i$, for any $N\geq 1$,  there exists $\alpha \neq 0$ such that 
\begin{align*}
\iota\left(h\otimes (Y^{n_{\mu,i}+ 1}\cdot v_{+})\right) &=
 \alpha \,\iota
\left(h\otimes X^{N}Y^{n_{\mu,i}+ N+1}\cdot v_{+}\right)\\
&=
 \alpha \,\iota
\left(h\otimes Z^{N}Y^{n_{\mu,i}+ N+1}\cdot v_{+}\right), \,\,\,\text{by} \,\eqref{eqn29new}\\
&= (-1)^{N}\alpha \,\iota\left(Z^{N}\cdot h\otimes
Y^{n_{\mu,i} + N+1}\cdot v_{+}\right)\\
&= 0,\text{~~ by~ Lemma \ref{lemma 1.3} for large $N$ (see the argument below).}
\end{align*}
This proves \eqref{neweqn29} and hence completes the proof of the theorem. 

We now explain more precisely how Lemma \ref{lemma 1.3} implies $Z^N\cdot h=0$.  
With respect to the pair $(\fb,\fh)$ stable under $\Gamma$  (note that $\fb$ might not be the same as $\fb_q$ given above \eqref{eq1.1.1.0} though $\fh_q$ is taken to be $\fh$),  since $\Gamma$ preserves the pair $(\fb,\fh)$, the group $\Gamma$ acts on the root system $\Phi(\fg,\fh)$ of $\fg$ by factoring through the group of outer automorphisms with respect to the pair $(\fb, \fh)$. In particular, $\Gamma$ preserves the set of positive (resp. negative) roots.  
From the construction of $x_i$ in $\S$\ref{Kac_Moody_Section}, $x_i$ is either a linear combination of positive root vectors or a linear combination of negative root vectors with respect to $\fb_q$. Thus, either  $\gamma\cdot x_i\in \mathfrak{n}$ for all $\gamma\in \Gamma$, or $\gamma\cdot x_i\in \mathfrak{n}^-$ for all $\gamma\in \Gamma$, where $\fb^-$ is the negative Borel of $\fb$ and $\mathfrak{n}$ (resp. $\mathfrak{n}^-$) is the nil-radical of $\mathfrak{b}$ (resp. $\mathfrak{b}^-$).   Therefore, we may apply Lemma  \ref{lemma 1.3} to show $Z^N\cdot h=0$.  
\end{proof}
\begin{remark} {\rm Observe that the condition that $\Gamma$ stabilizes a Borel subalgebra $\mathfrak{b}$ and hence also a Cartan subalgebra $\mathfrak{h}\subset \mathfrak{b}$ is equivalent to the condition that the image of $\Gamma$ in Aut $\mathfrak{g}$ is contained in $D\ltimes \text{Int} H$, where $D$ is the group of diagram automorphisms of $\fg$ and $H$ is the maximal torus of $G$ with Lie algebra $\mathfrak{h}$ ($G$ being the adjoint group with Lie algebra $\fg$).}
\end{remark}
The following result is  the twisted analogue of  ``Propagation of Vacua'' due to Tsuchiya-Ueno-Yamada [TUY].

\begin{coro}\label{coro2.2.3}
With the notation and assumptions as in Theorem~\ref{Propagation_thm} (in particular, $(\bar{\Sigma} ,\vec{o})$ is a $s$-pointed curve), for any smooth point $q\in {\Sigma}^o := {\Sigma}\setminus \pi^{-1}(\vec{o})$  (thus $p=\pi(q)$ is a smooth point of $\bar{\Sigma}$)  with $0\in D_{c, q}$ (cf. Corollary \ref{newcoroweight0}), there are canonical
isomorphisms:
\begin{itemize}
\item[{\rm(a)}]
  $\mathscr{V}_{\Sigma,\Gamma,\phi}(\vec{o},\vec{\lambda})\simeq
  \mathscr{V}_{\Sigma,\Gamma,\phi}((\vec{o},p),(\vec{\lambda},0))$,
  and

\item[{\rm(b)}]
  For $\bar{\Sigma}$ an irreducible curve, $\mathscr{V}_{\Sigma,\Gamma,\phi}(\vec{o},\vec{\lambda})\simeq
  [\mathscr{H} (0)\otimes
    V(\vec{\lambda})]_{\mathfrak{g}[\Sigma\backslash \pi^{-1}(p)]^\Gamma}$,
where the point $p$ is assigned weight $0$.
\end{itemize}
\end{coro}

\begin{proof}
(a): Apply Theorem \ref{Propagation_thm} for the case
  $\vec{p}=(p)$ and $\vec{\mu}=(0)$.

(b): It follows from Theorem~\ref{Propagation_thm} and the (a)-part.
(In Theorem~\ref{Propagation_thm} replace $\vec{o}$ by the
  singleton $(p)$, $\vec{\lambda}$ by $(0)$,
  $\vec{p}$ by $\vec{o}$ and
  $\vec{\mu}$ by $\vec{\lambda}$.)
\end{proof}

\begin{remark}  (a)  {\rm A much weaker form of the above Corollary part (a) (where $\Gamma$ is of order $2$ and  $\vec{o}$  consists of all the ramification points) is proved in [FS, Lemma 7.1]. It should be mentioned that they use the more general setting of twisted Vertex Operator Algebras. }
\vskip1ex
(b)  {\rm When all the marked points are unramified and $|\Gamma|$ is a prime, the Propagation of Vacua is proved in [D].}
\end{remark}
\vskip4ex

\section{Factorization Theorem}\label{Factorization_section}

The aim of this section is to prove the Factorization Theorem which  identifies the space of covacua for a genus $g$ nodal curve
$\bar{\Sigma}$ with a direct sum of the spaces of covacua for its normalization $\bar{\Sigma}'$ (which is a genus $g-1$ curve).

Let
  $\pi: {\Sigma}\to \bar{\Sigma}$ be  a  $\Gamma$-cover of a $s$-pointed curve $(\bar{\Sigma},  \vec{o})$. We do {\it not} assume that $\bar{\Sigma}$ is irreducible. Moreover, $\phi:\Gamma \to$ Aut$(\fg)$ is a group homomorphism. 
 
\begin{definition}
\label{stable_action}
 [BR, D\'efinition 4.1.4] Let ${\Sigma}$ be a reduced (but not necessarily connected) projective curve with at worst only simple nodal singularity. (Recall that a point $P\in {\Sigma}$ is called a {\em simple node}  if analytically a neighborhood of $P$ in ${\Sigma}$ is isomorphic with an analytic neighborhood of $(0,0)$ in the curve $xy=0$ in $\mathbb{A}^2$.) Then, the action of $\Gamma$ on $\Sigma$  at any simple node  
$q\in \Sigma$ is called {\it stable} if  the derivative $\dot{\sigma}$ of  any element $\sigma\in \Gamma_q$  acting on the Zariski tangent space $T_q(\Sigma)$ satisfies the following: 

\begin{align}\label{eqn4.0.1}\det (\dot{\sigma})&=1\,\,\,\text{ if $\sigma$ fixes the two branches at $q$,}\notag\\
&=-1\,\,\,\text{ if $\sigma$ exchanges the two branches.}
\end{align}
We say that $\Gamma$ {\it acts stably on $\Sigma$} if it acts stably on each of its nodes.

{\it From now on, by a node we will always mean a simple node.} 
\end{definition}
 Assume that $p\in \bar{\Sigma}$ is a node (possibly among other nodes) and also assume that 
  the fiber $\pi^{-1}(p)$ consists of nodal points. Assume further that the action of $\Gamma$  at the points $q\in \pi^{-1}(p)$ is stable. Observe that, in this case,  since $p$ is assumed to be a node, any $\sigma \in \Gamma_q$ can not exchange the two branches at $q$ for otherwise the point $p$ would be smooth. 

 We fix a level $c\geq 1$.


Let $\bar{\Sigma}'$ be the curve obtained from $\bar{\Sigma}$ by the normalization $\bar{\nu}:\bar{\Sigma}'\to \bar{\Sigma}$ at only the point $p$. Thus, $\bar{\nu}^{-1}(p)$ consists of two smooth points $p',p''$  in $\bar{\Sigma}'$ and 
$$
\bar{\nu}_{|\bar{\Sigma}'\backslash \{ p',p'' \}  }:\bar{\Sigma}'\backslash \{ p',p'' \}  \to \bar{\Sigma}\backslash \{p\}
$$
is a biregular isomorphism.   We  denote the preimage of any point of $\bar{\Sigma}\backslash \{p\}$ in $\bar{\Sigma}'\backslash \{ p',p''  \}$  by the same symbol.  
Let $\pi' : {\Sigma}' \to \bar{\Sigma}'$ be the pull-back of the Galois cover $\pi$ via
$\bar{\nu}$. In particular, $\pi'$ is a Galois cover with Galois group $\Gamma$.   Thus, we have the fiber diagram:
\[
\xymatrix@=.5cm{
{\Sigma}'\ar[dd]_{\pi' }
\ar[rr]^{\nu} && \Sigma
\ar[dd]^{\pi}\\
 & \square & \\
\bar{\Sigma}'  \ar[rr]_{\bar{\nu}} &&
\bar{\Sigma}.
}
\]


\begin{lemma}
\label{normalization_pullback}
With the same notation and assumptions as in Definition \ref{stable_action},  
\begin{enumerate}
\item the map $\nu$ is a normalization of ${\Sigma}$ at every point $q\in \pi^{-1}(p)$;
\item there exists a natural $\Gamma$-equivariant bijection $\kappa:  \pi'^{-1}(p')\simeq \pi'^{-1}(p'')$;
\item  for any $q\in \pi^{-1}(p) $,  we have 
\[ \Gamma_{q}= \Gamma_{q'}= \Gamma_{q''} , \]
where $\nu^{-1}(q)$ consists of two smooth points $q',q''$, and $\Gamma_{q}, \Gamma_{q'}$ and $\Gamma_{q''}$ are stabilizer groups of $\Gamma$ at $q, q'$ and $q''$ respectively. Moreover, $ \Gamma_{q}= \Gamma_{q'}= \Gamma_{q''}$ is a cyclic group.
\end{enumerate}
\end{lemma}
\begin{proof}
Let $q$ be any point in $\pi^{-1}(p)$ of ramification index $e_q$. Since $\pi^{-1}(p)$  consists of nodal points by assumption, there are two branches in the formal neighborhood of $q$. If any $\sigma \in \Gamma_q$ exchanges two branches then the point $p=\pi(q)$ is smooth in $\bar{\Sigma}$, which contradicts 
the assumption that $p$ is a nodal point. Thus,   $\Gamma_q$ must preserve branches.  In particular, since no nontrivial element  of $\Gamma$ fixes pointwise any irreducible component of $\Sigma$, $\Gamma_q$ is cyclic. Therefore, by the condition \eqref{eqn4.0.1},  we can choose a formal coordinate system $z',z''$ around the nodal point $q$ such that $\hat{\mathscr{O}}_{{\Sigma}, q }\simeq \mathbb{C}[[z',z'' ]]/(z'z'')$, and  a generator $\sigma_q$ of $\Gamma_q$ such that 
\[   \sigma_q(z')=\epsilon^{-1}z',  \text{ and } \sigma_q(z'')=\epsilon z'',  \]
where $\epsilon := e^{\frac{2\pi i}{e_q}}$ is the standard primitive     $e_q$-th root of unity.  (Observe that  $\epsilon$  must be a  primitive $e_q$-th root of unity, since $\Gamma_q$ acts faithfully on each of the two formal branches through $q$.) 

 We can choose a formal coordinate system $x',x''$ around $p$ in $\bar{\Sigma}$ such that $\hat{\mathscr{O}}_{{\bar{\Sigma}}, p }\simeq \mathbb{C}[[ x',x'' ]]/(x'x'')$ and the embedding $ \hat{\mathscr{O}}_{{\bar{\Sigma}}, p } \hookrightarrow    \hat{\mathscr{O}}_{{\Sigma}, q }  $ is given by $x'\mapsto (z')^{e_q}, x''\mapsto (z'')^{e_q}$.

The node $p$ splits into two smooth points $p',p''$ via $\bar{\nu}$.  Without loss of generality, we can assume $x'$ (resp. $x''$) is a formal coordinate around $p'$ (resp. $p''$) in $\bar{\Sigma}'$.  Then, $q$ will also split into two smooth points $q',q''$ via the map $\nu$, where    $z'$ (resp. $z''$) is a formal coordinate around $q'$ (resp. $q''$).  It shows that the map $\nu$ is a normalization at every point $q\in \pi^{-1}(p)$.  

The pullback gives a decomposition
\[(\pi\circ \nu)^{-1} (p)=  \nu^{-1}(\pi^{-1}(p) )=\pi'^{-1}(p')\sqcup \pi'^{-1}(p'') .  \]
From the definition of the fiber product, there exist $\Gamma$-equivariant  canonical bijections:
\begin{equation} \label{eqn4.1.1}\pi'^{-1}(p')\simeq \pi^{-1}(p) \,\,\,\text{and}\,\, \pi'^{-1}(p'')\simeq \pi^{-1}(p).
\end{equation}
Hence, we get a   $\Gamma$-equivariant canonical  bijection $\kappa:  \pi'^{-1}(p')\simeq \pi'^{-1}(p'')$.
For any $q\in \pi^{-1}(p) $,  $\nu^{-1}(q)=\{q',q''\}$.  By the choice of $q',q''$ as above,  $\pi'(q')=p'$ and $\pi'(q'')=p''$. 
Therefore, $\kappa$  maps $q'$ to $q''$. Moreove, from \eqref{eqn4.1.1},  the stabilizer groups $\Gamma_q$, $\Gamma_{q'} $ and $\Gamma_{q'} $ are all the same (and of order $e_q$). Since $q'$ (resp. $q''$) is a smooth point of $\Sigma'$, $\Gamma_{q'} $ (resp. $\Gamma_{q''} $) is cyclic.
\end{proof}

Let $\fg_p$ denote the Lie algebra $\fg[\pi^{-1}(p)]^\Gamma$  (observe that  we can attach a Lie algebra $\fg_p$ regardless of the smoothness of $p$). Then, the $\Gamma$-equivariant bijections $\nu: \pi'^{-1}(p')\simeq \pi^{-1}(p)$ and $\nu: \pi'^{-1}(p'')\simeq \pi^{-1}(p) $ (cf. equation \eqref{eqn4.1.1}) induce isomorphisms of Lie algebras $\varkappa': \fg_{p'}\simeq\fg_p$ and $\varkappa'':  \fg_{p''}\simeq\fg_{p}$ respectively.  Recall that $p', p''$ are smooth points of $\bar{\Sigma}'$. Let $D_{c,p'}$ (resp. $D_{c,p''}$) denote the finite set of highest weights of  irreducible representations of $\fg_p$ induced  via the isomorphism $\varkappa'$  
(resp.  $\varkappa''$) which give rise to integrable highest weight $\hat{\fg}_{p'}$-modules (resp. $\hat{\fg}_{p''}$-modules) with central charge $c$. 

Set 
\[ {\Sigma}^o={\Sigma}\backslash \pi^{-1}(\vec{o}), \,\text{ and }   {\Sigma}'^o={\Sigma}'\backslash \pi'^{-1}(\vec{o}). \]

The map ${\nu}$ on restriction gives rise to an isomorphism 
$$
\nu:  {\Sigma}'^o  \backslash \pi'^{-1}\{ p',p''\}\simeq \Sigma^o\backslash \pi^{-1} (p)\hookrightarrow \Sigma^o
$$
which, in turn, gives rise to a Lie algebra homomorphism 
$$\nu^*: \mathfrak{g}[ {\Sigma}^o]^\Gamma \to \mathfrak{g}[{\Sigma}'^o  \backslash \pi'^{-1}\{ p',p''\}]^\Gamma.$$

Let $\vec{\lambda}=(\lambda_1, \dots, \lambda_s)$ be an $s$-tuple of weights with $\lambda_i\in D_{c, o_i}$ `attached' to $o_i$.   We denote the highest weight of the dual representation  $V(\mu)^*$ of $\fg_p$  by $\mu^*$, thus  $V(\mu)^* \simeq V(\mu^*)$.

By Lemma \ref{normalization_pullback}, there exists a canonical  bijection $\kappa: \pi'^{-1}(p')\simeq \pi'^{-1}(p'')$ compatible with the action of $\Gamma$. Thus, it induces an isomorphism of Lie algebras $\fg_{p'}\simeq \fg_{p''}$.

\begin{lemma}\label{lemma 4.2} In the same setting as in Lemma \ref{normalization_pullback}, we have
\begin{enumerate}
\item there exists an isomorphism $\hat{\fg}_{p'}\simeq\hat{\fg}_{p''}$ which restricts to the  isomorphisms:
\[ \hat{\fp}_{p'}\simeq \hat{\fp}_{p''}, \, \hat{\fg}^+_{p'}\simeq \hat{\fg}^+_{p''}, \, \text{ and }  \fg_{p'}\simeq \fg_{p''} .  \]
See  the relevant notation in  \S \ref{conformal_block_section}.
\item  $ \mu \in D_{c,p'}$ if and only if $\mu^* \in  D_{c,p''}$.
\end{enumerate}

\end{lemma}
\begin{proof}
For any $q\in \pi^{-1}(p)$,  in view of Lemma \ref{lemma2.3}, the restriction gives isomorphisms ${\rm res}_{q'}: \hat{\fg}_{p'}\simeq \hat{L}(\fg, \Gamma_{q'})$ and ${\rm res}_{q''}: \hat{\fg}_{p''}\simeq \hat{L}(\fg, \Gamma_{q''})$. By  Lemma \ref{normalization_pullback}, $\Gamma_{q'}=\Gamma_{q''}$. As in \eqref{eq1.1.1.0}, let $\fb'$ (resp. $\fh' \subset \fb'$) be a suitable Borel (resp. Cartan) subalgebra of $\fg$ stable under $\Gamma_{q'}$. This gives rise to Chevalley generators $e_i \in \mathfrak{n}'$ and $f_i\in \mathfrak{n}'^-$, where $\mathfrak{n}'$ (resp. $\mathfrak{n}'^-$) is the nilradical of $\fb'$ (resp. the opposite Borel subalgebra $\fb'^-$). Let $\omega:\fg \to \fg$ be the Cartan involution taking the Chevalley generators of $\fg$: $e_j\mapsto -f_j, f_j\mapsto -e_j$ and $h \mapsto -h$ for any $h\in \fh'$. 

Write as in Section 2,
\[\sigma_{q'}=\tau' \epsilon^{\ad h'}\,\,\,\text{for a diagram automorphism $\tau'$ (possibly identity) and $h'\in \fh^{\tau'}$}.\]
Thus,
\begin{equation} \label{eqnnew4.2.1}
\omega^{-1} \sigma_{q'} \omega =\omega^{-1} \tau' \omega \epsilon^{\ad \omega^{-1}(h')}=\omega^{-1} \tau' \omega\epsilon^{\ad (-h')}.
\end{equation}
But, by the definition of (any diagram automorphism) $\tau'$ and $\omega$, it is easy to see that 
\begin{equation} \label{eqnnew4.2.2} \omega^{-1} \tau' \omega=\tau'.
\end{equation}
We now need to cosider two cases:

{\it Case I:} $\tau'$ is of order $1$ or $2$. In this case,

\begin{align} \label{eqnnew4.2.3}
\omega^{-1} \sigma_{q'} \omega &= \tau'\epsilon^{-\ad h'},\,\,\,\text{by \eqref{eqnnew4.2.1} and \eqref{eqnnew4.2.2}}\notag\\
&= \tau'^{-1}\epsilon^{-\ad h'},\,\,\,\text{since  $\tau'$ is assumed to be of  order $1$ or $2$} \notag\\
&=\sigma_{q'}^{-1}.
\end{align}

{\it Case II:} $\tau'$ is of order $3$, i.e., $\fg$ is of type $D_4$ with labelled nodes:

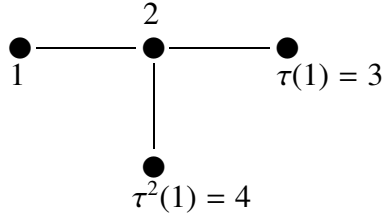
\begin{center}
\begin{picture}(150,80)
\put(0,65){\circle*{8}}
\put(50,65){\circle*{8}}
\put(100,65){\circle*{8}}
\put(50,20){\circle*{8}}

\put(6,65){\line(1,0){37}}
\put(56,65){\line(1,0){37}}
\put(50,59){\line(0,-1){32}}

\put(-4,51){1}
\put(46,74){2}
\put(96,51){$\tau'(1)=3$}
\put(42,5){$\tau'^2(1)=4$}


\end{picture}
\end{center}
 and $\tau'$ is the diagram automorphism induced from taking the nodes $1 \mapsto 3, 2\mapsto 2, 3\mapsto 4, 4 \mapsto 1$.   Let $\tau_1$  be the diagram automorphism induced from taking the nodes $1 \mapsto 1, 2\mapsto 2, 3\mapsto 4, 4 \mapsto 3$. Then,
\begin{equation} \label{eqnnew4.2.4}
\tau_1^{-1} \tau'\tau_1 =\tau'^{-1}. 
\end{equation}
In this case, we have

\begin{align} \label{eqnnew4.2.5}
(\omega\tau_1)^{-1} \sigma_{q'} \omega \tau_1&= \tau_1^{-1}\tau'\epsilon^{-\ad h'}\tau_1, \,\,\,\text{by \eqref{eqnnew4.2.1} and \eqref{eqnnew4.2.2}}\notag\\
&= \tau'^{-1}\tau_1^{-1}\epsilon^{-\ad h'}\tau_1,\,\,\,\text{by \eqref{eqnnew4.2.4}} \notag\\
&=\tau'^{-1}\epsilon^{-\ad h'},\,\,\,\text{since $(\tau_1)_{|\fh^{\tau'}}=$Id by [Ka, $\S$8.3, Case 4]}\notag\\
&=\sigma_{q'}^{-1}.
\end{align}
Let $\omega_o$ be the Cartan involution $\omega$ in the first case and $\omega\tau_1$ in the second case. Extend $\omega_o$ to an isomorphism of twisted affine Lie algebras:
\[\hat{\omega}_o: \hat{L}(\fg, \sigma_{q'}) \to   \hat{L}(\fg, \sigma_{q''}), \,\,\,\hat{\omega}_o(x[P(z')]):=\omega_o(x)[P(z'')], \hat{\omega}_o(C)=C,\]
for any $x\in \fg$ and $P\in \mathcal{K}$, where $\sigma_{q'}$ and $\sigma_{q''}=\sigma_{q'}^{-1}$ are the preferred generators of $\Gamma_q=\Gamma_{q'}=\Gamma_{q''}$
acting on a formal coordinate $z', z''$ around $q', q''$ respectively via $\epsilon^{-1}$ (see the proof of Lemma \ref{normalization_pullback}). Indeed, $\hat{\omega}_o$ is an isomorphism by the identities \eqref{eqnnew4.2.3} and  \eqref{eqnnew4.2.5}. Observe that $\hat{\omega}_o$ restricted to $\fh^{\sigma_{q'}}=\fh^{\sigma_{q''}}$ is nothing but the Cartan involution. 
Clearly, $\hat{\omega}_o$   restricts to an isomorphism  $\hat{\fp}_{p'}\simeq \hat{\fp}_{p''}$, $\hat{\fg}^+_{p'}\simeq \hat{\fg}^+_{p''}$ and $\fg_{p'}\simeq \fg_{p''}$ (see (\ref{2.1.2}) and (\ref{2.1.3}) for relevant notation).  This proves the first part of the lemma. 

From the isomorphism  $\hat{\omega}_o$, the second part of the lemma follows immediately since $\mathfrak{n}'^{\sigma_{q'}}$ is a maximal nilpotent subalgebra of $\fg^{\sigma_{q'}}$.
\end{proof}

We also give another proof of the second part of the above lemma.
\vskip1ex

\noindent
{\it Another proof of Lemma \ref{lemma 4.2} Part (2):}   Let $\sigma_{q'}$  (resp. $\sigma_{q''}$)  be the canonical generator of $\Gamma_{q'}$ (resp. $\Gamma_{q''}$). We can choose formal parameter
 $z'$ (resp. $z''$) around $q'$ (resp. $q''$) such that 
 \[ \sigma_{q'} (z')=\epsilon^{-1}z', \quad \sigma_{q''}(z'')=\epsilon^{-1} z'',   \]
 where $\epsilon=e^{\frac{2\pi i}{|\Gamma_q| } }$.     
 As in Section 2, we can write $\sigma_{q'}=\tau'\cdot \epsilon^{{\rm ad} h' } $. Let  $x'_{i}, y'_i, h'_i=[x'_i,y'_i] , i\in \hat{I}(\fg, \sigma_{q'})  $ be chosen as in Section 2, where 
 \[ x'_i\in (\fg^{\tau'})_{\alpha'_i}, y'_i\in (\fg^{\tau'})_{-\alpha'_i} , \,\text{ for any }i\in I(\fg^{\tau'}), \]
 where $\alpha'_i$ is the simple root of $\fg^{\tau'}$ associated to $i\in I(\fg^{\tau'})$, and 
 \[ x'_0\in (\fg^{\tau'})_{-\theta'_0}, \,y'_0\in (\fg^{\tau'})_{\theta'_0} . \]
Let  $s_i, i\in \hat{I}(\fg, \sigma_{q'})$ be the integers as in Section 2. We have, by the identity \eqref{eigenvalue_sigma},
\[  \sigma_{q'}(x'_i)=\epsilon^{s_i} x'_i, \quad \text{ and } \sigma_{q'}(y'_i)=\epsilon^{-s_i} y'_i , \]
 for any $i\in \hat{I}(\fg, \sigma_{q'})$.   Moreover, as in Section 2, the elements $x_i'[z'^{s_i}],  y'_i[z'^{-s_i}], h'_i+e_q^{-1}\langle x'_i, y'_i  \rangle s_i C$ in $\hat{L}(\fg, \sigma_{q'})$ are a set of Chevalley generators generating the non-completed Kac-Moody algebra $\tilde{L}(\fg, \sigma_{q'})\subset \hat{L}(\fg, \sigma_{q'})$, where $e_q:=|\Gamma_{q'}|$.   
 It is well-known that there is a natural bijection between the set of integrable highest weight representations of $\hat{L}(\fg, \sigma_{q'})$ and $\tilde{L}(\fg, \sigma_{q'})$.   
 
 We now introduce the following notation:
 \[ x''_i:= -y'_i, \, y''_i:=-x'_i, \text{ and } h''_i:=-h'_i , \]
 for any $i\in \hat{I}(\fg, \sigma_{q'})$.   Note that $\sigma_{q''}=(\sigma_{q'})^{-1}$.  We can identify $\hat{I}(\fg, \sigma_{q'})$ and $\hat{I}(\fg, \sigma_{q''})$, since $\fg^{\tau''}=\fg^{\tau'}$ where $\tau''=\tau'^{-1}$ is the diagram automorphism part of $\sigma_{q''}$.  
 
 Set $\alpha''_i=-\alpha'_i$ for any $i\in I(\fg^{\tau''})$, and $\theta''_0=-\theta'_0$.  We can choose $\alpha''_i, i\in I(\fg^{\tau''})$ as a set of simple roots for $\fg^{\tau''}$.  Then, $\theta''_0$ is the weight of $\fg^{\tau''}$ as in Section 2 with respect to this choice. Moreover,  $x''_i, y''_i, i\in I(\fg^{\tau''})$ is a set of Chevalley generators of $\fg^{\tau''}$, and $x''_0\in (\fg^{\tau''})_{-\theta''_0}, y''_0\in (\fg^{\tau''})_{\theta''_0}$ also satisfies the choice as in  \cite[\S 8.3]{Ka}.  
 We also note that
 \[  \sigma_{q''}(x''_i)=\epsilon^{s_i} x''_i, \, \text{ and } \sigma_{q''}(y''_i)=\epsilon^{-s_i} y''_i , \]
  for any $i\in \hat{I}(\fg, \sigma_{q'})$.  
As above,  we see that the elements $x''_i[ z''^{s_i}], y''_i[z''^{-s_i} ],  h''_i + |\Gamma_{q''}|^{-1} \langle x''_i, y''_i \rangle  s_i C$ as elements in $\hat{L}(\fg, \sigma_{q''})$ are Chevalley generators generating the non-completed Kac-Moody algebra $\tilde{L}(\fg, \sigma_{q''})\subset \hat{L}(\fg, \sigma_{q''})$.  
 Again, there is a natural bijection  between the set of integrable highest weight representations of $\hat{L}(\fg, \sigma_{q''})$ and $\tilde{L}(\fg, \sigma_{q''})$. 
 
 We now get  an isomorphism of Lie algebras:
 \[ \hat{\omega}:  \tilde{L}(\fg, \sigma_{q'})\simeq   \tilde{L}(\fg, \sigma_{q''})   \]
 given by 
 \[  x'_i[z'^{s_i} ] \mapsto x''_i[z''^{s_i} ], \, y'_i[z'^{-s_i} ]  \mapsto  y''_i[z''^{-s_i} ],  \]
 and 
 \[ h'_i+ e_q^{-1}\langle x'_i, y'_i  \rangle s_i C\mapsto  h''_i + e_q^{-1}\langle x''_i, y''_i \rangle s_i C ,\]  
 for any $i\in \hat{I}(\fg, \sigma_{q'})$.   Note that $\langle x'_i, y'_i \rangle =\langle x''_i, y''_i\rangle $ for any $i$. The map $\hat{\omega}$ is indeed an isomorphism, since these Chevalley generators correspond to the same vertices of the affine Dynkin diagram. 
 
Set  
\[ \tilde{L}(\fg, \sigma_{q'})^+= \hat{L}(\fg, \sigma_{q'})^+ \cap  \tilde{L}(\fg, \sigma_{q'}) ,\, \text{ and }   \tilde{L}(\fg, \sigma_{q'})^{\geq 0}= \hat{L}(\fg, \sigma_{q'})^{\geq 0}\cap  \tilde{L}(\fg, \sigma_{q'}) . \]
Similarly, we can introduce the Lie algebras $ \tilde{L}(\fg, \sigma_{q''})^+$ and  $ \tilde{L}(\fg, \sigma_{q''})^{\geq 0}$.
We can see easily that  
\[ \hat{\omega}(   \tilde{L}(\fg, \sigma_{q'})^+  ) =     \tilde{L}(\fg, \sigma_{q''})^+, \, \text{ and }  \hat{\omega}(\tilde{L}(\fg, \sigma_{q'})^{\geq 0}) =\tilde{L}(\fg, \sigma_{q''})^{\geq 0}.   \]
 
 Recall that $\fg^{\Gamma_q} \oplus  \mathbb{C} C$ is a Levi subalgebra of $\tilde{L}(\fg, \sigma_{q'})$, which is generated by  $x'_i, y'_i$ and $\fh^{\Gamma_q}\oplus \mathbb{C}C$ where $i\in \hat{I}(\fg, \sigma_{q'})^0$ consisting of $i\in \hat{I}(\fg, \sigma_{q'})$ such that $s_i=0$.  
Therefore, the isomorphism $\hat{\omega}$ also induces a Cartan involution $\omega$ on $\fg^{\Gamma_q}$, given on the Chevalley generators by 
   \[ e_i\mapsto -f_i, \, f_i\mapsto -e_i,  \, \text{ and } h \mapsto -h , \]
    for any $i\in \hat{I}(\fg, \sigma_{q'})^0$, and $h\in \fh^{\Gamma_q}$.   It is now clear that the isomorphism $\hat{\omega}$ induces a bijection $\kappa:  D_{c, q'}\simeq  D_{c, q''}$ given by $V\mapsto  V^*$.  This completes another proof of Lemma \ref{lemma 4.2}, part (2). \qed

\vskip1ex

Define the linear map
\begin{equation}
\label{Schur_map}
\hat{F}:\mathscr{H} (\vec{\lambda})\to \mathscr{H} (\vec{\lambda})\otimes \biggl(\bigoplus\limits_{\mu\in D_{c,p''}}\, V (\mu^*)  \otimes 
V(\mu)\biggr),\,\,\,
h\mapsto h\otimes \sum\limits_{\mu\in D_{c,p''}} I_{\mu},\,\,\,\text{for}\,\,h \in \mathscr{H}(\vec{\lambda}),
\end{equation}
where $I_{\mu}$ is the identity map thought of as an element of $V(\mu^*)\otimes V(\mu)\simeq \End_{\mathbb{C}}(V(\mu))$.

We  view $V(\mu^*)$ (resp. $V(\mu)$) as an irreducible representation of $\fg_{p'}$ (resp. $\fg_{p''}$) via the isomorphism $\varkappa'$ (resp. $\varkappa''$) defined above  Lemma \ref{lemma 4.2}.  Let $\mathscr{H}(\mu^*)$ (resp. $\mathscr{H}(\mu)$) denote the highest weight  integrable representation of $\hat{\fg}_{p'}$ (resp. $\hat{\fg}_{p''}$) associated  to $\mu^*$ (resp. $\mu$) of level $c$.  
Realize $\mathscr{H}(\vec{\lambda})\otimes   \mathscr{H}(\mu^*)   \otimes 
\mathscr{H}(\mu)$ (which contains $\mathscr{H} (\vec{\lambda})\otimes  V (\mu^*)  \otimes 
V(\mu)$) as a  $\mathfrak{g}[ {\Sigma}'^o  \backslash \pi'^{-1}\{ p',p'' \}]^\Gamma$-module  at the points 
$\vec{o}, p',p''$ respectively. Then, $I_\mu$ being a $\mathfrak{g}_{p}$-invariant, $
\hat{F}$ is a $\mathfrak{g}[{\Sigma}^o]^\Gamma$-module map, where we realize the range as a 
$\mathfrak{g}[{\Sigma}^o]^\Gamma$-module via the Lie algebra homomorphism ${\nu}^*$. Hence, $\hat{F}$  induces a linear map
$$F:\mathscr{V}_{{\Sigma},\Gamma,\phi}( \vec{o}, \vec{\lambda})\to \bigoplus\limits_{\mu\in D_{c, p''}}\mathscr{V}_{{\Sigma}' ,\Gamma,\phi}\left(\left( \vec{o},p',p''\right),\left( \vec{\lambda},\mu^{*},\mu\right)\right).
$$


The following theorem is the twisted analogue of the Factorization Theorem due to Tsuchiya-Ueno-Yamada [TUY]. 

\begin{theorem}\label{thm3.1.2}
With the setting as in Lemma \ref{normalization_pullback},  we further assume that $\Gamma$ stabilizes a Borel subalgebra $\mathfrak{b}$ and $\pi^{-1}(\vec{o})$ consists of smooth points of $\Sigma$. Then, 
 the map 
$$F:\mathscr{V}_{{\Sigma},\Gamma,\phi}( \vec{o}, \vec{\lambda})\to \bigoplus\limits_{\mu\in D_{c, p''}}\mathscr{V}_{{\Sigma}' ,\Gamma,\phi}\left(\left( \vec{o},p',p''\right),\left( \vec{\lambda},\mu^{*},\mu\right)\right)
$$
is  an isomorphism.

Dualizing the map $F$, we get an isomorphism
$$
F^{*}:\bigoplus_{\mu\in D_{c, p''}}\mathscr{V}^{\dagger}_{{\Sigma}' ,\Gamma,\phi }\left(\left( \vec{o},p',p''\right),\left( \vec{\lambda},\mu^{*},\mu\right)\right)\xrightarrow{\sim} \mathscr{V}^{\dagger}_{{\Sigma} ,\Gamma,\phi }( \vec{o}, \vec{\lambda}).
$$
\end{theorem}

\begin{proof}
As discussed above, the map $\hat{F}$ defined in (\ref{Schur_map}) is $\mathfrak{g}[ {\Sigma}^o ]^\Gamma$-equivariant. 
  By Propagation theorem (Theorem \ref{Propagation_thm}) at points $p'$ and $p''$, taking covariants on both sides of $\hat{F}$ with respect to the action of $\mathfrak{g}[ {\Sigma}^o]^\Gamma$ on the left side and with respect to the action of  $\mathfrak{g}[ {\Sigma'}^o]^\Gamma$ on the right side, we also obtain the map $F$.

We first prove the surjectivity of $F$. Fix a point $q\in \pi^{-1}(p)$,  we may view $V(\mu^*)$ and $V(\mu)$ as representations of the Lie algebra $\fg^{\Gamma_q}$ via the evaluation map ${\rm ev}_{q}: \fg_{p}\simeq \fg^{\Gamma_q}$ (cf. Lemma \ref{evaluation_lem}).  Correspondingly, we may view $D_{c,p''}$ as certain set of highest weights of $\fg^{\Gamma_q}$. Observe first that 
$\Gamma \cdot q'\cap \Gamma \cdot q'' = \emptyset$, since $\pi'$ is $\Gamma$-invariant and $\pi'(\Gamma \cdot q') = p'$ and $\pi'(\Gamma \cdot q'' )=p''$. Choose a function $f\in \mathbb{C}[\Sigma'^o]$ such that 
\begin{equation} \label{eqn4.3.1} f(q')=1 \,\,\,\text{and}\,\, f_{|\Gamma \cdot q'' \cup( \Gamma \cdot q'\setminus \{q'\})} =0.
\end{equation}
For any $x\in \fg^{\Gamma_q}$, let
\begin{equation} \label{eqn4.3.2} A(x[f]):= \frac{1}{|\Gamma_q|}\sum_{\gamma \in \Gamma}\, \gamma\cdot (x[f])\in \fg[\Sigma'^o]^\Gamma.
\end{equation}

  For any $h\in  \mathscr{H} (\vec{\lambda})$ and $ v\in \oplus_{\mu \in D_{c, p''}}\,\bigl(V(\mu)^*\otimes V(\mu)\bigr)$, as elements of 
 $$Q:= \mathscr{H} (\vec{\lambda})\otimes \biggl(\bigoplus\limits_{\mu\in D_{c,p''}}\, V (\mu^*)  \otimes 
V(\mu)\biggr),$$
we have the following equality (for any $x\in \fg^{\Gamma_q}$)
\begin{equation} \label{eqn3.1.2.3} A(x[f]) \cdot (h\otimes v)  - h\otimes (x\odot v)= (A(x[f])\cdot h)\otimes v,
\end{equation}
where the action $\odot$ of $\fg^{\Gamma_q}$ on $V(\mu^{*})\otimes V(\mu)$ is via its action on the first factor only. 
In particular, as elements of $Q$,
\begin{equation}  \label{eqn3.1.2.4} A(x[f]) \cdot (h\otimes \sum_{\mu \in D_{c, p''}}\, I_\mu)  - h\otimes \beta (x)= (A(x[f])\cdot h)\otimes \sum_{\mu \in D_{c, p''}}\, I_\mu,
 \end{equation}
where $\beta$ is the map defined by
\begin{equation}
\beta : U(\mathfrak{g}^{\Gamma_q})\to \bigoplus_{\mu\in D_{c, p'' }}\, V(\mu^{*})\otimes V(\mu),\quad \beta (a)=a\odot \sum\limits_{\mu}I_{\mu}.\label{thm3.1.2-eq1}
\end{equation}
Observe that $\Iim(\beta)$ is $\fg^{\Gamma_q} \oplus \fg^{\Gamma_q}$-stable under the component wise action of $\fg^{\Gamma_q} \oplus \fg^{\Gamma_q}$ on $V(\mu^{*})\otimes V(\mu)$ since $I_{\mu}$ is $\fg^{\Gamma_q}$-invariant under the diagonal action of $\fg^{\Gamma_q}$. Moreover,  $V(\mu^{*})\otimes V(\mu)$ is an irreducible $\fg^{\Gamma_q} \oplus \fg^{\Gamma_q}$-module with highest weight $(\mu^{*},\mu)$; and $\Iim(\beta)$ has a nonzero component in each $V(\mu^{*})\otimes V(\mu)$. Thus, $\beta$ is surjective.

From the surjectivity of $\beta$, we get that the map $F$ is surjective by combining the equation
\eqref{eqn3.1.2.4}  and the Propagation theorem (Theorem \ref{Propagation_thm}).

We next show that $F$ is injective. Equivalently, we show  that the dual map
$$
F^{*}:\bigoplus\limits_{\mu\in D_{c, p''}}\mathscr{V}^{\dagger}_{{\Sigma}' ,\Gamma,\phi }\left(( \vec{o},p',p''), ( \vec{\lambda},\mu^{*},\mu)\right)\to \mathscr{V}^{\dagger}_{{\Sigma},\Gamma,\phi}( \vec{o}, \vec{\lambda})
$$
is surjective.

From the definition of $\mathscr{V}^{\dagger}_{{\Sigma},\Gamma,\phi}$  and identifying the domain of $F^{*}$ via Theorem \ref{Propagation_thm}, we think of $F^{*}$ as the map
$$
F^{*}:\Hom_{\mathfrak{g}[{\Sigma}'^o]^\Gamma}\left(\mathscr{H} ( \vec{\lambda})\otimes \left(\oplus_{\mu\in D_{c, p'' }} V(\mu)^*\otimes V(\mu)\right),\mathbb{C}\right)\to \Hom_{\mathfrak{g}[{\Sigma}^o]^\Gamma}(\mathscr{H} ( \vec{\lambda}),\mathbb{C})
$$
induced from the inclusion
$$
\hat{F}:\quad  \mathscr{H} ( \vec{\lambda})\to \mathscr{H} ( \vec{\lambda})\otimes \left(\oplus_{\mu\in D_{c, p'' }}V(\mu)^*\otimes V(\mu)\right),
 h\mapsto h\otimes \sum\limits_{\mu\in D_{c, p''}}I_{\mu},\quad\text{for}\quad h\in \mathscr{H} ( \vec{\lambda}).$$

Let
$
\mathbb{C}_{p} [ {\Sigma}^o]\subset \mathbb{C}_{p''} [ {\Sigma}'^o]\subset \mathbb{C}[{\Sigma}'^o]
$
be the ideals of $\mathbb{C}[{\Sigma}'^o]$:
$$
\mathbb{C}_{p}[ {\Sigma}^o] :=\left\{f\in\mathbb{C}[ {\Sigma}^o]:f_{|\pi^{-1}(p)}=0\right\},
$$
and
$$
\mathbb{C}_{p'' }  [{\Sigma}'^o]=\{f\in \mathbb{C}[ {\Sigma}'^o]:f_{|\pi'^{-1}(p'')}=0\}.
$$
(Observe that, under the canonical inclusion $\mathbb{C}[ {\Sigma}^o]\subset \mathbb{C}[{\Sigma}'^o]$, $\mathbb{C}_{p} [ {\Sigma}^o] $ is an ideal of $\mathbb{C}[{\Sigma}'^o]$ consisting of those functions vanishing at $\pi'^{-1}\{p', p'' \}$.)
Now, define the Lie ideals of $\mathfrak{g}[ {\Sigma}'^o]^\Gamma$:
\begin{equation} \label{thm3.1.2-eq4}
 \mathfrak{g}_{p}[ {\Sigma}^o ]^\Gamma:= \left(\mathfrak{g}\otimes \mathbb{C}_{p} [ {\Sigma}^o]\right)^\Gamma,\,\,\text{and}\,\,\,
 \mathfrak{g}_{p''}[ {\Sigma}'^o]^\Gamma:= \left(\mathfrak{g}\otimes \mathbb{C}_{p''} [ {\Sigma}'^o]\right)^\Gamma.
\end{equation}
Define the linear map
$$
\fg^{\Gamma_q}   \to  \mathfrak{g}_{p''}[ {\Sigma}'^o]^\Gamma  \big/    \mathfrak{g}_{p}[ {\Sigma}^o]^\Gamma,\,\,\,
x\mapsto A(x[f])+ \mathfrak{g}_{p}[{\Sigma}^o]^\Gamma,
$$
where $x\in \fg^{\Gamma_q}$, 
 $f\in \mathbb{C}_{p''} [ {\Sigma'}^o]$ 
is any function satisfying~\eqref{eqn4.3.1} and $A(x[f])$ is defined by \eqref{eqn4.3.2}.   

Clearly, the above map is independent of the choice of $f$ satisfying ~\eqref{eqn4.3.1}. Moreover, it is a Lie algebra homomorphism:

For $x,y\in \fg^{\Gamma_q}$,
\begin{align*} 
\left[A(x[f]), A(y[f])\right]&=\frac{1}{|\Gamma_q|^2}\sum_{\gamma, \gamma'\in \Gamma}\, \left[\gamma\cdot (x[f]), \gamma'\cdot (y[f])\right]\\
&=\frac{1}{|\Gamma_q|^2}\sum_{\sigma\in \Gamma_q} \, \sum_{\gamma'\in \Gamma}\, \left[\gamma'\sigma\cdot (x[f]), \gamma'\cdot (y[f])\right]
+\frac{1}{|\Gamma_q|^2}\sum_{\gamma\notin \gamma'\Gamma_q}\,\sum_{\gamma'\in \Gamma}\, \left[\gamma\cdot (x[f]), \gamma'\cdot (y[f])\right]\\
&=\frac{1}{|\Gamma_q|}\sum_{ \gamma'\in \Gamma}\, \gamma'\cdot \left([x, y][f]\right)\,\,\text{mod}\,\, \fg_p[\Sigma^o]^\Gamma.
\end{align*}
To prove the last equality, observe that, for $\gamma\notin \gamma'\Gamma_q$, $ (\gamma \cdot f)\cdot  (\gamma' \cdot f)\in \mathbb{C}_p[\Sigma^o]$. Also, for $\sigma\in \Gamma_q$, $\sigma \cdot f-f\in  \mathbb{C}_p[\Sigma^o]$ and $f^2-f\in  \mathbb{C}_p[\Sigma^o]$.

Let
$$
\varphi : U(\fg^{\Gamma_q})\to U\left(\mathfrak{g}_{p''}[ {\Sigma}'^o]^\Gamma  \big/    \mathfrak{g}_{p}[ {\Sigma}^o]^\Gamma\right)
$$
be the induced homomorphism of the enveloping algebras.

To prove the surjectivity of $F^{*}$, take $\Phi\in \Hom_{\mathfrak{g}[{\Sigma}^o]^\Gamma}(\mathscr{H}( \vec{\lambda}),\mathbb{C})$ and define the linear map
$$
\tilde{\Phi}:\mathscr{H}( \vec{\lambda})\otimes \left(\bigoplus\limits_{\mu\in D_{c, p''}}V(\mu)^*\otimes V(\mu)\right)\to \mathbb{C}
$$
via
$$
\tilde{\Phi}(h\otimes \beta(a))=\Phi (\varphi(a^{t})\cdot h),\text{~ for~ } h\in \mathscr{H}( \vec{\lambda})\,\,\,\text{and}\,\, a\in U(\mathfrak{g}^{\Gamma_q}),
$$
where $t:U(\fg^{\Gamma_q}  )\to U(\fg^{\Gamma_q}  )$ is the anti-automorphism taking $x\mapsto  -x$ for $x\in\fg^{\Gamma_q} $, $\beta$ is the map defined by~\eqref{thm3.1.2-eq1} and $\varphi$ is defined above. (Observe that even though $\varphi(a)\cdot h$ is not well defined, but $\Phi(\varphi(a)\cdot h)$ is well defined, i.e., it does not depend upon the choice of the coset representatives in $\mathfrak{g}_{p''}[ {\Sigma}'^o]^\Gamma  \big/    \mathfrak{g}_{p}[ {\Sigma}^o]^\Gamma$.)

To show that $\tilde{\Phi}$ is well defined, we need to show that for any $a\in \Ker \beta$ and $h\in \mathscr{H}( \vec{\lambda})$,
\begin{equation}
\Phi (\varphi(a^{t})\cdot h)=0.\label{thm3.1.2-eq5}
\end{equation}

This will be proved in the next Lemma~\ref{lem3.1.4}.

We next show that $\tilde{\Phi}$ is a $\mathfrak{g}[{\Sigma}'^o]^\Gamma$-module map.  For any  element $X=\sum x_i[g_i]\in  \mathfrak{g}[{\Sigma}'^o]^\Gamma$ where $x_i\in \fg$ and $g_i\in \mathbb{C}[{\Sigma}'^o]$, we need to check that for any $h\in \mathscr{H}( \vec{\lambda})$ and $a\in U(\mathfrak{g}^{\Gamma_q} )$,   
\begin{equation}
\label{invariant_property}
 \tilde{\Phi}(X\cdot (h\otimes \beta(a)))=0.\end{equation}
Take any $\Gamma_q$-invariant $f'\in  \mathbb{C}[{\Sigma}'^o]$ (resp.  $f''\in  \mathbb{C}[{\Sigma}'^o]$) satisfying \eqref{eqn4.3.1} (resp.    $f''(q'')=1$
and $f''_{|\Gamma\cdot q'\cup (\Gamma\cdot q''\setminus \{q''\})}=0$). Then, 
\[   \mathbb{C}[{\Sigma'}^o] =  \mathbb{C}_p[{\Sigma}^o] + S_{f'} + S_{f''}, \,\,\,
 \text{where}\,\, S_{f'}:=
 \sum_{\gamma\in \Gamma/\Gamma_q}\mathbb{C} (\gamma\cdot f'), 
  S_{f''}:= \sum_{\gamma\in \Gamma/\Gamma_q}\mathbb{C} (\gamma\cdot f'').\]
Thus,
\[\fg[{\Sigma'}^o]^\Gamma =  \fg_p[{\Sigma}^o]^\Gamma + \left(\fg\otimes S_{f'}\right)^\Gamma +   \left(\fg\otimes S_{f''}\right)^\Gamma . \]
It suffices to prove the equation \eqref{invariant_property} in the following three cases of $X$:

{\it Case 1.}  $X\in  \fg_p[{\Sigma}^o]^\Gamma $: In this case 
\begin{align*}
\tilde{\Phi}\left(X\cdot (h\otimes \beta(a))\right) &= \tilde{\Phi} \left((X\cdot h)\otimes \beta(a)\right) +\tilde{\Phi}\left(h\otimes X\cdot\beta(a)\right) \\
&= \Phi \left(\varphi(a^{t})\cdot X\cdot h\right) ,\,\,\,\text{since} \,\, X\in  \fg_p[{\Sigma}^o]^\Gamma \\
&= \Phi \left(X\cdot \varphi(a^{t})\cdot h\right)+\Phi \left([\varphi(a^{t}), X]\cdot h\right)\\
&= 0,~~\text{since  $\Phi$ is a $\fg[{\Sigma}^o]^\Gamma$-module map and $ [\varphi(a^{t}), X]\in   \fg[{\Sigma}^o]^\Gamma$.}
\end{align*}

{\it Case 2.}  $X\in  \left(\fg\otimes S_{f'}\right)^\Gamma $: Write 
\[X=\sum_{\gamma\in \Gamma/\Gamma_q}\, x_\gamma[\gamma\cdot f'], \,\,\,\text{for some}\,\,  x_\gamma\in \fg.\]
Observe first  that since $\{\gamma\cdot f'\}_{\gamma\in \Gamma/\Gamma_q}$ are linearly independent, $x_1\in \fg^{\Gamma_q}$. Moreover, we claim that
\begin{equation} \label{neweqn4.3.1}X-\varphi(x_1)\in    \fg_p[{\Sigma}^o]^\Gamma, \,\,\,\text{i.e.,}\,\, X-\sum_{\gamma\in \Gamma/\Gamma_q}\, \gamma\cdot(x_1[ f'])
\in    \fg_p[{\Sigma}^o]^\Gamma.
\end{equation} 
To prove \eqref{neweqn4.3.1}, since $X$ and $\sum_{\gamma\in \Gamma/\Gamma_q}\gamma\cdot(x_1[ f'])$ both are $\Gamma$-invariant, it suffices to observe  that their difference vanishes both at $q'$ and $q''$. Now, 
\begin{align*}
\tilde{\Phi}\left(X\cdot (h\otimes \beta(a))\right) &= \tilde{\Phi} \left((X\cdot h)\otimes \beta(a)\right) +\tilde{\Phi}\left(h\otimes X\cdot\beta(a)\right) \\
&= \Phi \left(\varphi(a^{t})\cdot X\cdot h\right) - \Phi \left(\varphi(a^{t}x_1)\cdot h\right) \\
&= \Phi \left(\varphi(a^{t})(X-\varphi (x_1))\cdot h\right) \\
&= \Phi \left((X-\varphi (x_1))\varphi(a^{t}) \cdot h\right) + \Phi \left([\varphi(a^{t}), X-\varphi (x_1)]\cdot h\right)\\
&=0,
\end{align*}
by \eqref{neweqn4.3.1} and since $ \fg_p[{\Sigma}^o]^\Gamma $ is an ideal in  $\fg[{\Sigma'}^o]^\Gamma $.

{\it Case 3.}  $X\in  \left(\fg\otimes S_{f''}\right)^\Gamma $: Write 
\[X=\sum_{\gamma\in \Gamma/\Gamma_q}\, x_\gamma[\gamma\cdot f''], \,\,\,\text{for some}\,\,  x_\gamma\in \fg.\]
Same as in Case 2, we have $x_1\in \fg^{\Gamma_q}$. Moreover, we claim that
\begin{equation} \label{neweqn4.3.2}X+\varphi(x_1)\in    \fg[{\Sigma}^o]^\Gamma, \,\,\,\text{i.e.,}\,\, X+\sum_{\gamma\in \Gamma/\Gamma_q}\, \gamma\cdot(x_1[ f'])
\in    \fg[{\Sigma}^o]^\Gamma.
\end{equation} 
To prove \eqref{neweqn4.3.2}, it suffices to observe (from the $\Gamma$-invariance)  that  $X +\sum_{\gamma\in \Gamma/\Gamma_q}\gamma\cdot(x_1[ f'])$ takes the same value at both $q'$ and $q''$.  Now, 
\begin{align*}
\tilde{\Phi}\left(X\cdot (h\otimes \beta(a))\right) &= \tilde{\Phi} \left((X\cdot h)\otimes \beta(a)\right) +\tilde{\Phi}\left(h\otimes X\cdot\beta(a)\right) \\
 &= \tilde{\Phi} \left((X\cdot h)\otimes \beta(a)\right) +\tilde{\Phi}\left(h\otimes x_1\odot^r\beta(a)\right),\,\,\,\text{where $\odot^r$ denotes the action}\\
&  \text{\,\,\,\,\,\,\,\,\,\,\,\,\,\,\,\,\,\,\,\,\,\,\,\,\, of $\fg^{\Gamma_q}$ on $V(\mu^*)\otimes V(\mu)$ on the second factor only} \\
 &= \tilde{\Phi} \left((X\cdot h)\otimes \beta(a)\right) -\tilde{\Phi}\left(h\otimes \beta(ax_1)\right),\,\,\,\text{since the actions $\odot$ and  $\odot^r$ commute}\\
&  \text{\,\,\,\,\,\,\,\,\,\,\,\,\,\,\,\,\,\,\,\,\,\,\,\,\, and $I_\mu$ is a diagonal $\fg^{\Gamma_q}$-invariant} \\
&= \Phi \left(\varphi(a^{t})\cdot X\cdot h\right) + \Phi \left(\varphi(x_1)\varphi(a^{t})\cdot h\right)  \\
&= \Phi \left([\varphi(a^{t}), X]\cdot h\right) + \Phi \left((X+\varphi(x_1))\varphi(a^{t})\cdot h\right) \\
&=0,\,\,\,\text{since $[\varphi(a^{t}), X]\in  \fg[{\Sigma}^o]^\Gamma$ and using \eqref{neweqn4.3.2}}.
\end{align*}
This completes the proof of \eqref{invariant_property} and hence $\tilde{\Phi}$ is a $ \fg[{\Sigma'}^o]^\Gamma$-module map.

From the definition of $\tilde{\Phi}$, it is clear that
$
F^{*}(\tilde{\Phi})=\Phi.
$
This proves the surjectivity of $F^{*}$ (and hence the injectivity of $F$) modulo the next lemma. Thus, the theorem is proved (modulo the next lemma).
\end{proof}

\begin{definition}\label{defi3.1.3}
For any $\mu\in D$ (where $D$ is the set of dominant integral weights of $\fg^{\Gamma_q}$), consider the algebra homomorphism
$$
\overline{\beta}_{\mu}:U(\mathfrak{g}^{\Gamma_q})\to \End_{\mathbb{C}}(V(\mu)),
$$
defined by
$$
\overline{\beta}_{\mu}(a)(\bar{\nu})=a\cdot \bar{\nu},\text{~ for any~ }a\in U(\mathfrak{g}^{\Gamma_q})\text{~ and~ }\bar{\nu}\in V(\mu).
$$
\end{definition}

Let $K_{\mu}$ be the kernel of $\overline{\beta}_{\mu}$, which is a two sided ideal of $U(\mathfrak{g}^{\Gamma_q})$ (called a primitive ideal).
From the definition of $\beta$ (cf. equation \eqref{thm3.1.2-eq1}), it is easy to see that, under the identification of $V(\mu^*)\otimes V(\mu)$ with $\End_{\mathbb{C}}(V(\mu))$,
\begin{equation}
\beta(a)(\bar{\nu})=a^{t}\cdot \bar{\nu},\text{~ for any~ } a\in U(\mathfrak{g}^{\Gamma_q} )\text{~ and ~}\bar{\nu}\in V(\mu).\label{defi3.1.3-eq1}
\end{equation}
Thus,
\begin{equation}
\Ker \beta=\bigcap\limits_{\mu\in D_{c, p''}}K^{t}_{\mu}.\label{defi3.1.3-eq2}
\end{equation}
From the definition of $\overline{\beta}_{\mu}$, it follows immediately that for any left ideal $K\subset U(\mathfrak{g}^{\Gamma_q})$ such that $U(\mathfrak{g}^{\Gamma_q})/K$ is an integrable $\mathfrak{g}^{\Gamma_q}$-module, if the $\mathfrak{g}^{\Gamma_q}$-module $U(\mathfrak{g}^{\Gamma_q})/K$ has isotypic components of highest weights $\{\mu_{i}\}_{i\in \Lambda}\subset D$, then 
\begin{equation}\label{eqn4.2.3new}
K\supset \bigcap\limits_{i\in \Lambda}K_{\mu_{i}}.
\end{equation}
We are now ready to prove the following lemma.

\begin{lemma}\label{lem3.1.4}
With the notation as in the proof of Theorem~\ref{thm3.1.2} (cf. identity \eqref{thm3.1.2-eq5}), for any $a\in \Ker \beta$, $\Phi\in \Hom_{\mathfrak{g}[{\Sigma}^o ]^\Gamma}(\mathscr{H} ( \vec{\lambda}),\mathbb{C})$, and $h\in \mathscr{H} ( \vec{\lambda})$,
\begin{equation}
\Phi (\varphi (a^{t})\cdot h)=0.\label{lem3.1.4-eq1}
\end{equation}
\end{lemma}

\begin{proof}
Let $\fs_{p'}$ be the Lie algebra
$$
\fs_{p'}:= \left(\mathfrak{g}\otimes \mathbb{C}_{p''} [{\Sigma}'^o\backslash \pi'^{-1}(p') ] \right)^\Gamma \oplus 
\mathbb{C}\, C,
$$
where
$
 \mathbb{C}_{p''} [{\Sigma}'^o\backslash \pi'^{-1}(p') ]  \subset \mathbb{C}[{\Sigma}'^o \backslash \pi'^{-1}( p')]      
$
is the ideal consisting of functions vanishing at $\pi'^{-1}(p'')$, with the Lie bracket defined as in Formula (\ref{affine_curve_center_ext}) and $C$ is central in $\fs_{p'}$. 
 There is a Lie algebra embedding
\begin{equation}
\fs_{p'} \hookrightarrow {\hat{\mathfrak{g}}}_{p'}, \ \sum_ix_i[f_i]\mapsto \sum_ix_i[(f_i)_{p' }]\text{~~ and~~ }C\mapsto C.\label{lem3.1.4-eq3}
\end{equation}

Let $\mathscr{H} ( \vec{\lambda})^{*}$ be the full vector space dual of $\mathscr{H} ( \vec{\lambda})$.
The Lie algebra $\fs_{p'}$ acts on $\mathscr{H} ( \vec{\lambda})$ where $x[f]$ acts on $\mathscr{H} ( \vec{\lambda})$ as in (\ref{componentwise_action}) and the center $C$ acts by the scalar $-c$.  By the residue theorem, it is indeed a Lie algebra action.

This gives rise to the (dual) action of $\fs_{p'}$ on $\mathscr{H} ( \vec{\lambda})^{*}$. Let $M\subset \mathscr{H} ( \vec{\lambda})^{*}$ be the $\fs_{p'}$-submodule generated by $\Phi\in \mathscr{H}( \vec{\lambda})^{*}$.
We claim that the action of $\fs_{p'}$ on $M$ extends to a ${{\hat{\mathfrak{g}}}}_{p'}$-module structure on $M$ via the embedding~\eqref{lem3.1.4-eq3}. Let $\fs_{p'}^+\subset \fs_{p'}$ be the subalgebra $\mathfrak{g}_{p }[{\Sigma}^o ]^\Gamma$ defined by the equation \eqref{thm3.1.2-eq4} of Theorem \ref{thm3.1.2}. Then, by the definition of $\Phi$,
\begin{equation}
\fs_{p'}^+\cdot \Phi=0.\label{lem3.1.4-eq4}
\end{equation}
For any element $X=\sum_ix_i[f_i]\in \fs_{p'}$ with a basis $\{x_i\}$ of $\fg$ and $f_i\in  \mathbb{C}_{p''} [{\Sigma}'^o\backslash \pi'^{-1}(p') ]$, we define
\[o(X):= \text{max}_i\{o(f_i)\}, \]
where $o(f_{i})$ is the sum of orders of pole of $f_{i}$ at the points of $\pi'^{-1}(p' )$. (If $f_{i}$ is regular at a point in  $\pi'^{-1}(p' )$, we say that the order of pole at that point is $0$.)

Define an increasing filtration $\{\mathscr{F}_{d}(M)\}_{d\geq 0}$ of $M$ by
\begin{align*}
\mathscr{F}_{d}(M)&=\text{span of} \,\{\left(X_1\dots X_k\right)\cdot \Phi : \, X_i\in \fs_{p'} \text{~and~}\sum\limits^{k}_{i=1}o(X_{i})\leq d\}.
\end{align*}
From~\eqref{lem3.1.4-eq4}, it is easy to see that for any $\Psi\in \mathscr{F}_{d}(M)$, and any $Y=\sum x_i[g_i]\in \fg_{p}[\Sigma^o]^\Gamma$ such that each $g_i$ vanishes at every point of $\pi'^{-1}(p')$ of order at least  $d+1$,
\begin{equation}
Y\cdot \Psi=0.\label{lem3.1.4-eq5}
\end{equation}
Now,  for any  $y\in {\hat{\mathfrak{g}}}_{p'}$, pick $\hat{y} \in \fs_{p'}$ such that 
\begin{equation}
\hat{y}_{p'}-y\in \hat{\fg}_{p'}^{ d+1} ,\label{lem3.1.4-eq6}
\end{equation}
where $\hat{y}_{p'}$ denotes the restriction of $\hat{y}$ on $\pi'^{ -1}(\mathbb{D}_{p'})$ and  $\hat{\fg}_{p'}^{ d+1}  $ denotes elements of $\fg[\pi'^{ -1}(\mathbb{D}_{p'}) ]^\Gamma$ that vanish at each point of  $\pi'^{-1}(p')$ of order at least $d+1$ (note that $\hat{\fg}^1_{p'}=\hat{\fg}^+_{p'}$).
In fact, if $y\in \mathfrak{t}[\pi'^{ -1}(\mathbb{D}_{p'})]^\Gamma$ for some $\Gamma$-stable subspace $\mathfrak{t}$ of $\fg$, then we can take $\hat{y}\in  \left(\mathfrak{t}\otimes \mathbb{C}_{p''} [{\Sigma}'^o\backslash \pi'^{-1}(p') ] \right)^\Gamma$. 
 Define, for any $\Psi\in \mathscr{F}_{d}(M)$,
\begin{equation}
y\cdot \Psi: =\hat{y}\cdot \Psi,\,\text{ and } C\cdot \Psi :=c\Psi .  \label{lem3.1.4-eq7}\\
\end{equation}

From~\eqref{lem3.1.4-eq5}, it follows that~\eqref{lem3.1.4-eq7} gives a well defined action $y\cdot \Psi$ (i.e., it does not depend upon the choice of $\hat{y}$ satisfying~\eqref{lem3.1.4-eq6}). Observe that, taking $\hat{y}=0$,
\begin{equation}
y\cdot \Psi=0,\,\,\,\text{for}\,\, y\in \hat{\fg}_{p'}^{ d+1} .\label{lem3.1.4-eq9}
\end{equation}
Of course, the action of ${{\hat{\mathfrak{g}}}}_{p'}$ on $M$ defined by~\eqref{lem3.1.4-eq7} extends the action of $\fs_{p'}$ on $M$.

We next show that this action indeed makes $M$ into  a module for the Lie algebra ${{\hat{\mathfrak{g}}}}_{p'}$. To show this, it suffices to show that, for  $y_1, y_2\in \hat{\fg}_{p'}$ and $\Psi\in \mathscr{F}_{d}(M)$,
\begin{equation}
 y_1\cdot (y_2\cdot \Psi)-y_2\cdot (y_1\cdot \Psi)  
= [y_1, y_2]\cdot \Psi  \label{lem3.1.4-eq10}
\end{equation}
Take $\hat{y}_1, \hat{y}_2\in \fs_{p'}$ such that
$$
(\hat{y}_1)_{p'} - y_1 \,\,\,\text{and}\,\, (\hat{y}_2)_{p'} - y_2 \in \hat{\fg}_{p'}^{ d+1+o(y_1)+o(y_2)} ,
$$
where, for $y=\sum_i x_i[f_i]\in \hat{\fg}_{p'}$ ,  $o(y) :=\text{max}_i\{o(f_i)\}$, $o(f_i)$ being the sum of the orders of  poles at the points of  $\pi'^{-1}(p')$. Using the definition~\eqref{lem3.1.4-eq7} and observing that $o(\hat{y}_j)= o(y_j)$, it is easy to see that~\eqref{lem3.1.4-eq10} is equivalent to the same identity with $y_1$ replaced by $\hat{y}_1$ and $y_2$ by $\hat{y}_2$. The latter of course follows since $M$ is a representation of $\fs_{p'}$.
As a special case of~\eqref{lem3.1.4-eq9}, we get
\begin{equation}
{\hat{\mathfrak{g}}}^+_{p'}\cdot \Phi=0 .\label{lem3.1.4-eq11}
\end{equation}
We next show that $M$ is an integrable ${{\hat{\mathfrak{g}}}}_{p'}$-module. To prove this, it suffices to show that for any vector
 $y\in \left(\mathfrak{n}^\pm \otimes \mathbb{C}[\pi'^{-1}(\mathbb{D}^\times_{p'})]\right)^\Gamma$ ($\mathfrak{n}^+:=\mathfrak{n}$),  $y$  acts locally nilpotently on $M$, where $\fn$ (resp. $\fn^-$) is the nilradical of the Borel subalgebra $\fb$ (resp. of the opposite Borel subalgebra $\fb^-$) (cf. $\S$2). Since $M$ is generated by $\Phi$ as a ${{\hat{\mathfrak{g}}}}_{p'}$-module, by  [Ku, Lemma 1.3.3 and Corollary 1.3.4], it suffices to show that $y$ acts nilpotently on $\Phi$.

Choose $N_o >0$  such that 
\begin{equation} \label{neweqn100} (\text{ad}\,\fn)^{N_o}(\fg)=0,\,\,\,\text{and also}\,\, (\text{ad}\,\fn^-)^{N_o}(\fg)=0.
\end{equation}
 For any $y\in \left(\mathfrak{n}^\pm \otimes \mathbb{C}[\pi'^{-1}(\mathbb{D}^\times_{p'})]\right)^\Gamma \subset \hat{\fg}_{p'}$, pick $\hat{y}\in  \left(\mathfrak{n}^\pm\otimes \mathbb{C}_{p''} [{\Sigma}'^o\backslash \pi'^{-1}(p') ] \right)^\Gamma$ such that (cf. equation \eqref{lem3.1.4-eq6})
\begin{equation}\label{neweqn101}\hat{y}_{p'} -y \in \hat{\fg}_{p'}^{ o(y)(N_o-1)+1}.
\end{equation}
For any associative algebra $A$ and element $y\in A$, define the
operators $L_y (x)=yx$, $R_{y}(x)=xy$, and $\ad (y)=L_{y}-R_{y}$. 
Considering the operator
$R^{n}_{y}=(L_{y} - \ad (y))^{n}$ (for any $n\geq 1$) applied to $\hat{y}_{p'}-y$ in the algebra  $U(\hat{\fg}_{p'})$  and using the Binomial Theorem 
(since $L_y$ and $\ad (y)$ commute), we get
\begin{equation}\label{neweqn159}
(\hat{y}_{p'}-y)y^{n}=\sum\limits^{k}_{j=0}\binom{n}{j}(-1)^j y^{n-j}\left((\ad (y))^{j}(\hat{y}_{p'}-y)\right),
\end{equation}
where the summation runs only up to $k=\text{min} \{n, N_o-1\}$ because of the choice of $N_o$ satisfying \eqref{neweqn100}.
Then, for any $d\geq 1$, by induction on $d$ using ~\eqref{lem3.1.4-eq11} we get
\begin{equation}\label{neweqn54}
y^{d}\cdot \Phi=\hat{y}^{d}\cdot \Phi.
\end{equation}
To prove the above, observe that $(\hat{y}-y)y^d \cdot \Phi=0$ by the choice of $\hat{y}$ satisying \eqref{neweqn101} and the identities \eqref{neweqn159} and ~\eqref{lem3.1.4-eq11}.

For any positive integer $N$, let $\mathbb{C}_{p}[{\Sigma}^o]^{ N} \subset \mathbb{C}[{\Sigma}^o]$ be the ideal consisting of those $g\in \mathbb{C}[{\Sigma}^o]$ such that its pull-back to $\Sigma'^o$ via $\nu$  has a zero of order $\geq N$ at any point of $\pi'^{-1}(p')$. Let $\mathfrak{g}_{p}[{\Sigma}^o ]^{\Gamma,  N}\subset \mathfrak{g}[{\Sigma}^o ]^\Gamma$ be the Lie subalgebra defined as $\left[\mathfrak{g} \otimes \mathbb{C}_{p}[{\Sigma}^o]^{ N}   \right]^\Gamma$. By the same proof as that of Lemma \ref{lem2.1.3}, $\mathfrak{g}_{p}[{\Sigma}^o ]^{\Gamma,  N}\cdot \mathscr{H}( \vec{\lambda})$ is of finite codimension in $\mathscr{H}( \vec{\lambda})$.

  Let $V$ be a finite dimensional complement of $\mathfrak{g}_{p}[{\Sigma}^o ]^{\Gamma,o(y)(N_{o}-1)+1}\cdot \mathscr{H}( \vec{\lambda})$ in $\mathscr{H}( \vec{\lambda})$. Since $\hat{y}$ acts locally nilpotently on $\mathscr{H}( \vec{\lambda})$ (cf. Lemma \ref{lemma 1.3}) and $V$ is finite dimensional, there exists $N$ (which we take $\geq N_{o}$) such that
\begin{equation}
\hat{y}^{N}\cdot V=0.\label{lem3.1.4-eq12}
\end{equation}
Considering now the Binomial Theorem for the  operator
$L^{n}_{y}=(ad (y)+R_{y})^{n}$, we get (in any associative algebra)
\begin{equation*}
y^{n}x=\sum\limits^{n}_{j=0}\binom{n}{j}\left((ad (y))^{j}x\right)y^{n-j}.
\end{equation*}
Take any $\hat{z}\in \mathfrak{g}_{p}[{\Sigma}^o ]^{\Gamma,o(y)(N_{o}-1)+1}$. By the above identity  in the enveloping algebra\\ $U\left((\fg\otimes \mathbb{C}_{p''}[\Sigma'^o\setminus \pi'^{-1}(p')])^\Gamma \right)$, using the identity \eqref{neweqn100}, 
$$
\hat{y}^{N}\cdot \hat{z}=\sum\limits^{N_{o}-1}_{j=0}\binom{N}{j}\left((\ad \hat{y})^{j}(\hat{z})\right)\hat{y}^{N-j}.
$$
Thus,
\begin{equation}
\hat{y}^{N}\cdot \left(\mathfrak{g}_{p}[{\Sigma}^o ]^{\Gamma,o(y)(N_{o}-1)+1}\cdot \mathscr{H}( \vec{\lambda})\right)\subset \mathfrak{g}[\Sigma^o]^\Gamma\cdot \mathscr{H}( \vec{\lambda}).\label{lem3.1.4-eq13}
\end{equation}
Combining~\eqref{neweqn54} -~\eqref{lem3.1.4-eq13}, we get that
$$
y^{N}\cdot \Phi = \hat{y}^{N}\cdot \Phi =0 .
$$

This proves that $M\subset \mathscr{H}( \vec{\lambda})^{*}$ is an integrable ${{\hat{\mathfrak{g}}}}_{p'}$-module (generated by $\Phi$). Let $M_{o}\subset M$ be the $\mathfrak{g}_{p'}$-submodule generated by $\Phi$. Decompose $M_{o}$ into irreducible components:
$$
M_{o}=\bigoplus\limits_{\mu\in D}V(\mu)^{\oplus n_{\mu}},
$$
where $D$ is the set of dominant integral weights of $\fg_{p'}$.
Take any highest weight vector $v_{o}$ in any irreducible $\mathfrak{g}_{p'}$-submodule $V(\mu)$ of $M_{o}$. Since $
{\hat{\mathfrak{g}}}^+_{p'}$ annihilates $M_{o}$ (cf. equation ~\eqref{lem3.1.4-eq11}),
$v_{o}$ generates an integrable highest weight ${{\hat{\mathfrak{g}}}}_{p'}$-submodule of $M$ of highest weight $\mu$ with central charge $c$. In particular, any $V(\mu)$ appearing in $M_{o}$ satisfies $\mu\in D_{c, p'}$ (by the definition of $D_{c, p'}$), i.e., $\mu^*\in D_{c, p''}$ by Lemma \ref{lemma 4.2}.

 By the evaluation map ${\rm ev}_{q'}:\fg_{p'}\simeq \fg^{\Gamma_q}$,  we may view $M_o$ as a module over $\fg^{\Gamma_q}$ and each $V(\mu)$ the irreducible representation of $\fg^{\Gamma_q}$.  
Thus, from equation \eqref{eqn4.2.3new} of Definition \ref{defi3.1.3}, applied to the map
$$
U(\mathfrak{g}^{\Gamma_q})\to M_{o}, \quad a\mapsto a\cdot \Phi,
$$
we get that for any $a\in \Ker \beta$, $a \cdot \Phi=0$, i.e., $\Phi(\varphi(a^{t})\cdot h)=0$, for any $h\in \mathscr{H}( \vec{\lambda})$.
(Observe that $K_{\mu^*}= K_\mu^t$.) 
This proves the lemma and hence Theorem~\ref{thm3.1.2} is fully established.
\end{proof}
\vskip4ex

\section{Twisted Kac-Moody algebras and Sugawara construction over a base }\label{Sugawara_section}

We define twisted Kac-Moody Lie algebras, their Verma modules and   integrable highest weight modules with parameters
and prove the independence of parameters for the integrable highest weight  modules. We also prove that the Sugawara operators  acting on the integrable highest weight modules (of twisted affine Kac-Moody algebras) are independent of the parameters up to  scalars.

 Let $R$ be a commutative algebra over $\mathbb{C}$.  In this section, all commutative algebras are over $\mathbb{C}$, and we fix a root of unity $\epsilon=e^{\frac{2\pi i}{m}}$ of order $m$ and a central charge $c>0$. Also, as earlier, $\fg$ is a simple Lie algebra over $\bc$ and $\sigma$ is a Lie algebra automorphism such that $\sigma^m=\Id.$
\begin{definition}\label{defi5.1}
(a)  We say  that an $R$-algebra $\mathcal{O}_R$ is a {\it complete local  $R$-algebra} if there exists $t\in \mathcal{O}_R$ such that $\mathcal{O}_R\simeq R[[t]]$ as an $R$-algebra, where  $R[[t]]$ denotes the $R$-algebra of formal power series over $R$. We say such a  $t$ is an {\it $R$-parameter} of $\mathcal{O}_R$.  Let $\mathcal{K}_R$ be the $R$-algebra containing $\mathcal{O}_R$ by inverting $t$.  Thus,  $\mathcal{K}_R \simeq R((t))$. Note that $\mathcal{K}_R$ does not depend on the choice of the $R$-parameters.  
\vskip1ex
 
(b)  An {\it $R$-rotation}   of $\mathcal{O}_R$  of order $m$ is an $R$-algebra automorphism $\sigma$ of $\mathcal{O}_R$ (of order $m$) such that $\sigma(t)=\epsilon^{-1}t$ for some  $R$-parameter $t$.  Such an $R$-parameter $t$ is called a {\it $\sigma$-equivariant $R$-parameter}. Observe that any $R$-algebra automorphism of $\mathcal{O}_R$ of order $m$ may not be an $R$-rotation.
Clearly, an $R$-algebra automorphism of $\mathcal{O}_R$ extends uniquely as an automorphism of  $\mathcal{K}_R$, which we still denote by $\sigma$. 
  \end{definition}
    
  Given a pair $(\mathcal{O}_R, \sigma)$ of a complete local $R$-algebra  $\mathcal{O}_R$ and an $R$-rotation of order $m$, we can attach  an $R$-linear Kac-Moody algebra $\hat{L}(\fg, \sigma)_R$, 
\[  \hat{L}(\fg, \sigma)_R :=   (\fg\otimes_\mathbb{C}  \mathcal{K}_R  )^\sigma\oplus R \cdot C    ,\]
where  $C$ is a central element  of  $\hat{L}(\fg, \sigma)_R $, and for any $x[g],y[h]\in (\fg\otimes_\mathbb{C}  \mathcal{K}_R  )^\sigma$,
\begin{equation}
\label{Lie_bracket_R}
 [x[g], y[h] ]= [x,y][gh]+  \frac{1}{m} {\rm Res}_{t=0} \left((dg )h \right) \langle x, y\rangle C.  
 \end{equation}
Here the residue ${\rm Res} (dg )h$ is well-defined and independent of the choice of  $R$-parameters (cf. [H, Chap. III, Proof of Theorem 7.14.1]).  We denote by $ \hat{L}(\fg, \sigma)_R^{\geq 0}$ the $R$-Lie subalgebra $ (\fg\otimes_\mathbb{C}  \mathcal{O}_R  )^\sigma\oplus  R \cdot C 
$. 

 Given a complete local $R$-algebra $\mathcal{O}_R$, let $\mathfrak{m}_R$ denote  the ideal of $\mathcal{O}_R$ generated by a formal parameter $t$. Note that $\mathfrak{m}_R$ does not depend on the choice of $t$.  Then, $\mathcal{O}_R/\mathfrak{m}_R\simeq R$.  This allows us to give a natural map for an $R$-rotation $\sigma$ of $\mathcal{O}_R$ 
 $$\left(\fg\otimes_\bc \mathcal{O}_R\right)^\sigma \to (\fg \otimes R)^\sigma,$$ which is independent of the choice of the parameter $t$. 
   Given any morphism of commutative $\mathbb{C}$-algebras  $f:R\to R'$,  we define 
  \[ \mathcal{O}_{R} \hat{\otimes}_R R' := \varprojlim_{k} \left(( \mathcal{O}_R/{\mathfrak{m}_R^k} )\otimes_R R'\right) . \]
  Then, $\mathcal{O}_{R} \hat{\otimes}_R R' $ is a complete local $R'$-algebra. For any $R$-parameter $t\in \mathcal{O}_R$,  $t':=t\hat{\otimes} 1$ is an $R'$-parameter of  $\mathcal{O}_{R} \hat{\otimes}_R R' $. Let $\sigma$ be any $R$-rotation of $\mathcal{O}_R$ of order $m$. Then, it induces an $R'$-rotation of $\mathcal{O}_R\hat{\otimes}_R R'$. We still denote it by $\sigma$. 


\begin{lemma}
Let $\mathcal{O_R}$ be a complete local $R$-algebra with an $R$-rotation $\sigma$ of order $m$. Given any finite morphism of  commutative $\mathbb{C}$-algebras $f:  R\to R'$,  there exists a natural isomorphism of Lie algebras  $ \hat{L}(\fg, \sigma)_R\otimes_R R'  \simeq \hat{L}(\fg, \sigma)_{R'}$, where $\hat{L}(\fg, \sigma)_{R'} $ is the $R'$-Kac-Moody algebra attached to $\mathcal{O}_{R'}: =\mathcal{O}_R\hat{\otimes}_R R' $ and the induced rotation $\sigma$.
\end{lemma}
\begin{proof}
It suffices to check that $\mathcal{K}_R\otimes_R R'\simeq \mathcal{K}_{R'}$, which is well-known (since $f$ is a finite morphism).
\end{proof}

  Let $V$ be an irreducible representation of $\fg^\sigma$ with highest weight  $\lambda  \in D_c$, where   $D_{c}$ is defined  in Section \ref{Kac_Moody_Section}.  Then $V_R:=V\otimes_{\mathbb{C}} R$ is naturally a representation of $\fg^\sigma \otimes_\mathbb{C} R$.  Define the generalized Verma module 
 \[  \hat{M}(V, c)_R: = U_R(  \hat{L}(\fg, \sigma)_R )\otimes_{U_R(    \hat{L}(\fg, \sigma)_R^{\geq 0})    }  V_R   , \] 
 where $U_R(\cdot )$ denotes the universal enveloping algebra of  $R$-Lie algebra, and  $V_R$ is a module over $U_R(    \hat{L}(\fg, \sigma)_R^{\geq 0})  $ via the projection map $ \hat{L}(\fg, \sigma)_R^{\geq 0}\to (\fg^\sigma\otimes_{\mathbb{C}} R)\oplus  R \cdot C $ and such that $C $ acts on $V_R$ by $c$. 
\begin{lemma}
\label{Verma_pullback}
The Verma module $\hat{M}(V, c)_{R}$ is a free $R$-module.   Given any morphism $R\to R'$ of $\bc$-algebras, there exists a natural isomorphism $ \hat{M}(V, c)_R  \otimes_R R' \simeq   \hat{M}(V, c)_{R'}$ as $ \hat{L}(\fg,\sigma)_R \otimes R'$-modules, where $\hat{M}(V, c)_{R'}$ is the generalized Verma module attached to $\mathcal{O}_{R'}:= \mathcal{O}_R\hat{\otimes}_R R'$ and the action of $  \hat{L}(\fg,\sigma)_R\otimes_R R'$ on $\hat{M}(V, c)_{R'}$ is via the canonical morphism $ \hat{L}(\fg,\sigma)_R\otimes_R R' \to  \hat{L}(\fg,\sigma)_{R' } $. 
\end{lemma}
\begin{proof}
Let $t$ be a $\sigma$-equivariant $R$-parameter.  There exists a decomposition as $R$-module:
\[ ( \fg\otimes_\mathbb{C} R((t)))^\sigma = ( \fg\otimes_{\mathbb{C}}  R[[t]]  )^\sigma \oplus  (\fg\otimes_\mathbb{C} t^{-1} R[t^{-1}] )^\sigma  . \]

 Hence $\hat{M}(V,c)_{R}\simeq   {U}_R(  ( \fg\otimes  t^{-1}  R[t^{-1}])^\sigma  )\otimes_{\mathbb{C}}V$.   Note that $( \fg\otimes  t^{-1} R[t^{-1}])^\sigma$  is a  Lie algebra which is a free  module over $R$.  By Poincar\'e-Birkhoff-Witt theorem for any $R$-Lie algebra that is free as an  $R$-module  (cf. [CE, Theorem 3.1, Chapter XIII]),  
$\hat{M}(V,c)_{R}$ is a free $R$-module.

Note that there is a natural morphism $ \hat{L}(\fg, \sigma)_R\otimes_R R'   \to \hat{L}(\fg,\sigma)_{R'}$.  It induces a natural morphism  $\kappa: \hat{M}(V,c)_R\otimes R' \to \hat{M}(V,c)_{R'} $.  The map $\kappa$ is an isomorphism since it induces the following natural isomorphism
\[R' \otimes_R ( U_{R}( ( \fg\otimes_\mathbb{C} t^{-1} R [t^{-1}] )^\sigma )    \otimes_{\mathbb{C}}V ) \simeq  U_{R'}( ( \fg\otimes_\mathbb{C} t'^{-1} R' [t'^{-1}] )^\sigma )    \otimes_{\mathbb{C}}V ,\]
where $t'=t\hat{\otimes } 1$.
\end{proof}

 We can choose a $\sigma$-stable Borel subalgebra $\fb \subset \fg$,  a $\sigma$-stable Cartan subalgebra $\fh\subset \fb$, the elements  $\{x_i,y_i \}_{i\in  \hat{I}(\fg, \sigma)}$, and the set of non-negative  integers $\{ s_i\,|\,  i\in \hat{I}(\fg, \sigma) \}$ as in Section \ref{Kac_Moody_Section}, such that  $\hat{L}(\fg,\sigma)_R$ contains the elements  $x_i[t^{s_i}], y_i[t^{-s_i}]$.   Let  $V=V(\lambda)$ be the irreducible $\fg^\sigma$-module with highest weight  $\lambda\in D_c$  (cf. Lemma \ref{finite_set_weight_lem} for the description of $D_c$).    Let  $\hat{N}(V,c)_R$ be the $\hat{L}(\fg,\sigma)_R$-submodule of $\hat{M}(V, c)_R$ generated by $\{y_i[t^{-s_i}]^{n_{\lambda, i} +1 }\cdot v_\lambda\}_{i\in \hat{I}(\fg, \sigma)^+}$, where $v_\lambda$ is the highest weight vector of $V (\lambda)$, $n_{\lambda, i}$ is defined by the identity \eqref{integer_n_mu}  and (as in Section 2) $ \hat{I}(\fg, \sigma)^+:= \{
 i\in \hat{I}(\fg, \sigma): s_i> 0\}$.
\begin{lemma}
The module $\hat{N}(V,c)_R$ does not depend on the choice of the $\sigma$-equivariant  $R$-parameter $t$.
\end{lemma}
\begin{proof}
Let $t'$ be another $\sigma$-equivariant $R$-parameter.    It suffices to show that for each $i\in \hat{I}(\fg,\sigma)^+$,  $y_i[t'^{-s_i}]^{n_{\lambda, i} +1 }\cdot v_\lambda=c  y_i[t^{-s_i}]^{n_{\lambda, i} +1 }\cdot v_\lambda$ for some constant $c\in R^\times$ (where $R^\times$ denotes the set of units in $R$).  
By the  $\sigma$-equivariance of $t$ and $t'$, we can  write  $ t'^{-s_i}=c t^{-s_i}+  \sum_{k> -s_i, m|s_i+k}  a_k  t^{k} $, for some $c\in R^\times$ and $a_k\in R$. 
\vskip1ex

Case 1:   If $i\in \hat{I}(\fg, \sigma)^+$ and $0< s_i <  m $, then $ t'^{-s_i}=c t^{-s_i}+ g$, where  $g=\sum_{k> 0}  a_k t^{k} $ with $a_k\in R$  (since $0<s_i< m$ and $m|(s_i+k)$).  Since $y_i[g]\cdot v_\lambda = 0$, it is clear that $y_i[t'^{-s_i}]^{n_{\lambda, i}+1} \cdot v_\lambda=\left(c  y_i[t^{-s_i}]\right)^{n_{\lambda, i}+1}  \cdot v_\lambda$.
 
 \vskip1ex
  Case 2:    $s_o=m$, then $ t'^{-m}=c t^{-m}+ g$, where $g=\sum_{k\geq  0}  a_k t^{k}$ with $a_k\in R$.  Since $s_o=m$, by [Ka, Identity 8.5.6], each $s_j= 0$ for $j\neq o$ and $r=1$. Thus, the simple root vectors of $\fg^\sigma$ are $\{x_j\}_{j\neq o}$ with (simple) roots  $\{\alpha_j\}_{j\neq o}$. Since $y_o$ is a root vector of the root $\theta_0$ and $\theta_0$ is a  positive linear combination $\sum_{j \neq o}\, a_j\alpha_j$, $y_o$ is a positive root vector of $\fg^\sigma$. Hence,   $y_o\cdot v_\lambda=0$. Thus,  it follows that $y_o[g]\cdot v_\lambda=0$.  Hence, $y_o[t'^{-m}]^{n_{\lambda, o}+1} \cdot v_\lambda=\left(c  y_o[t^{-m}]\right)^{n_{\lambda, o}+1} \cdot v_\lambda$.
\vskip1ex

Case 3:   If $s_i=m$ for $i\neq o$,  then again by [Ka, Identity 8.5.6], $r=1$ and each $s_j=0$ for $j\neq i$. Thus, the simple root vectors of $\fg^\sigma$ are $\{x_j\}_{j\neq i}$ with (simple) roots  $\{\alpha_j\}_{j\notin \{i, o\}} \cup \{-\theta_0\}$. Hence, $-a_i\alpha_i= -\theta_0+\sum_{j\notin \{i, o\}} \,a_j\alpha_j$ giving that $-\alpha_i$ is a positive root of $\fg^\sigma$ (since $a_i, a_j>0$ being coefficients of the highest root written as a sum of simple roots) and hence $y_i \cdot v_\lambda = 0$. Rest of the argument is the same as in Case 2. This proves the lemma.
\end{proof}

We now define the following  $R$-linear representation of $\hat{L}(\fg, \sigma)_R$:
\[ \mathscr{H}(V)_R :=  \hat{M}(V, c)_R/ \hat{N}(V,c)_R. \]


\begin{lemma}
\label{integrable_rep_base_change}
(1)  $\hat{N}(V,c)_R$ and
$\mathscr{H}(V)_R$ are free over $R$. The module $\hat{N}(V,c)_R$ is a $R$-module direct summand of $ \hat{M}(V,c)_R$.
\vskip1ex

(2) 
For any morphism $f: R\to R'$ of commutative  $\mathbb{C}$-algebras,  there exists a natural isomorphism $ \hat{N}(V,c)_R\otimes_R R' \simeq  \hat{N}(V,c)_{R'}$ and $ \mathscr{H}(V)_R\otimes_R R'\simeq  \mathscr{H}(V)_{R'}$ as modules over $ \hat{L}(\fg,\sigma) _R\otimes_R R' $, where the action of $ \hat{L}(\fg,\sigma) _R\otimes_R R' $ on $ \hat{N}(V,c)_{R'}$ and $\mathscr{H}(V)_{R'}$ is via the canonical morphism   $ \hat{L}(\fg,\sigma) _R\otimes_R R' \to \hat{L}(\fg,\sigma) _{R'}$.

\vskip1ex

(3) Choose any $\sigma$-equivariant $R$-parameter $t$.  Then, $ \hat{N}(V,c)_R\subset \hat{M}(V,c)_R^+$. Moreover, for any other 
$ \hat{L}(\fg,\sigma) _R$-graded submodule  $A$ of $\hat{M}(V,c)_R$ such that $ A\cap V_R =(0)$, $A$ is contained in 
$\hat{N}(V,c)_R$. Here $\hat{M}(V,c)_R^+:= \oplus_{d\geq 1}\hat{M}(V,c)_R(d)$ and (for $d \geq 0$)
$$\hat{M}(V,c)_R (d):= \sum_{n_i\geq 0, \sum_in_i= d}\, X_1[t^{-n_1}] \cdots X_k[t^{-n_k}]\cdot V_R \subset \hat{M}(V,c)_R, \,\,\,\text{where $X_i[t^{-n_i}] \in \hat{L}(\fg, \sigma)_R$}.$$
Further, $A$ being graded means $A= \oplus_{d\geq 0} \, A\cap (\hat{M}(V,c)_R(d))$.

Observe that $\hat{M}(V,c)_R (d)$ does depend upon the choice of the parameter $t$. 

Hence,  $\hat{N}(V,c)_R$  and  $\mathscr{H}(V)_R$ do not depend on the choice of  $\fb,\fh$, and $x_i,y_i, i\in \hat{I}(\fg,\sigma)$.  
\end{lemma}
\begin{proof}
 Fix a $\sigma$-equivariant $R$-parameter $t$.   For each $i\in \hat{I}(\fg,\sigma)^+$,  the element $y_i[t^{-s_i}]^{n_{\lambda,i}+ 1} \cdot v_\lambda $ is a highest weight vector (cf. identity \eqref{eqn1.2.0}). 
  Hence,  
\begin{equation}\label{eqn5.5.1} \hat{N}(V,c)_R=\sum_{i\in \hat{I}(\fg,\sigma)^+  } U_R( ( \fg\otimes_{\mathbb{C}} R[t^{-1}]  )^\sigma  ) y_i[t^{-s_i}]^{n_{\lambda,i} +1} \cdot v_\lambda  . \end{equation}
Note that $U_R( ( \fg\otimes_{\mathbb{C}} R[t^{-1}]  )^\sigma  )\simeq U( ( \fg\otimes_{\mathbb{C}} \mathbb{C}[t^{-1}]  )^\sigma  ) \otimes_\mathbb{C}R $ as $R$-algebras.   Since $R$ is flat over $\mathbb{C}$, it is easy to see that (as a submodule of $\hat{M}(V,c)_R = \hat{M}(V,c)_\bc\otimes_\bc R$, where 
$\hat{M}(V,c)_\bc :=U( \hat{L}(\fg, \sigma)_\bc)\otimes_{U( \hat{L}(\fg, \sigma)_\bc^{\geq 0})} \, V$ is a $\bc$-lattice in 
$\hat{M}(V,c)_R$, which depends on the choice of $t$,  where $\hat{L}(\fg, \sigma)_\bc := \left[\fg\otimes_\bc \bc((t))\right]^\sigma \oplus \bc C$ and
$\hat{L}(\fg, \sigma)_\bc^{\geq 0} := \left[\fg\otimes_\bc \bc[[t]]\right]^\sigma \oplus \bc C$)
\begin{equation} \label{eqn5.5.2} \hat{N}(V,c)_R\simeq \hat{N}(V,c)_\bc\otimes_{\mathbb{C}} R,
 \end{equation}
 where $\hat{N}(V, c)_\bc:= (\sum_{i\in \hat{I}(\fg,\sigma)^+  } U( ( \fg\otimes_{\mathbb{C}} \mathbb{C}[t^{-1}]  )^\sigma  ) y_i[t^{-s_i}]^{n_{\lambda,i}+1} \cdot v_\lambda)$ is a $\bc$-lattice of $\hat{N}(V,c)_R$.  Hence,  $ \hat{N}(V,c)_R$ is free over $R$ and it is a direct summand (as an $R$-module) of $ \hat{M}(V,c)_R$. From this we readily see that
\begin{equation} \label{eqn5.5.3}  \mathscr{H}(V)_R\simeq    \mathscr{H}(V)_\bc  \otimes_\mathbb{C} R , 
\end{equation}
where $\mathscr{H}(V)_\bc :=\hat{M}(V,c)_\bc/\hat{N}(V,c)_\bc$ is a $\bc$-lattice in $\mathscr{H}(V)_R$ (depending on the choice of $t$), 
and hence it is also free over $R$.  It finishes the proof of part (1) of the lemma.

By the above equation \eqref{eqn5.5.2} and the assosiativity of the tensor product: $(M\otimes_R S)\otimes_S T\simeq M\otimes_R T$,  we also have $\hat{N}(V,c)_R\otimes_R R' \simeq \hat{N}(V,c)_{R'}$.  Similarly, by the above equation  $\eqref{eqn5.5.3},  \mathscr{H}(V)_R\otimes_R R' \simeq \mathscr{H}(V)_{R'}$. It concludes part (2) of the lemma.

We now proceed to prove  part (3) of the lemma.  By the equation \eqref{eqn5.5.1}, $\hat{N}(V,c)_R \subset \hat{M}(V,c)_R^+$. Observe first that 
\begin{equation} \label{eqn5.5.4} \left\{v\in \mathscr{H}(V)_\bc: X[t^n] \cdot v=0 \forall n>0 \,\,\text{and}\, X[t^n]\in \hat{L}(\fg, \sigma)_\bc\right\} =V.
\end{equation}
This is easy to see since  $\mathscr{H}(V)_\bc$ is an irreducible $\hat{L}(\fg, \sigma)_\bc$-module. Choosing a basis of $R$ over $\bc$, from this we easily conclude that 
 \begin{equation} \label{eqn5.5.5} \left\{v\in \mathscr{H}(V)_R: X[t^n] \cdot v=0 \forall n>0 \,\,\text{and}\, X[t^n]\in \hat{L}(\fg, \sigma)_\bc\right\} =V_R.
\end{equation}
For any nonzero $v\in A':=A/(A\cap \hat{N}(V,c)_R)\hookrightarrow \mathscr{H}(V)_R$, $v=\sum v_d$ with
$v_d \in \mathscr{H}(V)_R (d)$, set $|v|=\sum d: v_d\neq 0$, where the gradation $\mathscr{H}(V)_R (d)$ is induced from that of $ \hat{M}(V,c)_R$. Choose
a nonzero $v^o\in A'$ such that $|v^o|\leq |v|$ for all nonzero $v\in
A'$.  Then,
  \begin{equation} \label{eqn5.5.6}
X[t^n]\cdot v^o=0\quad\text{ for all $n\geq 1$ and
$X[t^n]\in \hat{L}(\fg, \sigma)_\bc$.}
  \end{equation}
For, otherwise, $|X[t^n]\cdot v^o|<|v^o|$, which contradicts
the choice of $v^o$. 

By the equation \eqref{eqn5.5.5}, we get that $v^o\in V_R$, which contradicts the choice of graded $A$ since $ A\cap V_R=(0)$. Thus, $A'=0$, i.e., $A\subset 
\hat{N}(V,c)_R.$ This proves the third part of the lemma.
\end{proof}

We now begin with the definition of  Sugawara operators $\{\,\Xi_n\,|\, n\in \mathbb{Z}\}$ for the Kac-Moody algebra $\hat{L}(\fg, \sigma)_R$ attached to a complete local $R$-algebra  $\mathcal{O}_R$ with an $R$-rotation of order $m$, and an automorphism $\sigma$ of $\fg$ such that $\sigma^m=\Id$.   We  $\bold{fix}$ a $\sigma$-equivariant $R$-parameter $t$. 

Recall the eigenspace decomposition $\fg=\bigoplus_{\underline{n}\in \mathbb{Z}/m\mathbb{Z} } \fg_{\underline{n}} $ of $\sigma$,
 where 
\[\fg_{\underline{n}}:=\{ x\in \fg \,| \, \sigma(x)= \epsilon^{n}  x \}. \,  \]
  Note that $\sigma$ preserves the normalized invariant form $\langle , \rangle$ on $\fg$, i.e., for any $x,y\in \fg$ we have $\langle\sigma(x), \sigma(y)\rangle=\langle x,y\rangle$.  For each $\underline{n}\in \mathbb{Z}/m\mathbb{Z}$ it induces a non-degenerate bilinear form $\langle, \rangle:\fg_{\underline{n}} \times  \fg_{-\underline{n}}\to \mathbb{C}$.  We choose a basis $\{u_a \,|\,  a \in A_{\underline{n}} \}$ of $\fg_{\underline{n}\,}$ indexed by a set $A_{\underline{n}}$.  Let $\{u^a \,|\, a\in A_{\underline{n}}  \}$ be the basis of $\fg_{-\underline{n}}$ 
 dual to the basis $\{u_a\,|\,  a\in A_{\underline{n}}\, \}$ of  $\fg_{\underline{n}}$.  

The normalized invariant $R$-bilinear form on $\hat{L}(\fg, \sigma)_R$ is given as follows (cf.\,\cite[Theorem 8.7]{Ka}), 
\[  \langle x[f], y[g]\rangle = \frac{1}{r} \left({\rm Res }_{t=0}\,  t^{-1}f(t)g(t)\right)  \langle x,y \rangle ,\,\, \langle x[f], C\rangle =0,  \text{ and } \langle C, C \rangle =0,    \]
where $x[f],y[g]\in \hat{L}(\fg, \sigma)_R$ and $r$ is the order of the diagram automorphism  associated to $\sigma$ . Then, the following relation is  satisfied: 
 $$\langle u_a[t^n],u^b[t^{-k}]\rangle=\frac{1}{r}\delta_{a,b}\delta_{n,k}\,\,\,\text{ for any $a\in A_{\underline{n}}$ and $b\in A_{\underline{k}}$}.$$
\begin{definition}\label{defi3.2.3}
 An  $R$-linear  $\hat{L}(\fg, \sigma)_R$-module  $M_R$  is called {\em smooth} if for any $v \in M_R$, there exists an integer $d$ (depending upon $v$) such that 
\begin{equation*}
x[f]\cdot v=0,\text{for all $ f  \in t^d \mathcal{O}_R$  and $x[f]\in(\fg\otimes_\mathbb{C}\mathcal{O}_R)^\sigma $}.\label{defi3.2.3-eq1}
\end{equation*}
Observe that this definition does not depend upon the choice of the parameter $t$. 
\end{definition}
The generalized Verma module $\hat{M}(V,c)_R$ (and  hence the  quotient module $\mathscr{H}(V)_R$) is clearly smooth. 

 We construct the following $R$-linear {\it  Sugawara operators} on any smooth representation $M_R$  of $\hat{L}(\fg, \sigma)_R$ of level $c\neq -\check{h}$
 (which depends on the choice of $t$),
\begin{equation}
\label{Virasoro_I}
 L^t_0 := \frac{1}{2(c+\check{h})} \left(\sum_{a\in A_{\underline{0}}}  u_{a}u^a + 2\sum_{ n>0 } \sum_{a\in A_{-\underline{n}}}  u_{a}[t^{-n}] u^a[t^n] +\frac{1}{2m^2}\sum_{n=0}^{m-1}\, n(m-n)\dim \fg_{\underline{n}} \right).
 \end{equation}
 \begin{equation}
 \label{Virasoro_II}
 L^t_k := \frac{1}{mk} [-t^{mk+1}\partial_t, L_0^t] =  \frac{1}{mk(c+\check{h})} \left(\sum_{ n>0 } n \sum_{a\in A_{-\underline{n}}}  \left(u_{a}[t^{-n+mk}] u^a[t^n] -u_a[t^{-n}]u^a[t^{n+mk}]\right)\right),
    \,\,\,\text{for $k\neq 0$},
\end{equation}
where   $\check{h}$ is the dual Coxeter number of $\fg$. Note that the smoothness ensures that $L^t_k$ is a well-defined operator on $M_R$ for each $k\in \mathbb{Z}$. Moreover,  it is easy to see that $L^t_{k}$ does not depend upon the choice of the basis $\{u_{a}\}$ of $\mathfrak{g}$.

The following result can be found in  \cite[\S3.4]{KW}, \cite{W}.

\begin{proposition}\label{prop3.2.2}
 {\it For any $n, k\in \mathbb{Z}$ and $x\in \mathfrak{g}_{\underline{n}}$, as  operators on a smooth representation $M_R$ of $\hat{L}(\fg, \sigma)_R$ of central charge $c\neq -\check{h}$,
 
 (a) 
$$
\left[x[t^n], L^t_{k}\right]=\frac{n}{m} \ x[t^{n+mk}].
$$
In particular, $L_0^t$ commutes with $\fg^\sigma$.

\vskip1ex
(b) 
$$[L^t_{n},L^t_{k}]=(n-k)L^t_{n+k}+\delta_{n,-k}\dfrac{n^{3}-n}{12}\dim \mathfrak{g} \frac{c}{c+\check{h}}.$$ }
\end{proposition}
Let us recall the definition of the {\em Virasoro algebra} $\Vir_R$ over $R$. It is the Lie algebra over $R$ with $R$-basis $\{d_{n};\bar{C}\}_{n\in \mathbb{Z}}$ and the commutation relation is given by
\begin{equation}
[d_{n},d_{k}]=(n-k)d_{n+k}+\delta_{n,-k}\dfrac{n^{3}-n}{12}\bar{C}; \,\,\, [d_{n},\bar{C}]=0.\label{defi3.2.3-add-eq1}
\end{equation}

An $R$-derivation of $\mathcal{K}_R$ is an $R$-linear map $\theta: \mathcal{K}_R\to \mathcal{K}_R$ such that $\theta(fg)=\theta(f)g+f\theta(g)$,  for any $f,g\in \mathcal{K}_R$.  
Let $\Theta_{\mathcal{K}_R/R }$ denote  the  Lie algebra of  all continuous $R$-derivations of $\mathcal{K}_R$, where we put the $\mathfrak{m}$-adic topology on 
 $\mathcal{K}_R$, i.e., $\{f+\mathfrak{m}^N\}_{N\in \bz, f\in \mathcal{K}_R}$ is a basis of open subsets. (Here $\mathfrak{m}^N$ denotes $t^N\mathcal{O}_R$, which does not depend upon the choice of $t$.)  With the choice of the $R$-parameter $t$ in $\mathcal{O}_R$, we have the equality 
$\Theta_{\mathcal{K}_R/R }=R((t))\partial_t$, where $\partial_t$ is the derivation on $\mathcal{K}_R$ such that $\partial_t(R)=0$ and $\partial_t(t)=1$.   Let $\Theta_{\mathcal{K}_R,R } $ denote the $\mathbb{C}$-Lie algebra of all continuous $\bc$-linear derivations $\theta$ of $\mathcal{K}_R$ that are  liftable from $R$, i.e.,  the restriction $\theta|_{R}$ is a $\bc$-linear derivation of $R$.  Let $\Theta_R$ denote the Lie algebra of 
$\bc$-linear derivations of $R$.  There exists a short exact sequence:
 \begin{equation}
0\to \Theta_{\mathcal{K}_R/R}\to \Theta_{\mathcal{K}_R,R } \xrightarrow{\rm Res}  \Theta_{R}\to 0,
\end{equation}
where $\rm Res$ denotes the restriction map of derivations from $\mathcal{K}_R$ to $R$.  See more details in \cite[Section 2]{L}.   It induces the folllowing short exact sequence:  
 \begin{equation}
0\to \Theta_{\mathcal{K}_R/R}^\sigma \to \Theta_{\mathcal{K}_R,R }^\sigma \xrightarrow{\rm Res}  \Theta_{R}\to 0,
\end{equation} 
where $ \Theta_{\mathcal{K}_R/R}^\sigma $ (resp. $\Theta_{\mathcal{K}_R,R }^\sigma$) is the space of $\sigma$-equivariant derivations in $ \Theta_{\mathcal{K}_R/R}$ (resp. $\Theta_{\mathcal{K}_R,R }$).
Then, $\Theta_{\mathcal{K}_R/R }^\sigma=R((t^m)) t\partial_t $.  
We define a central extension 
$$\widehat{\Theta_{\mathcal{K}_R/R}^\sigma}:= {\Theta_{\mathcal{K}_R/R}^\sigma} \oplus R \bar{C} $$ 
of the  $R$-Lie algebra $\Theta_{\mathcal{K}_R/R}^\sigma$ by
\begin{equation}
\label{completed_Virasoro}
\left[ f\partial_{t}, g\partial_{t} \right] =
\left(f\partial_{t }(g)-g\partial_{t}(f)\right)\partial_{t}
+{\Res}_{t=0} \left(t^{3m} A^3 (t^{-1} f ) t^{-1}g  t^{-1}dt \right)  \frac{\bar{C}}{12m},
\end{equation}
for $f\partial_{t}$, $g\partial_{t}\in   R((t^m)) t
\partial_{t}  $, where $A$ is the operator $t^{-m}(m+t\partial_t) $. Observe that this bracket corresponds to the bracket of the Virasoro algebra defined by the identity 
\eqref{defi3.2.3-add-eq1} if we take $d_{k}=-\frac{1}{m }t^{m k+1} \partial_{t }$ for any $k\in \mathbb{Z}$.  In this case $\bar{C}$ corresponds to $\bar{C}$. Therefore, $\widehat{\Theta_{\mathcal{K}_R/R}^\sigma}$ defines a completed version of the Virasoro algebra over $R$.

For any $\theta\in {\Theta_{\mathcal{K}_R/ R}^\sigma}$ with $\theta=\sum_{k\geq -N} a_{mk+1}  t^{mk+1}\partial_t$, we define a Sugawara operator associated to $\theta$
\begin{equation}
\label{General_Virasoro}
 L^t_{\theta}: =\sum_{k\geq -N} (-m a_{mk+1}) L^t_k  
 \end{equation}
 on  smooth modules  of $\hat{L}(\fg,\sigma)_R$ of central charge $c\neq  -  \check{h}$. 
 
 In the following lemma, the operator $ L^t_{\theta}$ is described more explicitly on any smooth module.
\begin{lemma}
For any $\theta\in R((t^m)) t\partial_t$,   the operator  $ L^t_{\theta}$ acts on any smooth module $M_R$ with central charge $c\neq   -\check{h}$ as follows:
\begin{equation}
\label{Virasoro_Verma_I}
 L^t_{\theta}(u_1[f_1]\cdots  u_n[f_n] \cdot v    )=   u_1[f_1]\cdots  u_n[f_n]\cdot L^t_\theta(v) + \sum_{i=1}^n  \left(u_1[f_1]\cdots u_i[\theta(f_i)]\cdots u_n[f_n]\cdot v 
 \right)   
 \end{equation}
 where $u_1[f_1],\ldots, u_n[f_n]\in \hat{L}(\fg,\sigma)_R$ and $v\in M_R$.
\end{lemma}
\begin{proof}
It is enough to show that $[L^t_\theta,  u_i[ f_i] ]=u_i[ \theta(f_i)] $ for each $i=1,\ldots, n$:
\begin{align}\label{eqn5.8.1}
[L^t_\theta,  u_i[ f_i] ] &=\sum_{k\geq -N} (-m a_{mk+1}) [L^t_k, u_i[ f_i] ]\notag\\
                                  &=\sum_{k\geq -N} (- a_{mk+1}) u_i[- t^{mk+1}\partial_t (f_i)  ] \notag\\
                                   &=u_i[\theta(f_i)],
\end{align}
where the first equality follows from the definition (\ref{General_Virasoro}), and the second equality follows from part (a) of Proposition \ref{prop3.2.2}.  
\end{proof}

Note that the choice of a $\sigma$-equivariant $R$-parameter $t$ gives the $R$-module  splitting $ \Theta_{\mathcal{K}_R,R }^\sigma= \Theta_{\mathcal{K}_R/R }^\sigma \oplus  \iota_t(\Theta_R)$, where (for $\delta \in \Theta_R$) $\iota_t(\delta) (f)=\sum_k \delta (a_k) t^k$ if $f=\sum_k  a_k t^k$. For any $f\partial_t\in \Theta_{\mathcal{K}_R/R }^\sigma $ and $\delta \in \Theta_R$,
\begin{equation}
\label{coefficient_wise_derivation}  [\iota_t(\delta), f\partial_t]= \iota_t( \delta) (f)\partial_t , \,\,[\iota_t(\delta_1), \iota_t(\delta_2)]=\iota_t [\delta_1, \delta_2], \,\,\text{and define}\,\,[\iota_t(\delta), r\bar{C}]=\delta(r)\bar{C} \,\,\text{for $r\in R$}. \end{equation}

The $\bc$-linear brackets (\ref{completed_Virasoro}) and (\ref{coefficient_wise_derivation}) define a completed extended Virasoro algebra $\widehat{\Theta_{\mathcal{K}_R, R}^\sigma}$  over $R$ (which is a $\bc$-Lie algebra), where 
\[  \widehat{\Theta_{\mathcal{K}_R, R}^\sigma}= \widehat{\Theta_{\mathcal{K}_R/R}^\sigma}    \oplus  \iota_t(\Theta_R).   \]   
Take  any smooth $\hat{L}(\fg, \sigma)_R$-module $M_R$ with $\bc$-lattice $M_\bc$ (i.e., $M_\bc \otimes R \simeq M_R$) stable under  $\hat{L}(\fg, \sigma)_\bc$. (Observe that $\hat{L}(\fg, \sigma)_\bc$ depends upon the choice of the parameter $t$.) Let $\delta$ act $\bc$-linearly on $M_R$ via its action only on the $R$-factor under the decomposition $M_R\simeq M_\bc\otimes R$. We denote this action on $M_R$ by $L^t_\delta$. Observe that $L^t_\delta$ depends upon the choice of the parameter $t$ 
as well as the choice of the $\bc$-lattice $M_\bc$ in $M_R$.
 
For any $\theta\in  \Theta_{\mathcal{K}_R, R}^\sigma $, write $\theta=\theta'+ \iota_t(\theta'') $ (uniquely), where $\theta'\in \Theta_{\mathcal{K}_R/ R}^\sigma $ and $\theta''\in \Theta_R$.   We define the extended Sugawara operator $L^t_\theta$ associated to $\theta$  acting on any smooth  $M_R :=  M_\bc\otimes_\bc R$ (with $\bc$-lattice $M_\bc$ as above) by
\begin{equation}  
\label{extended_Sugawara}
 L^t_\theta :=L^t_{\theta'} +  L^t_{\theta''}.  \end{equation}
Then,
\begin{equation}
\label{Virasoro_Verma_II}
L^t_{\theta''} (  u_1[f_1]\cdots  u_n[f_n] \cdot v    )=   \sum_{i=1}^n  u_1[f_1]\cdots u_i[ \iota_t(\theta'' )(f_i)]\cdots u_n[f_n]\cdot v  , 
 \end{equation}
for $v\in M_\bc$ and $ u_i[f_i] \in \hat{L}(\fg, \sigma)_R. $ From this we can easily deduce the more general formula when $v\in M_R$. 

The following proposition  follows easily from Proposition \ref{prop3.2.2} and the definition of the operator $L^t_\theta$.
\begin{proposition}\label{lem3.2.4}
{\it (1)
Let $M_R$ be a smooth module of $\hat{L}(\fg, \sigma)_R$ with $\bc$-lattice $M_\bc$ as above with respect to a $\sigma$-equivariant $R$-parameter $t$ and central charge  $c\neq -\check{h}$. Then, we have a $\bc$-Lie algebra homomorphism
$$\Psi: \widehat{\Theta_{\mathcal{K}_R,R}^\sigma}\to \End_\bc(M_R)$$
given by
\begin{equation}
r\bar{C}\mapsto r\left(\frac{c\dim \mathfrak{g}}{c+\check{h}}\right)I_{M_R}; \ \theta  \mapsto L^t_{\theta},\quad\text{for any }\quad \theta\in \Theta_{\mathcal{K}_R,R}^\sigma , r\in R.\label{lem3.2.4-eq1}
\end{equation}
Moreover, $\Psi$ is an $R$-module map under the $R$-module structure on $\End_\bc(M_R)$ given by 
$$(r\cdot f) (v)=r\cdot f(v),\,\,\,\text{for $r\in R, v\in M_R, f\in \End_\bc(M_R)$}.$$
Note that $\Theta_R$ is an $R$-module under $(r\cdot \delta)(s)=r\cdot \delta(s)$, for $r,s\in R$ and $\delta\in \Theta_R$.
\vskip1ex

(2) Further, for any $\theta\in \Theta_{\mathcal{K}_R,R}^\sigma$,  $v\in M_R $ and $a\in R$, 
\begin{equation}
\label{Leibnitz_derivation}
 L^t_{\theta}(a\cdot v)=\theta(a) \cdot v+ a \cdot L^t_{\theta}(v). \end{equation}}
\end{proposition}

The following lemma shows that the  representation of $\Theta_{\mathcal{K}_R,R}^\sigma$ on $\hat{M}(V,c)_R$ (and hence  on $\mathscr{H}(V)_R$) is independent of the choice of the $\sigma$-equivariant $R$-parameters up to a multiple of the identity operator. 
\begin{lemma}
\label{parameter_Sugawara}
Let $V=V(\lambda)$ be an irreducible $\fg^\sigma$-module with highest weight $\lambda\in D_c$. Let $t'$ be another $\sigma$-equivariant $R$-parameter in $\mathcal{O}_R$. For any $\theta\in \Theta_{\mathcal{K}_R,R}^\sigma$ there exists $b(\theta, \lambda, t, t')\in \bc$ such that 
 $L^t_{\theta}= L^{t'}_{\theta} +  b(\theta, \lambda, t, t') \Id$ on $\hat{M}(V,c)_R$ and hence on $\mathscr{H}(V)_R$.
 
 Here, with the choice of the parameter $t$, we have chosen the $\bc$-lattice $\hat{M}(V,c)_\bc$ of $\hat{M}(V,c)_R$ to be $U\left((\fg\otimes\bc((t)))^\sigma\oplus \bc C\right)\cdot V$ and the $\bc$-lattice of $\mathscr{H}(V)_R$ to be the image of  $\hat{M}(V,c)_\bc$ .  
 \end{lemma}
\begin{proof} Assume first that $\theta \in  \Theta_{\mathcal{K}_R/R}^\sigma$. Let  $L^t_\theta$ and $L^{t'}_{\theta}$ denote the Sugawara operators associated to $\theta$ with respect to the parameters $t$ and $t'$ respectively. 
For any $u[f]\in \hat{L}(\fg,\sigma)_R$,  from the identity \eqref{eqn5.8.1}, we have the following formula:
\begin{equation} \label{eqn5.10.1} [u[f],  L^t_\theta  ]= -u[\theta(f)], \,  [u[f],  L^{t'}_\theta  ]= -u[\theta(f)].  
\end{equation}
This gives
\begin{equation} \label{eqn5.10.2} [u[f],  L^t_\theta -L_\theta^{t'} ]= 0 \,\,\,\text{for any $u[f]\in \hat{L}(\fg, \sigma)_R$}.
\end{equation}
It  follows that  $L^t_\theta -L_\theta^{t'}$ commutes with the action of $\hat{L}(\fg, \sigma)_R$   (and hence also with the action of $L^t_0$
as in the identity \eqref{Virasoro_I}) on $\hat{M}(V,c)_R$.  In particular, using Proposition \ref{prop3.2.2}(a) for $k=0$,  $ (L^t_\theta- L^{t'}_\theta)v_o = \lambda v_o$ for some $\lambda \in \bc$, where $v_o$ is a highest weight vector of $V$. 
This shows that the  $R$-linear map $ L^t_\theta- L^{t'}_\theta = b(\theta, \lambda, t, t') \Id$ on the whole of $\hat{M}(V, c)_R$.
This proves the lemma in the case  $\theta \in  \Theta_{\mathcal{K}_R/R}^\sigma$.

We now prove the general case. 
Different choices of $\sigma$-equivariant $R$-parameters $t,t'$ give different splittings $\iota_t, \iota_{t'}$, 
\[ \Theta_{\mathcal{K}_R,R}^\sigma =R((t^m))t \partial_t \oplus   \iota_t(\Theta_R)= R((t'^m))t' \partial_{t'}\oplus  \iota_{t'}( \Theta_R).   \]
For any $\theta\in  \Theta_{\mathcal{K}_R,R}^\sigma$, we may write uniquely
\begin{equation}\label{eqn5.10.3} \theta=\theta'_t+ \iota_t(\theta''_t)= \theta'_{t'} + \iota_{t'}( \theta''_{t'} ), 
\end{equation}
where $\theta'_t\in R((t^m))t \partial_t $,  $\theta'_{t'}\in R((t'^m))t' \partial_{t'}$, and $\theta''_t, \theta''_{t'} \in \Theta_R$.  Observe that 
\begin{equation}\label{eqn5.10.4} \theta''_t= \theta''_{t'}= \theta_{|R}.
\end{equation}
Applying the equation \eqref{eqn5.10.3} to $t'$ , 
we get   
\[  \theta'_{t'}=\theta'_{t}+  \frac{\iota_t(\theta''_t)(u)   }{t\partial_t(u) +u} t\partial_t  , \]
where $u=t'/t\in \mathcal{O}_R^\times$. Note that $ \frac{\iota_t(\theta''_t)(u)   }{t\partial_t(u) +u} t\partial_t \in  R[[t^m]]t\partial_t$, and the constant  coefficient of $t\partial_t $ in $ \frac{\iota_t(\theta''_t)(u)   }{t\partial_t(u) +u} t\partial_t$ is $\frac{\theta_t''(u_0)}{u_0}$ where $u_0\in R^\times$ is the leading coefficient of $u=u_0+u_mt^m+\cdots \in R[[t^m]]$.

As proved above,  $L^t_{\theta'_{t'}}-L^{t'}_{\theta'_{t'}}$ is a scalar operator, as $\theta'_{t'}\in   \Theta_{\mathcal{K}_R/R}^\sigma $.   So, to prove that $L^t_\theta -L^{t'}_{\theta}$ is a scalar operator, it suffices to prove that 
$L^t_\beta +L^t_{\theta^{''}_t} - L^{t'}_{\theta^{''}_{t'}}$ is a scalar operator, since $L^t_\beta=L^t_{\theta'_t} - L^{t}_{\theta'_{t'}}$, where $\beta := \theta'_t-\theta'_{t'}$. Now, for $u_1[f_1], \ldots, u_n[f_n]\in \hat{L}(\fg, \sigma)_R$ and $v\in V$, by the identities \eqref{Virasoro_Verma_I} and 
 \eqref{Virasoro_Verma_II}, we get
 
 \begin{align*} &\left(L^t_\beta +L^t_{\theta^{''}_t} - L^{t'}_{\theta^{''}_{t'}} \right) (u_1[f_1]\dots u_n[f_n]\cdot v)\\
 &= u_1[f_1]\dots u_n[f_n]\cdot (L^t_\beta v) +\sum_{i=1}^n\,u_1[f_1]\dots u_i[\beta(f_i)+\iota_t(\theta^{''}_t)f_i-\iota_{t'}(\theta^{''}_{t'})f_i]\dots u_n[f_n]\cdot v \\
 &=u_1[f_1]\dots u_n[f_n]\cdot (L^t_\beta v) ,\,\,\, \text{since $  \beta(f_i)+\iota_t(\theta^{''}_t)f_i-\iota_{t'}(\theta^{''}_{t'})f_i =0$ by \eqref{eqn5.10.3} }\\
 &=  \frac{m\theta^{''}_t(u_0)}{u_0} u_1[f_1]\dots u_n[f_n]\cdot (L^t_0 v),  \,\,\, \text{since $L^t_k\cdot v=0$ for all $ k>0$ and $\beta \in R[[t^m]]t\partial_t$}\\
 &= \frac{m\theta^{''}_t(u_0)}{u_0} u_1[f_1]\dots u_n[f_n]\cdot dv, \,\,\, \text{for some constant $d\in \bc$},
 \end{align*}
 by the definition of $L^t_0$  since $\sum_{a\in A_{\underline{0}}}\,u_au^a$ is the Casimir operator of $\fg^\sigma$. This proves the lemma.
  \end{proof}


\section{Flat projective connection on sheaf of twisted covacua}
\label{Projective_Connection_section}

We define the sheaf of twisted covacua for  a family $\Sigma_T$ of $s$-pointed $\Gamma$-curves. We further show that this sheaf is locally free of finite rank for a smooth family $\Sigma_T$ over a smooth base $T$. In fact, we prove that it admits  a flat projective connection. 

{\it In this section, we take the parameter space  $T$ to be an irreducible  scheme over $\mathbb{C}$ and let $\Gamma$ be a finite group.} We fix a group homomorphism $\phi: \Gamma\to {\rm Aut}(\fg)$.
\begin{definition}\label{defi6.1}
A {\it family of curves over $T$}  is a proper and flat morphism $\xi: \Sigma_T\to T$ such that every geometric  fiber is a connected  reduced curve (but not necessary irreducible).  For any $b\in T$ the fiber $\xi^{-1}(b)$ is  denoted by $\Sigma_b$.
\end{definition}

Let $\Gamma$ act faithfully   on $\Sigma_T$ and  that $\xi$ is $\Gamma$-invariant (where $\Gamma$ acts trivially on $T$).  Let $\pi: \Sigma_T\to \Sigma_T/\Gamma= \bar{\Sigma}_T$ be the quotient map, and let 
$\bar{\xi}: \bar{\Sigma}_T \to T$ be the induced family of  curves over $T$. Observe that $\bar{\xi}$ is also  proper. 
For any section $p$ of $\bar{\xi}$,  denote by $\pi^{-1}(p)$ the set of sections $q$ of $\xi$ such that $\pi\circ q=p$.

\begin{definition}   \label{defi3.3.1}

{\it A family of $s$-pointed $\Gamma$-curves over} $T$ is a family of curves $\xi: \Sigma_T\to T$ over $T$ with an action of a finite group $\Gamma$ as above, and a collection of sections $\vec{q}:=(q_1,\cdots, q_s)$ of $\xi$, such that 
  \begin{enumerate}
\item  $p_1,\cdots ,p_s$ are mutually non-intersecting to each other, and, 
 for each $i$,  $\pi^{-1}(p_i(T))$ is contained in the smooth locus of $\xi$ and  $\pi^{-1}(p_i(T)) \to T$ is \'etale,  where $p_i=\pi\circ q_i$ is the section of $\bar{\xi}$;
\item  for any geometric point $b\in T$, $\left(\bar{\Sigma}_b,   p_1(b), \dots,  p_s(b)\right)$ is a $s$-pointed curve in the sense of Definition \ref{defi2.1.1}. Moreover, $\pi_b: \Sigma_b\to \bar{\Sigma}_b$ is a $\Gamma$-cover in the sense of Definition \ref{Gamma_curve}.

\end{enumerate}



\end{definition}

Let $\Sigma^o_T$ denote the open subset  $\Sigma_T\backslash  \cup_i^s  \pi^{-1}(p_i(T))$   of $\Sigma_T$. Let $\xi^o: \Sigma^o_T\to T$ denote the restriction of $\xi$.  
\begin{lemma}\label{lemma6.3}
The morphism $\xi^o: \Sigma^o_T\to T$ is affine.
\end{lemma}
\begin{proof}
See a proof by R. van Dobben de Bruyn on mathoverflow \cite{vDdB}.
\end{proof}

Let $f:T'\to T$ be a morphism of schemes. Then, we can pull-back the triple $({\Sigma}_{T},\Gamma, \vec{q} )$ to $T'$ to get a family $(f^*(\Sigma_T), \Gamma, f^*(\vec{q}))$ of pointed $\Gamma$-curves over $T'$, where $f^*(\Sigma_T)=T'\times_T \Sigma_T$, $f^*(\vec{q})=(f^*q_1,\ldots, f^*q_s)$.

\begin{lemma}
\label{Fiber_Of_Section}
 Let $\Gamma$ act on each geometric fiber $\Sigma_b$ of $\Sigma_T$ stably (cf. Definition \ref{stable_action}). Then, 
for any section $p$ of $\bar{\xi}$ such that $\pi^{-1}(p(T))$ is contained in the smooth locus of $\xi$, if $\pi^{-1}(p)$ is nonempty (where  $\pi^{-1}(p) = \{\text{sections $q$ of $\Sigma_T \to T$ such that $\pi \circ q =p$}\})$,  then 
\vskip1ex

(1)   For any $q\neq q' \in \pi^{-1}(p)$,  $q(T)$ and $q'(T)$ are disjoint.
\vskip1ex

(2)   $\Gamma$ acts on $\pi^{-1}(p)$ transitively,  and   the stabilizer group $\Gamma_q$ is equal to the stabilizer group $\Gamma_{q(b)}$ at the point $q(b)\in \Sigma_b$ for any geometric point $b\in T$. 
\end{lemma}
\begin{proof}
It is easy to see that  $\pi^{-1}(p)$ is finite (it also follows from the equation \eqref{eqn6.4.1}). Let $\pi^{-1}(p)=\{ q_1,q_2,\ldots, q_k   \}$.   For each $b\in T$,  $\{ q_1(b), q_2(b),\ldots, q_k(b) \}$ is a $\Gamma$-stable set and it is contained in the fiber $\pi^{-1}(p(b))$.  Since
 $\Gamma$ acts on $\pi^{-1}(p(b))$ transitively, it follows that 
\[   \pi^{-1}(p(b))=\{ q_1(b), q_2(b),\ldots, q_k(b) \}    . \]
From this it is easy to see that $\Gamma$ acts transitively on $\pi^{-1}(p)$.

Set $Z:=\xi(  \cup_{i\neq j}   q_i(T)\cap  q_j(T)   )$. Then,  $Z$ is a proper closed subset of $T$. Let $U$ be the open subset  $T \backslash Z$ of $T$.  
Then, $\{ q_1(U),q_2(U), \ldots, q_k(U)  \}$ are mutually disjoint to each other.   In particular, 

\begin{equation}\label{eqn6.4.1} k=\frac{|\Gamma| }{|\Gamma_{q_i(b)}| },\,\,\,\text{for any $b\in U$}. 
\end{equation}
By \cite[Lemma 4.2.1]{BR},  for each $1\leq i\leq k$, the order of the stabilizer group $\Gamma_{q_i(b)}$ is constant along $T$.  
For any $b' \in Z$,  there exists $i\neq j$ such that $q_i(b')=q_j(b')$. It follows that  $ \frac{|\Gamma| }{|\Gamma_{q_i(b')}| } = |\pi^{-1}(p(b'))|< k$, which is a contradiction.  Therefore, $T=U$, i.e. $\{q_1(T),q_2(T),\cdots, q_k(T)  \}$ 
are disjoint to each other.   It finishes the proof of  part $(1)$.

Let $\Gamma_q$ be the stabilizer group of $q\in \pi^{-1}(p)$.  It is clear that 
\begin{equation} \label{eqn6.4.2} \Gamma_q\subset \Gamma_{q(b)},\,\,\,\text{for any geometric point $b\in T$}.
\end{equation}
We have
$$k=\frac{|\Gamma| }{|\Gamma_{q(b)}| } \leq \frac{|\Gamma| }{|\Gamma_{q}| }\leq k,$$
where the first equality follows from the equation \eqref{eqn6.4.1} (since $U=T$) and the second inequality follows from the equation \eqref{eqn6.4.2}.
The third inequality follows since $k:=|\pi^{-1}(p)|$. Thus, we get $\Gamma_q=\Gamma_{q(b)}$, for any geometric point $b\in T$ by  the equation \eqref{eqn6.4.2}, and, moreover, $\Gamma$ acts transitively on $\pi^{-1}(p)$.
It concludes part $(2)$ of the lemma.
\end{proof}

\begin{definition}
(1)  A {\it formal disc  over} $T$ is a formal scheme $(T, \mathcal{O}_T )$  over $T$ (in the sense of \cite[Chap. II, \S9]{H}), where $\mathcal{O}_T$ is an $\mathscr{O}_T$-algebra which  has the following property:   For any point $b\in T$ there exists an affine open subset $U\subset T$ containing $b$  such that  $\mathcal{O}_T(U )$ is a complete local  $\mathscr{O}_T(U)$-algebra (see Definition \ref{defi5.1} (a)). 

  Let  $(T, \mathcal{K}_T )$ be the locally ringed space over $T$ defined so that $ \mathcal{K}_T(U)$ is the   $\mathscr{O}_T(U)$-algebra containing 
$\mathcal{O}_T(U)$ obtained by inverting a (and hence any)   $\mathscr{O}_T(U)$-parameter $t_U$ of $\mathcal{O}_T(U)$. Then,  $(T, \mathcal{K}_T )$
is called the {\it associated formal punctured disc over} $T$.
\vskip1ex

(2) A {\it  rotation of a formal disc} $(T, \mathcal{O}_T )$ over $T$ of order $m$ is an   $\mathscr{O}_T$-module automorphism $\sigma$ of $(T, \mathcal{O}_T)$ of order $m$ such that,  for any $b\in T$,   $\sigma(t_U)=\epsilon^{-1}t_U$ for some formal parameter $t_U$ around $b$, where $\epsilon:=e^{\frac{2\pi i}{m}}$.  
\end{definition}

\begin{lemma}\label{Rotation_Formal_Parameter}
With the assumption and notation as in Definition \ref{defi6.1}, let  $q$ be a  section of  $\xi: \Sigma_T\to T$ such that $q(T)$ is contained in the smooth locus of $\xi$. Then,

\vskip1ex
(1)  The formal scheme   $(T, \xi_* \hat{\mathscr{O } }_{ \Sigma_T,   q(T)  }  )$ is a formal disc over $T$, where  $\hat{\mathscr{O } }_{ \Sigma_T,   q(T)}$ denotes the formal completion of  $\Sigma_T$ along $q(T)$ (cf. \cite[Chap. II, $\S$9]{H}). 
\vskip1ex

(2) The stabilizer group $\Gamma_q$ is a cyclic group acting faithfully on the formal disc $(T, \xi_*( \hat{\mathscr{O } }_{ \Sigma_T,   q(T)  }  ))$. Further, $\Gamma_q$ has a generator $\tilde{\sigma}_q$ acting via rotation of $(T, \xi_*( \hat{\mathscr{O } }_{ \Sigma_T,   q(T)  }  ))$. 
Moreover, the action of  $\Gamma_q$ on  local $\tilde{\sigma}_q$-equivariant parameters is given by a primitive character $\chi$.
\end{lemma}
\begin{proof}
Part (1) follows from \cite[Corollaire 16.9.9, Th\'eor\`em 17.12.1(c')]{EGA}.  For part (2), choose a formal parameter $t_U\in  \left(\xi_* \hat{\mathscr{O } }_{ \Sigma_T,   q(T)  }\right) (U) $.   Let $\sigma_q\in \Gamma_q$ be a generator.  Set $\tilde{t}_U:=\frac{1}{|\Gamma_q|}\sum_{i=0}^{|\Gamma_q|-1} \epsilon_q^{- ij} \sigma_q^i(t_U)  $,  where $\epsilon_q :=e^{\frac{2\pi i}{ |\Gamma_q| }}$ and $\sigma_q(t_U)= \epsilon_q^jt_U+ $ higher terms. Then, $\tilde{t}_U$ is a  formal parameter in $\left(\xi_* \hat{\mathscr{O } }_{ \Sigma_T,   q(T)  }\right)  (U)$ such that 
\begin{equation} \label{eqref6.6.1} \sigma_q(\tilde{t}_U)=\epsilon_q^j\tilde{t}_U.
\end{equation}
 Since $\Gamma$ acts faithfully on $\Sigma_T$, the action of $\Gamma_q$ is faithful on the formal disc $(T, \xi_* \hat{\mathscr{O } }_{ \Sigma_T,   q(T)  }  )$. In particular, by equation \eqref{eqref6.6.1}, $\epsilon_q^j$ is a primitive $|\Gamma_q|$-th root of unity. Thus, we can find a generator 
$\tilde{\sigma}_q(U)\in \Gamma_q$ such that 
\begin{equation}\label{eqn6.6.2} \tilde{\sigma}_q(U)\tilde{t}_U=\epsilon_q^{-1}\tilde{t}_U.
\end{equation}
 In fact, $\tilde{\sigma}_q(U)$ is the unique generator of $\Gamma_q$ satisfying the above equation \eqref{eqn6.6.2} for any formal parameter $\tilde{t}_U$.    From this it is easy to see that  the generator $\tilde{\sigma}_q(U)$ does not depend upon $U$. We denote it by $\tilde{\sigma}_q$.  It determines the primitive character $\chi$ of $\Gamma_q$, which satisfies $\chi(\tilde{\sigma}_q)=\epsilon_q^{-1}$.
\end{proof}

Denote by $\mathcal{O}_q$  the sheaf of $\mathscr{O}_T$-algebra $\xi_* \hat{\mathscr{O } }_{ \Sigma_T,   q(T)  }  $ over $T$, and let $(T,\mathcal{K}_q)$ be the associated formal punctured disc over $T$.  
For any section $q$ of $\xi$ contained in the smooth locus of $\xi$, define the sheaf of Kac-Moody algebra $\hat{L}(\fg, \Gamma_q)_{T} $ over $T$ by 
\[\hat{L}(\fg,\Gamma_q)_T:= (\fg\otimes_\mathbb{C} \mathcal{K}_q  )^{\Gamma_q} \oplus  \mathscr{O}_T C,  \]
where the Lie bracket is defined as in (\ref{Lie_bracket_R}).  For any $\lambda \in D_{c,q}:=D_{c, \tilde{\sigma}_q}$  we  define  a sheaf of integrable representation $\mathscr{H}(\lambda)_T$ over $T$ as follows: For any open affine subset $U\subset T$ such that $\mathcal{O}_q(U)$ is a complete local $\mathscr{O}_T(U)$-algebra, 
\[U\mapsto   \mathscr{H}(\lambda)_{\mathscr{O}_T(U) } .\]
By Lemma \ref{integrable_rep_base_change},  this gives a well-defined sheaf over $T$. For each  section $p$ of $\bar{\xi}$ such that $\pi^{-1}(p)$ is non-empty and some (and hence any) $q\in \pi^{-1}(p)$ is contained in the smooth locus of $\xi$,  we may define the following sheaf of Lie algebras over $\mathscr{O}_T$ (cf. Definition \ref{Gamma_curve}), 
\[ \hat{\fg}_p:=(\oplus_{q\in \pi^{-1}(p) }  \fg\otimes_{\mathbb{C} } \mathcal{K}_q    )^\Gamma\oplus   \mathscr{O}_T \cdot C, \, \text{ and } \fg_p:=( \oplus _{q\in \pi^{-1}(p) }  \fg\otimes_{\mathbb{C}} \xi_* \mathscr{O}_{q(T  )} )^\Gamma  .   \]
The restriction gives an isomorphism $\hat{\fg}_p\simeq   \hat{L}(\fg, \Gamma_q)_T$, and $\fg_p\simeq  \fg^ {\Gamma_q}\otimes_\mathbb{C} \mathscr{O}_T$ as in Lemmas \ref{evaluation_lem} and  \ref{lemma2.3}.  For any $\lambda \in D_{c, q}$,  we still denote by $\mathscr{H}(\lambda)_T$ the associated representation of $\hat{\fg}_p$ via the isomorphism $\hat{\fg}_p\simeq   \hat{L}(\fg, \Gamma_q)_T$.



\begin{definition}[Sheaf  of twisted conformal blocks]\label{def_sheaf_conformal_block}
Let $({\Sigma}_{T},\Gamma, \vec{q})$ be a family of $s$-pointed $\Gamma$-curves over an irreducible  scheme $T$.  Set $\vec{p}=\pi\circ \vec{q}$.  
 Let $\vec{\lambda}=(\lambda_{1},\ldots, \lambda_{s})$ be a $s$-tuple of highest weights,  where $\lambda_i\in D_{c, q_i}$ for each $i$.  
 
Now, let us consider the sheaf of $\mathscr{O}_{T}$-module:
\begin{equation}
 \mathscr{H}(\vec{\lambda})_{T}:=   \mathscr{H}(\lambda_1)_T\otimes_{\mathscr{O}_T }  \mathscr{H}(\lambda_2)_T \otimes_{\mathscr{O}_T }\cdots  \otimes_{\mathscr{O}_T} \mathscr{H}(\lambda_s)_T ,     \quad\text{and}\label{defi3.3.2-eq1}
\end{equation}
\begin{equation}\hat{\mathfrak{g}}_{\vec{p}}:= \left(\oplus_{i=1}^s\left(\oplus_{\bar{q}_i\in \pi^{-1}(p_i)} \, \mathfrak{g}\otimes_{\mathbb{C}}\mathcal{K}_{\bar{q}_i} \right)^\Gamma \right)  \oplus \mathscr{O}_{T}C.    \label{defi3.3.2-eq2}
\end{equation}
We can define a $\mathscr{O}_{T}$-linear bracket in $\hat{\mathfrak{g}}_{\vec{p}}$ as in \eqref{pointwise-affine-Lie}; in particular,  $C$ is a central element  of $\hat{\mathfrak{g}}_{\vec{p}}$.   Then, $\hat{\mathfrak{g}}_{\vec{p}}$ is a sheaf of $\mathscr{O}_{T}$-Lie algebra.  There is a natural $\mathscr{O}_T$-linear Lie algebra homomorphism  
\[  \oplus_{i=1}^s  \hat{\fg}_{p_i}  \to  \hat{\mathfrak{g}}_{\vec{p}}, \,\,\,\text{where $C_i\mapsto C$}.   \]

The componentwise action of $ \oplus_{i=1}^s  \hat{\fg}_{p_i} $ on $\mathscr{H}(\vec{\lambda})_T$ induces an action of $ \hat{\mathfrak{g}}_{\vec{p}}$ on $\mathscr{H}(\vec{\lambda})_T$.  We also introduce the following $\mathscr{O}_{T}$-Lie algebra under the pointwise bracket:
\begin{equation}
\mathfrak{g}( \Sigma_T^o  )^\Gamma:=[\mathfrak{g}\otimes_{\mathbb{C}} \xi^o_{*}  \mathscr{O}_{{{\Sigma}}^o_{T}}  ]^\Gamma, \, \text{ where }   \Sigma_T^o=\Sigma_T\backslash( \cup_{i=1}^s   \pi^{-1}(p_i(T)) )  .\label{defi3.3.2-eq4}
\end{equation}
There is an embedding of sheaves of $\mathscr{O}_{T}$-Lie algebras:
\begin{equation}
\beta : \mathfrak{g}( \Sigma_T^o  )^\Gamma  \hookrightarrow \hat{\mathfrak{g}}_{\vec{p}}, \,\,\, \sum_k x_k [ f_k] \mapsto \sum_{q\in \pi^{-1} (\vec{p})}  \sum_k x_k [ (f_k)_{q}]  ,
\label{defi3.3.2-eq6}
\end{equation}
for  $x_k\in \mathfrak{g}$ and $f_k\in \xi^o_{*}  \mathscr{O}_{{{\Sigma}}^o_{T}}$  such that $\sum_kx_k[f_k]\in \fg[{\Sigma}^o_{T}]^\Gamma $ ,
where $(f_k)_{q}$ denotes the image of $f_k$ in $\mathcal{K}_q$ via the localization map $\xi^o_{*}  \mathscr{O}_{{{\Sigma}}^o_{T}} \to \mathcal{K}_q$.

By the Residue Theorem, $\beta$ is indeed a Lie algebra embedding. (Observe that Lemma \ref{lemma6.3} has been used to show that $\beta$ is an embedding.)

Finally, define the {\em sheaf of twisted covacua} (also called the {\it sheaf of twisted dual conformal blocks}) $\mathscr{V}_{\Sigma_T, \Gamma, \phi}(\vec{q},\vec{\lambda})$ over $T$ as the quotient sheaf of $\mathscr{O}_{T}$-modules
\begin{equation}
\mathscr{V}_{\Sigma_T, \Gamma, \phi}(\vec{q},\vec{\lambda}):  =\mathscr{H}(\vec{\lambda})_{T}\Big/  \mathfrak{g}( \Sigma_T^o  )^\Gamma \cdot \mathscr{H}(\vec{\lambda})_{T},\label{defi3.3.2-eq7}
\end{equation}
where $ \mathfrak{g}( \Sigma_T^o  )^\Gamma$ acts on $\mathscr{H}(\vec{\lambda})_{T}$ via the embedding $\beta$ (given by \eqref{defi3.3.2-eq6}) and 
$
\mathfrak{g}( \Sigma_T^o  )^\Gamma\cdot \mathscr{H}(\vec{\lambda})_{T}\subset \mathscr{H}(\vec{\lambda})_{T}
$
denotes the image sheaf under the sheaf homomorphism 
\begin{equation}
\alpha_T:  \mathfrak{g}( \Sigma_T^o  )^\Gamma\otimes_{\mathscr{O}_{T}}\mathscr{H}(\vec{\lambda})_{T}\to \mathscr{H}(\vec{\lambda})_{T}\label{defi3.3.2-eq8}
\end{equation}
induced from the action of $ \mathfrak{g}( \Sigma_T^o  )^\Gamma$ on $\mathscr{H}(\vec{\lambda})_{T}$.

Here we use the notation $\mathscr{V}_{\Sigma_T, \Gamma, \phi}(\vec{q},\vec{\lambda})$ to denote the sheaf of twisted covacua (see Remark \ref{notation_warning}).
\end{definition}

\begin{theorem}
\label{covacua_coherent_basechange}
(1) The sheaf  $\mathscr{V}_{\Sigma_T, \Gamma, \phi} (\vec{q}, \vec{\lambda} ) $ is a coherent $\mathscr{O}_T$-module.
\vskip1ex

(2) For any morphism $f: T'\to T$  between schemes, there exists a natural isomorphism
\[  \mathscr{O}_{T'}\otimes_{ \mathscr{O}_{T}}\,( \mathscr{V}_{\Sigma_T, \Gamma, \phi} (\vec{q}, \vec{\lambda} )  )\simeq  \mathscr{V}_{ f^*(\Sigma_T), \Gamma, \phi} (f^*( \vec{q}), \vec{\lambda} ) .  \]
In particular,  for any point $b\in  T$ the restriction $\mathscr{V}_{\Sigma_T, \Gamma, \phi} (\vec{q}, \vec{\lambda} )|_b $ is the space of twisted dual conformal blocks attached to $(\Sigma_b, \Gamma, \phi, \vec{p}(b), \vec{\lambda})$.   
\end{theorem}
\begin{proof}

We first prove part (1).   Recall the embedding $\beta: \mathfrak{g}( \Sigma_T^o  )^\Gamma \hookrightarrow  \hat{\fg}_{\vec{p}}$ of $\mathscr{O}_{T}$-Lie algebras from \eqref{defi3.3.2-eq6}. Also, consider the $\mathscr{O}_T$-Lie subalgebra
$$
{\hat{\mathfrak{p}}}_{\vec{p}}:=
\left[\oplus_{q\in \pi^{-1} (\vec{p}) }  \mathfrak{g}\otimes_{\mathbb{C}}   \xi_* \hat{\mathscr{O } }_{ \Sigma_T,   q(T)  }     \right]^\Gamma \oplus \mathscr{O}_{T}C
$$
of $ \hat{\fg}_{\vec{p}}$ and let $ \mathfrak{g}( \Sigma_T^o  )^\Gamma  +{\hat{\mathfrak{p}}}_{\vec{p}}$ be the $\mathscr{O}_T$-subsheaf  of $ \hat{\fg}_{\vec{p}}$  spanned by $\text{Im\,}\beta$ and ${\hat{\mathfrak{p}}}_{\vec{p}}$. Then, as can be seen, the quotient sheaf $\hat{\fg}_{\vec{p}}  \big/ \left( \mathfrak{g}( \Sigma_T^o  )^\Gamma  +{\hat{\mathfrak{p}}}_{\vec{p}}   \right)$ is a coherent $\mathscr{O}_{T}$-module (cf. [L, Lemma 5.1]).
Thus, locally  we can find a finite set of elements $\{x_{j}\}$ of $ \hat{\fg}_{\vec{p}}$ such that each $x_{j}$ acts locally finitely on $\mathscr{H}(\vec{\lambda})_T$ and 
$$
\hat{\fg}_{\vec{p}}=\mathfrak{g}( \Sigma_T^o  )^\Gamma    +    {\hat{\mathfrak{p}}}_{\vec{p}} +   \sum\limits_{j}\mathscr{O}_{T}x_{j}
$$
(cf. [Ku, Proof of Lemma 10.2.2]). Now, following the proof of Lemma \ref{lem2.1.3} and recalling that the Poincar\'e-Birkhoff-Witt theorem holds for any Lie algebra $\mathfrak{s}$ over a commutative ring $R$ such that $\mathfrak{s}$ is free as an $R$-module (cf. [CE, Theorem 3.1, Chapter XIII]), we get part (1) of the theorem.

We now prove part (2).  By the definition of the sheaf of covacua,  $\mathscr{V}_{\Sigma_T, \Gamma, \phi}(\vec{q},\vec{\lambda})$ is the cokernel of the $\mathscr{O}_T$-morphism $\alpha_T:  \mathfrak{g}( \Sigma_T^o  )^\Gamma\otimes_{\mathscr{O}_{T}}\mathscr{H}(\vec{\lambda})_{T}\to \mathscr{H}(\vec{\lambda})_{T} $, which gives rise to the exact sequence (on tensoring with $\mathscr{O}_{T'}$):  

\begin{equation*}
\vcenter{
\xymatrix{
\mathscr{O}_{T'}\bigotimes\limits_{\mathscr{O}_{T}}\left({\mathfrak{g}}(\Sigma^o_{T})^\Gamma\bigotimes\limits_{\mathscr{O}_{T}}\mathscr{H}(\vec{\lambda})_{T}\right)\ar@{}[d]_{\rotatebox{90}{\hole{$\backsimeq$}}} \ar[r]^-{\text{Id\,}\otimes \alpha_T} &\mathscr{O}_{T'}\bigotimes\limits_{\mathscr{O}_{T}}\, \mathscr{H}({\vec{\lambda})_{T}}\ar[d]^{\rotatebox{90}{$\backsim$}} \ar[r] & \mathscr{O}_{T'}\bigotimes\limits_{\mathscr{O}_{T}}\mathscr{V}_{\Sigma_{T}, \Gamma, \phi}(\vec{q},\vec{\lambda})  \ar@{}[d]_{\rotatebox{90}{\hole{$\backsimeq$}}}  \ar[r] & 0\\ 
{\mathfrak{g}}(\Sigma^o_{T'})^\Gamma\bigotimes\limits_{\mathscr{O}_{T'}}\mathscr{H}(\vec{\lambda})_{T'}\ar[r] & \mathscr{H}(\vec{\lambda})_{T'} \ar[r] & \mathscr{V}_{\Sigma_{T'}, \Gamma, \phi}(f^*\vec{q},\vec{\lambda}) \ar[r]& 0 ,
}}
\end{equation*}
where we have identified the bottom left term of the above under

$$
\left(\mathscr{O}_{T'}\bigotimes\limits_{\mathscr{O}_{T}}{\mathfrak{g}}(\Sigma^o_{T})^\Gamma\right)\bigotimes\limits_{\mathscr{O}_{T'}}\left(\mathscr{O}_{T'}\bigotimes\limits_{\mathscr{O}_{T}}\mathscr{H}(\vec{\lambda})_{T}\right)\simeq {\mathfrak{g}}(\Sigma^o_{T'})^\Gamma\bigotimes\limits_{\mathscr{O}_{T'}}\mathscr{H}(\vec{\lambda})_{T'}$$
and the second vertical isomorphism is obtained by Lemma \ref{integrable_rep_base_change}. The right most vertical isomorphism follows from the Five Lemma, proving the second part of the lemma.

\end{proof}

{\it In the rest of this section we assume that the family $\xi:\Sigma_T\to T$ of $s$-pointed $\Gamma$-curves is such that  $T$ is a $\bold{smooth}$ and irreducible scheme over $\mathbb{C}$ and $\xi: \Sigma_T\to T$ is a $\bold{smooth}$ morphism. }  In particular, $\Sigma_T$ is a smooth scheme. 

  Let $\Theta_T$ be the sheaf of vector fields on $T$.  Let $\Theta_{\Sigma^o_T/T}$ denote the $\mathscr{O}_T$-module of  vertical vector fields on $\Sigma^o_T$ with respect to $\xi^o$, and let $\Theta_{\Sigma^o_T,T}$ denote the $\mathscr{O}_T$-module of vector fields $V$ on $\Sigma^o_T$ that locally  descend to vector fields on $T$ (i.e., there exists an open cover $U_i$ of $T$   such that $(d\xi^o) (V_{|{\xi^o}^{-1}(U_i)})$ is a vector field on $U_i$).   Since $\xi^o: \Sigma^o_T\to T$ is an affine and smooth morphism, there exists a short exact sequence of  $\mathscr{O}_T$-modules:
 \begin{equation}
0\to \Theta_{{{\Sigma}}^o_{T}/T}\to \Theta_{{{\Sigma}}^o_{T},T}\xrightarrow{d{\xi}^{o}}    \Theta_{T}\to 0. \label{thm3.4.1-eq2}
\end{equation}
This short exact sequence induces the following short exact sequence of $\mathscr{O}_T$-modules: 
 \begin{equation}
 \label{ses_vf}
0\to \Theta_{{{\Sigma}}^o_{T}/T}^\Gamma \to \Theta_{{{\Sigma}}^o_{T},T}^\Gamma  \xrightarrow{d{\xi}^{o}}    \Theta_{T}\to 0,  
\end{equation}
where $\Theta_{\Sigma^o_T/T}^\Gamma $ (resp. $\Theta_{\Sigma^o_T,T}^\Gamma$) denotes  the $\mathscr{O}_T$-submodule of $\Gamma$-invariant vector fields in $\Theta_{{{\Sigma}}^o_{T}/T}$ (resp. $\Theta_{\Sigma^o_T,T}$).

For any $b\in T$, we can find an affine open subset $b\in U\subset T$, and  a $s$-tuple of formal parameters $\vec{t}:=(t_1,t_2,\cdots, t_s)$ where $t_i$ is a formal $\Gamma_{q_i}$-equivariant $\mathscr{O}_T(U)$-parameter around $q_i$ (cf. Lemma \ref{Rotation_Formal_Parameter}). For any $\theta\in  \Theta_{\Sigma^o_T, T}^\Gamma(U) :=\Theta_{\Sigma^o_T, T}(U)^\Gamma$, we denote by $\theta_i$ the image in $ \Theta_{\mathcal{K}_{q_i},  T }(U) ^{\Gamma_{q_i}} $, where $ \Theta_{\mathcal{K}_{q_i},  T }(U)$ is the space of continuous $\bc$-linear derivations  of $\mathcal{K}_{q_i}$ under the $\mathfrak{m}$-adic topology (given below Proposition \ref{prop3.2.2}) that are liftable from the vector fields on $U$.     We define the operator $L^{\vec{t}}_{\theta}$ on $\mathscr{H}(\vec{\lambda})_U$ by 
\begin{equation}
\label{Connection_opeartor_definition}
L^{\vec{t}}_\theta( h_1\otimes \cdots \otimes h_s):= \sum_i h_1\otimes  \cdots  \otimes L^{t_i}_{\theta_i}\cdot h_i\otimes \cdots \otimes h_s,   
\end{equation}
where, for $1\leq i \leq s$,  $L^{t_i}_{\theta_i}$ is the extended Sugawara operator associated to $\theta_i$ with respect to the Kac-Moody algebra $\hat{L}(\fg, \Gamma_{q_i})_{U}$  defined by (\ref{extended_Sugawara}), where we choose the $\bc$-lattice in $\mathscr{H}(\lambda_i)_U$ as in Lemma \ref{parameter_Sugawara}.
   
   \begin{lemma}  For any $\theta \in \Theta_{\Sigma^o_T,T} (U)^\Gamma  $,   the operator $L^{\vec{t}}_{ \theta}$ preserves $\fg(\Sigma^o_U)^\Gamma\cdot   \mathscr{H}(\vec{\lambda})_U$, where   $\Sigma^o_U:= {\xi^o}^{-1}(U).$
   \end{lemma}
   \begin{proof}
   For any $x[f]\in \fg(\Sigma^o_U)$,  let  $A(x[f])$ denote the average $\sum_{\sigma\in \Gamma} \sigma(x)[\sigma(f)]\in \fg(\Sigma^o_U)^\Gamma $. 
   For any $\vec{h} =h_1\otimes \cdots \otimes h_s\in \mathscr{H}(\vec{\lambda})_{U}$, and $\theta\in \Theta_{\Sigma^o_T, T} (U)^\Gamma$, by the formulae \eqref{Virasoro_Verma_I} and \eqref{Virasoro_Verma_II}, one can easily check that
   \[ L^{\vec{t}}_\theta ( A(x[f])\cdot \vec{h} )=  A(x[ \theta(f)])(\vec{h})+ A(x[f])(L^{\vec{t}}_\theta\cdot \vec{h}) .   \]
It follows thus  that  $L^{\vec{t}}_{ \theta}$ preserves $\fg(\Sigma^o_U)^\Gamma\cdot   \mathscr{H}(\vec{\lambda})_U$.
  
   \end{proof}
From the above lemma,   the operator $L^{\vec{t}}_{ \theta}$ induces an operator denoted $\nabla^{\vec{t}}_\theta  $ on $\mathscr{V}_{\Sigma_T, \Gamma,\phi}(\vec{q}, \vec{\lambda} )|_U$.
\begin{theorem}
\label{Locally_free_smooth_base}
With the same notation and assumptions as in Theorem \ref{covacua_coherent_basechange}, assume, in addition,  that $\xi:\Sigma_T\to T $ is a smooth morphism and $T$ is smooth. Then,  $\mathscr{V}_{\Sigma_T, \Gamma, \phi}(\vec{q} ,\vec{\lambda})$ is a locally free $\mathscr{O}_T$-module of finite rank.
\end{theorem}
\begin{proof}
It is enough to show that  the space of twisted covacua $\mathscr{V}_{\Sigma_T, \Gamma,\phi}(\vec{q}, \vec{\lambda} )|_U$ is locally free for any cover of affine open subsets $U\subset T$ with a $s$-tuple of formal parameters $\vec{t}:=(t_1,\cdots, t_s)$ around $\vec{q}:=(q_1,\cdots, q_s)$,  where   $t_i$ is a $\Gamma_{q_i}$-equivariant $\mathscr{O}_T(U)$-parameter.   From the short exact sequence (\ref{ses_vf}), we may assume (by shrinking $U$ if necessary) that there exists a $\mathscr{O}_T(U)$-linear section $a: \Theta_T(U)\to \Theta_{\Sigma^o_T, T}(U)^\Gamma$ of $d\xi^o|_U: \Theta_{\Sigma^o_T, T}(U)^\Gamma \to \Theta_T(U)$.  By part (2) of Proposition \ref{lem3.2.4}, the following map   
\[  \theta\mapsto  \nabla^{\vec{t} }_{a(\theta)}:\mathscr{V}_{\Sigma_T, \Gamma,\phi}(\vec{q}, \vec{\lambda} )|_U \to \mathscr{V}_{\Sigma_T, \Gamma,\phi}(\vec{q}, \vec{\lambda} )|_U\]
defines a connection on $\mathscr{V}_{\Sigma_T, \Gamma,\phi}(\vec{q}, \vec{\lambda} )|_U$.  Thus, by the same proof as in \cite[Theorem 1.4.10]{HTT}, $\mathscr{V}_{\Sigma_T, \Gamma,\phi}(\vec{q}, \vec{\lambda} )|_U$ is locally free. By Theorem \ref{covacua_coherent_basechange},  $\mathscr{V}_{\Sigma_T, \Gamma,\phi}(\vec{q}, \vec{\lambda} )$ is a coherent $\mathscr{O}_T$-module and hence it is of finite rank. 
\end{proof}

Let $\mathscr{V}$ be a locally free $\mathscr{O}_T$-module  of finite rank.  Let $\mathcal{D}_1(\mathscr{V})$ denote the   $\mathscr{O}_T$-module of operators $P: \mathscr{V}\to  \mathscr{V}$ such that for any  $f\in \mathscr{O}_T$, the Lie bracket $[P, f]$ is an $\mathscr{O}_T$-module morphism from $\mathscr{V}$ to $\mathscr{V}$. Clearly,  $\mathcal{D}_1(\mathscr{V})$ is a $\bc$-Lie algebra. 
\begin{definition}[\cite{L}]
A {\it flat projective connection} over $\mathscr{V}$  is a sheaf of $\mathscr{O}_T$-modules $\mathscr{L}\subset \mathcal{D}_1(\mathscr{V})$ containing $\mathscr{O}_T$ (where $\mathscr{O}_T$ acts on $\mathscr{V}$ by multiplication) such that $\mathscr{L}$ is a $\bc$-Lie subalgebra and 
\begin{equation}\label{seq110}
\xymatrix{
0\ar[r] &  \mathscr{O}_T  \ar[r]^{i}  & \mathscr{L}   \ar[r]^{\rm Symb}  &  \Theta_T\ar[r] & 0 \\ 
}
\end{equation}
is a short exact sequence, where $\rm Symb$ denotes the symbol map defined by
$$({\rm Symb} \,P) (f)=[P, f] \,\,\,\text{for $P\in \mathscr{L}$ and $f\in \mathscr{O}_T$}. $$
Observe that ${\rm Symb}$ is a Lie algebra homomorphism. 
\end{definition}

Following this definition,  choose a local section $\nabla:  \Theta_U\to \mathscr{L}|_U $ of $\rm Symb$ on some open subset $U\subset T$.  Then, $\nabla$ defines a connection on $\mathscr{V}$ over $U$, since for any $X\in \Theta_U, f\in \mathscr{O}_U$ and $v\in \mathscr{V}$, 
\[ \nabla_X(f\cdot v)=f\nabla_Xv+X(f)\cdot v . \]
For any $X,Y\in \Theta_U$, the curvature $\mathscr{K}(X,Y):= [\nabla_{X},\nabla_Y]-\nabla_{[X,Y]}\in \mathscr{O}_U$.

We now construct a flat projective connection $\mathscr{L}_T$ on the sheaf of covacua $\mathscr{V}_{ \Sigma_T,\Gamma, \phi }(\vec{q}, \vec{\lambda} )$.
Let $U$ be any affine open subset of $T$ with a $s$-tuple of parameters $\vec{t}$ around $\vec{q}$ as above.   Define $\mathscr{L}_T(U)$ to be the $\mathscr{O}_T(U)$-module spanned by  $\{ \nabla^{\vec{t}}_\theta  \,|\,  \theta\in   \Theta_{\Sigma^o_T,T}(U)^\Gamma  \} $ and $\mathscr{O}_T(U)$.   
By Lemma \ref{parameter_Sugawara},   $\mathscr{L}_T(U)$  does not depend on the choice of parameters $\vec{t}$.  Therefore, 
the assignment $U\mapsto \mathscr{L}_T(U)$ glues to be a sheaf $\mathscr{L}_T$ over $T$.  

\begin{theorem}\label{thm3.4.1}
With the same notation and assumptions as in Theorem \ref{Locally_free_smooth_base}, assume further that the ramification locus in each geometric fiber $\Sigma_b$ of $\Sigma_T$ is contained in $\Gamma\cdot \vec{q}(b)$. Then, 
the sheaf $\mathscr{L}_T$ of operators on $\mathscr{V}_{\Sigma_T, \Gamma, \phi}(\vec{q}, \vec{\lambda})$ is a flat projective connection on $\mathscr{V}_{\Sigma_T, \Gamma, \phi}(\vec{q}, \vec{\lambda})$.  
\end{theorem}

\begin{proof}
For any $b\in T$, choose an affine open subset $b\in U$ with an $s$-tuple of parameters $\vec{t}$ around $\vec{q}$. Given  $\theta_1,\theta_2 \in \Theta_{\Sigma^o_T, T}(U)^\Gamma$, 
by Proposition \ref{lem3.2.4} and the formula (\ref{Connection_opeartor_definition}), the difference $\nabla^{\vec{t}}_{[\theta_1,\theta_2]}-[\nabla^{\vec{t}}_{\theta_1},\nabla^{\vec{t}}_{\theta_2} ]$ is a  $\mathscr{O}_T(U)$-scalar operator  and so is $[\nabla^{\vec{t}}_\theta, f]$ for $f\in \mathscr{O}_T(U)$.  
It follows that $\mathscr{L}_T$ is a sheaf of $\bc$-Lie algebra acting on $\mathscr{V}_{\Sigma_T, \Gamma, \phi}(\vec{q}, \vec{\lambda})$.

Note that $\Theta_{\Sigma^o_T/T}^\Gamma|_b\simeq \Theta(\Sigma^o_b)^\Gamma$, where $\Sigma^o_b$ is the affine curve $\Sigma_b\backslash \pi^{-1}(\vec{p}(b))$, $\vec{p}=\pi\circ \vec{q}$ and $\Theta(\Sigma^o_b)^\Gamma$ is the Lie algebra of $\Gamma$-invariant vector fields on $\Sigma^o_b$. 
In view of the part (2) of Theorem \ref{covacua_coherent_basechange},  $\mathscr{V}_{\Sigma_T, \Gamma, \phi}(\vec{q}, \vec{\lambda})|_b\simeq \mathscr{V}_{\Sigma_b, \Gamma, \phi}(\vec{q}(b), \vec{\lambda})$.  Therefore, 
 the $\mathscr{O}_T(U)$-linear map $\nabla^{\vec{t}}:  \Theta_{\Sigma^o_T/T}(U)^\Gamma \to  {\rm End}_{\mathscr{O}_T(U)} (\mathscr{V}_{\Sigma_T, \Gamma, \phi}(\vec{q}, \vec{\lambda})|_U)$
induces a projective representation of $\Theta(\Sigma^o_b)^\Gamma$ on the space of covacua  $\mathscr{V}_{\Sigma_b, \Gamma, \phi}(\vec{q}(b), \vec{\lambda})$ attached to the $s$-pointed curve $(\Sigma_{b}, \vec{q}(b))$. By Lemma \ref{parameter_Sugawara}, the map $\nabla^{\vec{t}}$ is independent of the choice of $\vec{t}$ if we consider it projectively as a map
$$\nabla^{\vec{t}}:  \Theta_{\Sigma^o_T/T}(U)^\Gamma \to  {\rm End}_{\mathscr{O}_T(U)} (\mathscr{V}_{\Sigma_T, \Gamma, \phi}(\vec{q}, \vec{\lambda})|_U)/\mathscr{O}_T(U).$$

 Note that $\Theta (\Sigma^o_b)^\Gamma$ is isomorphic to the Lie algebra  $\Theta(\bar{\Sigma}^o_{b})$ of vector fields on the affine curve $\bar{\Sigma}_b\backslash \vec{p}(b)$ (since $\Gamma\cdot \vec{q}(b)$ contains the ramification locus). 
   By \cite[Lemma 2.5.1]{BFM}, $\Theta(\bar{\Sigma}^o_{b})$ and hence   $\Theta(\Sigma^o_b)^\Gamma$  is an infinite dimensional simple Lie algebra.   Since $\mathscr{V}_{\Sigma_b, \Gamma, \phi}(\vec{q}(b), \vec{\lambda})$ is a finite-dimensional vector space,  $\Theta(\Sigma^o_b)^\Gamma$ can only act by scalars on  $\mathscr{V}_{\Sigma_b, \Gamma, \phi}(\vec{q}(b), \vec{\lambda})$. Therefore, for any $\theta\in  \Theta_{\Sigma^o_T/ T}(U)^\Gamma$, the operator $\nabla^{\vec{t}}_{\theta}$ acts via multiplication by an element of $\mathscr{O}_T(U)$ on  
$\mathscr{V}_{\Sigma_T, \Gamma, \phi}(\vec{q}, \vec{\lambda})|_U$.  Thus, the sequence \eqref{seq110} for $\mathscr{L}=\mathscr{L}_T$
is exact. 
It follows that $\mathscr{L}_T$ is indeed a flat projective connection on  $\mathscr{V}_{\Sigma_T, \Gamma, \phi}(\vec{q}, \vec{\lambda})$.
\end{proof}

\section{Local freeness of the sheaf  of twisted conformal blocks on stable compactification of Hurwitz stacks}
\label{Hurwitz_Section}

We consider families of  stable $s$-pointed $\Gamma$-curves and we show that the sheaf of twisted covacua over the  stable compactificaiton of Hurwitz stack is locally free.

{\it In this section, we fix a group homomorphism $\phi: \Gamma\to {\rm Aut }(\fg)$ such that $\Gamma$ stabilizes a Borel subalgebra $\fb$ of $\fg$. }

\begin{definition}
\label{stable_family_cover} \cite[D\'efinition 4.3.4]{BR}
We say that a family of $s$-pointed $\Gamma$-curves $(\Sigma_T, \vec{q})$ over a scheme $T$  (see Definition \ref{defi3.3.1}) is {\it stable} if  
\vskip1ex

(0) Each geometric fiber $\Sigma_b$ of $\Sigma_T$ is (connected) with only nodal singularity;
\vskip1ex

(1) the  family of $s$-pointed curves  $( \bar{\Sigma}_T, \vec{p})$  is stable where $\vec{p}:=\pi\circ \vec{q}$, i.e., for any geometric point $b\in T$ the  fiber $\bar{\Sigma}_b$ is a  connected reduced curve with at most  nodal singularity and the automorphism group of  the pointed curve $(\bar{\Sigma}_b, \vec{p}(b))$ is finite;

\vskip1ex

(2)  the action of $\Gamma$ on each geometric fiber $\Sigma_b$ is stable in the sense of Definition \ref{stable_action} (in particular, $\Sigma_b$ has only nodal singularity). Moreover, $\Gamma\cdot \vec{q}(b)$ contains all the ramification points for any $b\in T$. 

\vskip1ex
A $s$-pointed $\Gamma$-curve is called {\it stable} if it is stable as a family over a point.

\end{definition}

\begin{remark} \label{remark8.2} {\rm For a family  of $s$-pointed $\Gamma$-curves $(\Sigma_T, \vec{q})$ over $T$ satisfying the properties (0) and (2)  as above, the stability of 
 $( \bar{\Sigma}_T, \vec{p})$ is equivalent to the stability of $(\Sigma_T, \Gamma\cdot \vec{q})$
 (cf. \cite[Proposition 5.1.3]{BR}).
 
 Moreover, under the assumption that $\Gamma\cdot \vec{q}$ contains all the ramification points in $\Sigma$, at any nodal point $q\in \Sigma$, $q$ being unramified and stable, ${\rm det}(\dot \sigma)=1$,   $\sigma$ fixes the two branches for any $\sigma\in \Gamma_q$ and $\Gamma_q$ is cyclic (cf. \cite[Corollaire 4.3.3 and the comment after Definition 6.2.3]{BR}).  In this case, any stable $s$-pointed $\Gamma$-curve $(\Sigma, \vec{p})$ is exactly a $s$-pointed admissible $\Gamma$-cover in the sense of Jarvis-Kaufmann-Kimura \cite[Definition 2.1, 2.2]{JKK}. The only difference is that, in our definition, stable $s$-pointed $\Gamma$-curves are connected, and admissible $s$-pointed $\Gamma$-covers defined in \cite{JKK} can be disconnected. } 
\end{remark}

Let $(C_o, \vec{q}_o )$ be a $s$-pointed $\Gamma$-curve such that $\Gamma$ acts stably on $C_o$ (cf. Definition \ref{stable_action}).
 Let $\tilde{C}_o$ be the normalization of $C_o$ at the points $\Gamma\cdot r$, where $r$ is a (stable) nodal point of $C_o$.   The nodal point $r$ splits into two smooth points $r',r''$ in $\tilde{C}_o$. Let $z'$ (resp. $z''$) be a local parameter at $r'$ (resp. $r''$) for the curve $\tilde{C}_o$. The following lemma shows that there exists a canonical smoothing deformation of $(C_o, \vec{q}_o )$ over a formal disc $\mathbb{D}_\tau:={\rm Spec}\, \mathbb{C}[[\tau]]$. We denote by $\mathbb{D}^\times_\tau$ the associated punctured formal disc ${\rm Spec}\, \mathbb{C}((\tau))$.
\begin{lemma}
\label{smoothing_construction}
With the same notation as above, we assume that the stabilizer group $\Gamma_r$ at $r$ is cyclic and does not exchange the branches.  
Then, there exists a  formal deformation $(C, \vec{q})$ of the $s$-pointed $\Gamma$-curve $(C_o, \vec{q}_o )$ over a formal disc $\mathbb{D}_\tau$ with the formal parameter $\tau$, such that the following properties hold:
\vskip1ex

(1)  over the closed point $o\in \mathbb{D}_\tau$, $(C,\vec{q})|_{\tau=0}=(C_o,\vec{q}_o)$;
\vskip1ex

(2) over the punctured formal disc $\mathbb{D}^\times_\tau$,  $(C,\vec{q})|_{\mathbb{D}^\times_\tau}$ is a $s$-pointed smooth projective curve over $\mathbb{C}((\tau))$ if $\Gamma\cdot r$ are the only nodes in $C_o$;

\vskip1ex

(3) 
 the completed local ring $\hat{\mathscr{O}}_{C, r}$ of $\mathscr{O}_C$ at $r$ is isomorphic to $\mathbb{C}[[z',z'', \tau]]/\langle \tau- z'z''\rangle \simeq \mathbb{C}[[z',z'']]$, where $\Gamma_r$ acts on $z'$  (resp. $z''$) via a primitive character $\chi$ (resp. $\chi^{-1}$);
\vskip1ex

 (4)  there exists a $\Gamma$-equivariant  isomorphism of $\mathbb{C}[[\tau]]$-algebras
\begin{equation} \label{eqn8.3.1}
 \kappa:  \hat{\mathscr{O}}_{C\backslash \Gamma\cdot r,\, C_o \backslash \Gamma \cdot r }\simeq\mathscr{O}_{C_o\backslash \Gamma\cdot r}[[\tau]],
\end{equation}
 where $ \hat{\mathscr{O}}_{C\backslash \Gamma\cdot r,\, C_o \backslash \Gamma \cdot r }$
 is the completion  of   ${\mathscr{O}}_{C\backslash \Gamma\cdot r}$
along  $C_o \backslash \Gamma \cdot r $. 
 
\end{lemma}
\begin{proof}
In the non-equivariant case, this smoothing construction as formal deformation is sketched by Looijenga in \cite[Section 6]{L}, and detailed argument from formal deformation to algebraic deformation can be found in 
 \cite[\S 6.1]{D}.   These constructions/arguments can be easily generalized to the equivariant setting when $\Gamma_r$ acts on the node stably and does not exchange the branches. 
\end{proof}

 Let $\hat{L}(\fg, \Gamma_{r'})$ (resp. $\hat{L}(\fg, \Gamma_{r''}) $) be the Kac-Moody algebra attached to the point $r'$ (resp. $r''$) in $\tilde{C}_o$.  Recall that (Lemma \ref{lemma 4.2}) $\mu^*\in D_{c, r'}$ if and only if $\mu\in D_{c, r''}$ where $V(\mu^*)\simeq V(\mu)^*$. (By Lemma \ref{normalization_pullback}, $\Gamma_{r'}=\Gamma_{r''}$ and hence $\fg^{\Gamma_{r'}}=\fg^{\Gamma_{r''}}$.)
Let $\mathscr{H}(\mu^*)$ (resp. $\mathscr{H}(\mu)$) be the highest weight  integrable representation of $\hat{L}(\fg, \Gamma_{r'})$ (resp. $\hat{L}(\fg, \Gamma_{r''})$) as usual.
\begin{lemma}
\label{non-degenerate_pairing}
There exists a non-degenerate pairing $b_\mu:  \mathscr{H}(\mu^*) \times \mathscr{H}(\mu )\to \mathbb{C}$ such that for any $h_1\in \mathscr{H}(\mu^*), \, h_2\in \mathscr{H}(\mu) $, and $x[z'^n]\in \hat{L}(\fg, \Gamma_{r'})$,
\[ b_\mu(    x[z'^n]\cdot h_1, h_2  )+b_\mu(h_1,    x[z''^{-n} ]\cdot h_2 ) =0. \]
Note that $x[z'^n]\in \hat{L}(\fg, \Gamma_{r'})$ if and only  if $x[z''^{-n}]\in \hat{L}(\fg, \Gamma_{r''})$.
\end{lemma}
\begin{proof}
From  Lemma \ref{lemma 4.2} (especially see  `another proof of Lemma \ref{lemma 4.2} Part (2)'),  there exists an isomorphism $\hat{\omega}:  \tilde{L}(\fg, \Gamma_{r'})\simeq \tilde{L}(\fg, \Gamma_{r''})$, such that the representation of $\tilde{L}(\fg, \Gamma_{r''})$ on $\mathscr{H}(\mu)$ via $\hat{\omega}^{-1}$ is isomorphic to $\mathscr{H}(\mu^*)$, where $\tilde{L}(\fg, \Gamma_{r'})$ is the non-completed Kac-Moody algebra.  By \cite[\S 9.4]{Ka}, there exists a contravariant form $\bar{b}_\mu: \mathscr{H}(\mu^*)\times \mathscr{H}(\mu^*)\to \mathbb{C}$ such that 
\[ \bar{b}_\mu(x[f]\cdot h_1, h_2 )+\bar{b}_\mu(h_1, \varpi(x[f])h_2)=0, \text{ for any } x[f]\in \tilde{L}(\fg, \Gamma_{r'}), h_1,h_2\in \mathscr{H}(\mu^*), \]
where $\varpi$ is the Cartan involution of $\tilde{L}(\fg, \Gamma_{r'})$ mapping $x'_i[z'^{s_i}]$ (resp. $y'_i[z'^{-s_i}]$) to $-y'_i[z'^{-s_i}]$ (resp. $-x'_i[z'^{s_i}]$) for any $i\in \hat{I}(\fg,\Gamma_{r'})$, see these notation in the second proof of Lemma \ref{lemma 4.2} part (2).   Observe that the composition $ \hat{\omega}\circ \varpi: \tilde{L}(\fg, \Gamma_{r'})\to \tilde{L}(\fg, \Gamma_{r''})$ is  an isomorphism of Lie algebras mapping $x[z'^n]$  to $x[z''^{-n}]$.  Hence, the lemma follows after we identify the second copy of $\mathscr{H}(\mu^*)$ in $\bar{b}_\mu$ with $\mathscr{H}(\mu)$ via $\hat{\omega}^{-1}$  mentioned above.
\end{proof}

There exist direct sum decompositions by $t$-degree (putting the $t$-degree  of the highest weight vectors at $0$):
\[ \mathscr{H}(\mu^*)=\bigoplus_{d=0}^\infty   \mathscr{H}(\mu^*)_{-d} ,\,  \mathscr{H}(\mu)=\bigoplus_{d=0}^\infty   \mathscr{H}(\mu)_{-d}.  \]

The non-degenerate pairing $b_\mu$ in Lemma \ref{non-degenerate_pairing} induces a non-degenerate pairing $b_{\mu,d}: \mathscr{H}(\mu^*)_{-d}\times  \mathscr{H}(\mu)_{-d}\to \mathbb{C}$ for each $d\geq 0$.  Let $b^*_{\mu,d}\in  (\mathscr{H}(\mu^*)_{-d})^*\otimes  (\mathscr{H}(\mu)_{-d})^*$ be the dual of $b_{\mu, d}$. The contravariant form $\bar{b}_\mu$  on  $ \mathscr{H}(\mu^*)$ with respect to $\tilde{L}(\fg, \Gamma_{r'})$  induces an isomorphism $c'_\mu:  (\mathscr{H}(\mu^*)_{-d})^*\simeq  \mathscr{H}(\mu^*)_{-d}$. Similarly, the contravariant form on  $ \mathscr{H}(\mu)$ with respect to $\tilde{L}(\fg, \Gamma_{r''})$  induces an isomorphism $c''_\mu: (\mathscr{H}(\mu)_{-d})^*\simeq  \mathscr{H}(\mu)_{-d}$ (the Cartan involution on $\tilde{L}(\fg, \Gamma_{r''})$  taken here is obtained from $\varpi$ on $\tilde{L}(\fg, \Gamma_{r'})$  via the isomorphism $\hat{\omega}$).  Set $\Delta_{\mu,d}: =(c'_\mu\otimes c''_\mu)(b^*_{\mu,d})\in\mathscr{H}(\mu^*)_{-d}\otimes  \mathscr{H}(\mu)_{-d}$ if $d\geq 0$ and $0$ if $d<0$. Note that $\Delta_{\mu, 0}$ is exactly the element $I_\mu$ induced from the identity map on $V(\mu)$ (see the formula (\ref{Schur_map})).
  In view of Lemma \ref{non-degenerate_pairing}, $\Delta_{\mu, d}$ satisfies the following property (for any $d,n\in \bz$)
\begin{equation}
\label{Gluing_element_vanishing}
  (x[z'^n]\otimes 1)\cdot \Delta_{\mu, d+n} + (1\otimes x[z''^{-n}])\cdot  \Delta_{\mu, d}=0, \, \text{ for any } x[z'^n]\in \hat{L}(\fg, \Gamma_{r'}) .  \end{equation} 

We now construct the following `gluing' tensor element (following \cite[Lemma 6.5]{L} in the non-equivariant setting),
\[ \Delta_\mu:=\sum_{d\geq 0} \Delta_{\mu, d} \tau^d\in  \left(\mathscr{H}(\mu^*)\otimes \mathscr{H}(\mu)\right)[[\tau]] .  \]

Let $\theta',\theta''$ be the maps defined via Lemma \ref{smoothing_construction} Part (3),
\[  \theta': \hat{\mathscr{O}}_{C, r}  \to  \mathbb{C}((z'))[[\tau]], \text{ and } \theta'': \hat{\mathscr{O}}_{C, r}\to  \mathbb{C}((z''))[[\tau]], \]
where $ \hat{\mathscr{O}}_{C, r}$ is the completion of $ {\mathscr{O}}_{C}$ along $r$.
Thus, for any $f(z',z'')=\sum_{i\geq 0,j\geq 0} a_{i,j}z'^i z''^j\in \hat{\mathscr{O}}_{C, r} $ (see Lemma \ref{smoothing_construction} Part (3)),  we have
\[\theta'(f)=f(z', \tau/z') =\sum_{j\geq 0} (\sum_{i\geq 0} a_{i,j} z'^{i-j}  )  \tau^j  , \]
and 
\[\theta''(f)=f(\tau/z'', z'') =\sum_{i\geq 0} (\sum_{j\geq 0} a_{i,j} z''^{j-i}  )  \tau^i  . \]
The morphisms $\theta',\theta''$ induce a $\mathbb{C}[[\tau]]$-module morphism $\theta: (\fg\otimes  \hat{\mathscr{O}}_{C, r} )^{\Gamma_{r}} \to   ( \fg\otimes \mathbb{C}((z')) )^{\Gamma_{r'}}[[\tau]] \oplus    ( \fg\otimes \mathbb{C}((z'')) )^{\Gamma_{r''}} [[\tau]]$, where $\tau$ acts on
 $ \hat{\mathscr{O}}_{C, r}$ via
$$\tau\cdot f(z', z'')= z'z'' f(z', z'') .$$
Thus,  we get an  injective map from $(\fg\otimes  \hat{\mathscr{O}}_{C, r} )^{\Gamma_r}$ into $\hat{L}(\fg,\Gamma_{r'})[[\tau]]\oplus \hat{L}(\fg,\Gamma_{r''})[[\tau]]$ (but not a Lie algebra homomorphism), which acts on $(\mathscr{H}(\mu^*)\otimes \mathscr{H}(\mu))[[\tau]] $.

\begin{lemma}
\label{Gluing_tensor_lemma}
The  element  $\Delta_\mu\in \left(\mathscr{H}(\mu^*)\otimes \mathscr{H}(\mu)\right)[[\tau]] $ is annihilated by  $(\fg\otimes \hat{  \mathscr{O}}_{C,r})^{\Gamma_r}$ via the morphism $\theta$ defined as above. 
\end{lemma}
\begin{proof} For any  $ x[z'^iz''^j]\in  (\fg\otimes \hat{\mathscr{O}}_{C,r})^{\Gamma_r}$,
\begin{align*} x[z'^iz''^j]\cdot \Delta_\mu &=\sum_{d\in \bz}\,(x[z'^{i-j}]\otimes 1) \Delta_{\mu, d}\tau^{d+j} + \sum_{d\in \bz}\,
(1\otimes x[z''^{j-i}]) \Delta_{\mu, d}\tau^{d+i}\\
&= -\sum_{d\in \bz}\,(1\otimes x[z''^{j-i}]) \Delta_{\mu, d+j-i}\tau^{d+j}+ \sum_{d\in \bz}\,
(1\otimes x[z''^{j-i}]) \Delta_{\mu, d}\tau^{d+i}\,,\,\,\,\text{by \eqref{Gluing_element_vanishing}}\\
&=0.
\end{align*}
From this it is easy to see that $x[f]\cdot \Delta_\mu=0$ for any $x[f]\in  (\fg\otimes \hat{\mathscr{O}}_{C,r})^{\Gamma_r} $. This proves the lemma.
\end{proof}

For each $i=1,\dots, s$, let $\mathscr{H}(\lambda_i)_{\mathbb{D}_\tau}$ (resp. $\mathscr{H}(\lambda_i)$) denote the integrable representation of $\hat{L}(\fg, \Gamma_{q_i})_{\mathbb{D}_\tau}$ (resp. $\hat{L}(\fg, \Gamma_{q_{i,o}})$ ) attached to $\hat{\mathscr{O}}_{C, q_i}$ (resp. $\hat{\mathscr{O}}_{C_o, q_{i,o}}$) as in Section \ref{Sugawara_section}, and let $ \mathscr{H}(\vec{\lambda} )_{\mathbb{D}_\tau }$ (resp.  $\mathscr{H}(\vec{\lambda} )$) denote their tensor product over $\mathbb{C}[[\tau]]$ (resp. $\mathbb{C}$). 
For each $i=1,\dots, s$,  we choose a $(\Gamma_{q_i},\chi_i)$-equivariant formal parameter $z_i$ around $q_i$, i.e. $\hat{\mathscr{O}}_{C,q_i}\simeq \mathbb{C}[[\tau]][[z_i]]$, where $\chi_i$ is a primitive character of $\Gamma_{q_i}$.  It gives rise to a trivialization (cf. Formula (\ref{eqn5.5.3}))
\[t_{\vec{\lambda}}: \mathscr{H}(\vec{\lambda} )_{\mathbb{D}_\tau }\simeq  \mathscr{H}(\vec{\lambda} )\otimes_\mathbb{C}  \mathbb{C}[[\tau]].   \]

We now construct a morphism of $\mathbb{C}[[\tau]]$-modules:
\[ \tilde{F}_{\vec{\lambda}}:   \mathscr{H}(\vec{\lambda} )\otimes_\mathbb{C}  \mathbb{C}[[\tau]] \longrightarrow  \bigoplus_{\mu\in D_{c, r''} } ( \mathscr{H}(\vec{\lambda} )\otimes \mathscr{H}(\mu^*)\otimes \mathscr{H}(\mu) ) [[\tau]] \]
given by
\[ \sum_{i=0}^\infty  h_i \tau^i  \mapsto  \sum_{i,d=0}^\infty (h_i\otimes \Delta_{\mu,d} )\tau^{i+d} ,\]
where,  for each $i$, $ h_i\in  \mathscr{H}(\vec{\lambda} ) $. Finally, we set 
\[F_{\vec{\lambda}}:=\tilde{F}_{\vec{\lambda}} \circ t_{\vec{\lambda}} : \mathscr{H}(\vec{\lambda} )_{\mathbb{D}_\tau }\longrightarrow  \bigoplus_{\mu\in D_{c, r''} } ( \mathscr{H}(\vec{\lambda} )\otimes \mathscr{H}(\mu^*)\otimes \mathscr{H}(\mu) ) [[\tau]] .    \]
Consider the following canonical homomorphisms (obtained by  restrictions):
$$ \mathscr{O}_{{C}\backslash \Gamma\cdot \vec{q}}\to \mathscr{O}_{C\backslash \Gamma\cdot (\vec{q} \cup \{r\} )}  \to
 \hat{\mathscr{O}}_{C\backslash \Gamma\cdot (\vec{q}\cup\{r\}), C_o\backslash \Gamma\cdot (\vec{q}_o\cup\{r\}) }\overset{\kappa'}{\cong}
\mathscr{O}_{C_o\backslash \Gamma\cdot (\vec{q}_o\cup\{r\}) }[[\tau]],
$$
where the isomorphism $\kappa'$ is induced from the isomorphism $\kappa$ of Lemma \ref{smoothing_construction} (see the isomorphism \eqref{eqn8.3.1}).
This gives rise to a Lie algebra homomorphism (depending upon the isomorphism $\kappa'$):
$$\kappa_{\vec{q}}: \fg[C\backslash \Gamma\cdot \vec{q}]^\Gamma \to \left[\fg\otimes \mathscr{O}_{\tilde{C}_o\backslash \Gamma\cdot (\vec{q}_o\cup\{r',r''\}) }\right]^\Gamma [[\tau]].$$
Hence, the Lie algebra $\fg[C\backslash \Gamma\cdot \vec{q}]^\Gamma $ acts on 
 $\left( \mathscr{H}(\vec{\lambda} )\otimes \mathscr{H}(\mu^*)\otimes \mathscr{H}(\mu) \right) [[\tau]] $ via the action of 
$ \left[\fg\otimes \mathscr{O}_{\tilde{C}_o\backslash \Gamma\cdot (\vec{q}_o\cup\{r',r''\}) }\right]^\Gamma$ on  $ \mathscr{H}(\vec{\lambda} )\otimes \mathscr{H}(\mu^*)\otimes \mathscr{H}(\mu) $ at the points $\{\vec{q}_o, r', r''\}$ as given just before Theorem \ref{thm3.1.2} and extending it $\bc[[\tau]]$-linearly. 

Recall from Definition \ref{def_sheaf_conformal_block} the action of $\fg[C\backslash \Gamma\cdot \vec{q}  ]^\Gamma$ on $ \mathscr{H}(\vec{\lambda} )_{\mathbb{D}_\tau}$. Further, $\fg[C\backslash \Gamma\cdot \vec{q}  )^\Gamma$ acts on $ (\mathscr{H}(\mu^*)\otimes \mathscr{H}(\mu)) [[\tau]]$ via the Lie algebra homomorphism (obtained by the restriction):
$$\fg[C\backslash \Gamma\cdot \vec{q}  ]^\Gamma \to (\fg\otimes \hat{\mathscr{O}}_{C, r})^{\Gamma_r}$$
and the action of $(\fg\otimes \hat{\mathscr{O}}_{C, r})^{\Gamma_r}$ on $ (\mathscr{H}(\mu^*)\otimes \mathscr{H}(\mu)) [[\tau]]$ (which is a Lie algebra action only projectively)  is given just before Lemma \ref{Gluing_tensor_lemma}.
\begin{theorem}
\label{key_theorem}
We have the following:
\begin{enumerate}
\item the morphism $F_{\vec{\lambda}}$ is $\fg[C\backslash \Gamma\cdot \vec{q}]^\Gamma$-equivariant; 
\item the morphism $F_{\vec{\lambda}}$ induces an isomorphism  of sheaf of covacua over $\mathbb{D}_\tau$:
\begin{equation}
\label{gluding_morphism}
 \bar{F}_{\vec{\lambda}}:    \mathscr{V}_{C, \Gamma, \phi}(\vec{q},\vec{\lambda} )\to  \bigoplus_{\mu\in D_{c, r''}}  \mathscr{V}_{  \tilde{C}_o,  \Gamma,\phi }\left((\vec{q}_o, r', r'') , (\vec{\lambda}, \mu^*, \mu)\right) [[\tau]] .\end{equation}
 \end{enumerate}
 Note that here we take  slightly different notation for the spaces/sheaves of covacua, see Remark \ref{notation_warning}.

\end{theorem}
\begin{proof}
By Lemma \ref{Gluing_tensor_lemma},  the morphism 
\[  F'_{\vec{\lambda}}: \mathscr{H}(\vec{\lambda} )_{\mathbb{D}_\tau}  \to   \mathscr{H}(\vec{\lambda} )_{\mathbb{D}_\tau} \otimes_{\mathbb{C}[[\tau]] }  \bigoplus_{\mu \in D_{c, r''}}\, ( \mathscr{H}(\mu^*)\otimes \mathscr{H}(\mu)) [[\tau]]       \]
given by $h\mapsto \sum_{\mu \in D_{c,r''}} h\otimes \Delta_{\mu}$,  is a morphism of $\fg[C\backslash \Gamma\cdot \vec{q} ]^\Gamma$-modules.  Moreover, there exists an embedding obtained from the isomorphism $t_{\vec{\lambda}}$:
\[   i: \mathscr{H}(\vec{\lambda} )_{\mathbb{D}_\tau }  \otimes_{\mathbb{C}[[\tau]] }  \bigoplus_{\mu \in D_{c, r''}}\,
 ( \mathscr{H}(\mu^*)\otimes \mathscr{H}(\mu)) [[\tau]]     \hookrightarrow   \bigoplus_{\mu \in D_{c, r''}}\, 
  (  \mathscr{H}(\vec{\lambda} )\otimes  \mathscr{H}(\mu^*)\otimes \mathscr{H}(\mu)) [[\tau]].\]
Observe that $F_{\vec{\lambda}}=i\circ F'_{\vec{\lambda}}$. It concludes part (1) of the theorem.

We now proceed to prove part (2) of the theorem.   Using part (1) of the theorem and the morphism $\kappa_{\vec{q}}$, we get the $\mathbb{C}[[\tau]]$-morphism (\ref{gluding_morphism}). 
 Taking quotient by $\tau$, by the Factorization Theorem (Theorem \ref{thm3.1.2}) the morphism $\bar{F}_{\vec{\lambda}}$ gives rise to an isomorphism 
\[    \mathscr{V}_{C_o, \Gamma, \phi}(\vec{q}_o,\vec{\lambda} )\to  \bigoplus_{\mu\in D_{c, r''}}  \mathscr{V}_{  \tilde{C}_o,  \Gamma,\phi }
\left((\vec{q}_o, r', r'' ), ( \vec{\lambda}, \mu^*, \mu)\right). \]
As a consequence of the Nakayama Lemma (cf. [AM, Exercise 10, Chap. 2]),  $\bar{F}_{\vec{\lambda}}$ is surjective. (Observe that by Theorem \ref{covacua_coherent_basechange}, both the domain and the range of $\bar{F}_{\vec{\lambda}}$ are finitely generated $\bc[[\tau]]$-modules.)
Now, since the range of $\bar{F}_{\vec{\lambda}}$ is a free $\bc[[\tau]]$-module, we get that $\bar{F}_{\vec{\lambda}}$ splits over $\bc[[\tau]]$. Thus, applying the
Nakayama lemma (cf. [AM, Proposition 2.6]) again  to the kernel $K$ of $\bar{F}_{\vec{\lambda}}$, we get that $K=0$. Thus, $\bar{F}_{\vec{\lambda}}$ is an isomorphism, proving (2).
\end{proof}

\begin{definition}
\label{ramification_datum}
We say that a stable $s$-pointed $\Gamma$-curve $(\Sigma, q_1,\cdots, q_s)$ has  {\it marking data}  $\eta=\left((\Gamma_1,\chi_1), (\Gamma_2,\chi_2),\cdots, (\Gamma_s, \chi_s)\right)$ if for each $i$, the stabilizer group at $q_i$ is a (cyclic) subgroup $\Gamma_i \subset \Gamma$ and $\chi_i$ is the induced (automatically primitive) character of $\Gamma_i$ on the tangent space $T_{q_i} \Sigma$. 

We now introduce the  moduli stack $\overline{\mathscr{H}M}_{g, \Gamma,  \eta}$, which associates to each $\mathbb{C}$-scheme $T$ the groupoid of stable family $\xi:\Sigma_T \to T$ of $s$-pointed $\Gamma$-curves over $T$, such that  each geometric fiber is of genus $g$ and  is of marking data  $\eta$.   In particular,  $\bigcup_{i=1}^s \Gamma\cdot q_i(T)$ contains the ramification divisor of $\pi: \Sigma_T \to \Sigma_T/\Gamma$ (cf. \cite[Definition 4.1.6]{BR}).   Note that for any  stable family of $s$-pointed $\Gamma$-curves in $\overline{\mathscr{H}{M}}_{g, \Gamma,  \eta}$, its geometric fibers  contain at worst  only nodal singularity such that  their stabilizer groups are cyclic which do not exchange the branches (cf. \cite[Corollaire 4.3.3]{BR}, and the comment after \cite[Definition 6.2.3]{BR}).

For any $\gamma_1, \dots, \gamma_s\in \Gamma$, $(\Sigma, \gamma_1q_1, \dots, \gamma_sq_s)$ has the conjugate marking data 
$$((\gamma_1\Gamma_1\gamma_1^{-1}, ^{\gamma_1}\chi_1), \dots, (\gamma_s\Gamma_s\gamma_s^{-1}, ^{\gamma_s}\chi_s)),$$ where 
$ (^{\gamma_i}\chi_i) (\gamma_ig_i\gamma_i^{-1}):= \chi_i(g_i)$, for $g_i\in \Gamma_i$. We denote by $[\eta]$, the $\Gamma^s$-conjugacy class of $\eta$. 
\end{definition}

\begin{theorem} \label{thm8.8}
$\overline{\mathscr{H}{M}}_{g, \Gamma,  \eta}$ is a proper and smooth Deligne-Mumford stack of finite type.
\end{theorem}
\begin{proof} (Sketch)
We can associate to $(\Sigma_T, \vec{q})$ (a stable family $\xi:\Sigma_T \to T$ of $s$-pointed $\Gamma$-curves)  the $\Gamma$-stable relative Cartier divisor $\bigcup_{i}\Gamma \cdot (q_i(T))$ in $\Sigma_T$ which is \'etale over $T$.  The $\Gamma^{\times s}$-conjugacy classes $[\eta]$ of $\eta$ is the marking type of  $\bigcup_{i}\, \Gamma \cdot (q_i(T))$.  Let $[\xi]$ be the subclass of those conjugacy classes  $[\Gamma_i,\chi_i]$ such that $\Gamma_i$ is nontrivial. Then, $[\xi]$ is the associated ramification datum of stable $s$-pointed $\Gamma$-curves in $\overline{\mathscr{H}{M}}_{{g}, \Gamma,  \eta}$. Let $\overline{\mathscr{H}M}_{g, \Gamma,[\xi], [\eta]}$ be the stable compactification of Hurwitz stack defined in \cite[Definition 6.2.3]{BR}.  The natural morphism  $\overline{\mathscr{H}{M}}_{{g}, \Gamma,  \eta}\to \overline{\mathscr{H}M}_{g, \Gamma,[\xi], [\eta]}$  is clearly representable,  \'etale and essentially surjective.  By \cite[Th\'eor\`em 6.3.1]{BR},  $\overline{\mathscr{H}M}_{g, \Gamma,[\xi], [\eta]}$  is a smooth proper Deligne-Mumford stack,  and hence  so is $\overline{\mathscr{H}{M}}_{g, \Gamma,  \eta}$.
\end{proof}

Let $D_{c,i}$ be the set of dominant weights of $\fg^{\Gamma_i}$ associated to the highest weight  integrable  (irreducible)
 representations of $\hat{L}(\fg, \Gamma_i, \chi_i):= \hat{L}(\fg, \sigma_i)$, where $\sigma_i\in \Gamma_i$ is the unique element such that $\chi_i(\sigma_i)= e^{\frac{2\pi i}{|\Gamma_i|}}$.  Choose a collection $\vec{\lambda}=(\lambda_1, \cdots, \lambda_s)$ of dominant weights, where $\lambda_i\in D_{c, i}$ for each $i$. 
Recall that being a Deligne-Mumford stack, $\overline{\mathscr{H}{M}}_{{g}, \Gamma,  \eta}$ has an atlas $\delta: X\to \overline{\mathscr{H}{M}}_{g, \Gamma, \eta }$ such that   $\delta$ is \'etale and surjective. By Theorem \ref{thm8.8}, $X$ is a smooth (but not necessarily connected) scheme of finite type over $\mathbb{C}$.

We can attach to $\delta: X\to \overline{\mathscr{H}{M}}_{g, \Gamma,  \eta}$  the coherent sheaf $\mathscr{V}_{\Sigma_X,  \Gamma,\phi }(  \vec{q}_X , \vec{\lambda})$   of conformal blocks, where $(\Sigma_X, \vec{q}_X)$ is the associated stable family of $s$-pointed $\Gamma$-curves over $X$.  This attachment can be done componentwise on $X$ via Definition \ref{def_sheaf_conformal_block}. 

For any two atlases $X,Y$ of $\overline{\mathscr{H}{M}}_{g, \Gamma,  \eta}$ and a morphism $f: Y\to X$ compatible with the atlas structures,   by Theorem \ref{covacua_coherent_basechange},  there exists a canonical isomorphism 
\[  \alpha_f: f^*\mathscr{V}_{\Sigma_X,  \Gamma,\phi }(  \vec{q}_X , \vec{\lambda}) \simeq  \mathscr{V}_{\Sigma_Y,  \Gamma,\phi }(  \vec{q}_Y , \vec{\lambda}), \]
 where $(\Sigma_X, \vec{q}_X)$ and $(\Sigma_Y, \vec{q}_Y)$ are the families of stable $s$-pointed $\Gamma$-curves associated to these two atlases $X,Y$.  Given three atlases $X,Y,Z$ and morphisms $g:Z\to Y$ and $f: Y\to X$, Theorem \ref{covacua_coherent_basechange} ensures the obvious cocycle condition.  Therefore, we get 
a coherent  sheaf $\mathscr{V}_{g, \Gamma, \phi}(\eta,  \vec{\lambda})$ on  $\overline{\mathscr{H}{M}}_{g, \Gamma,   \eta}$ such that $\delta^* \mathscr{V}_{g, \Gamma, \phi}(\eta,  \vec{\lambda}) \simeq \mathscr{V}_{\Sigma_X,  \Gamma,\phi }(  \vec{q}_X , \vec{\lambda})$ for any  atlas $\delta:X\to  \overline{\mathscr{H}{M}}_{g, \Gamma,   \eta}$.  Some basics of coherent sheaves on Deligne-Mumford stacks can be found  in \cite[Definition C.20]{Ku2}.


\begin{theorem}
\label{Hurwitz_locally_free}
For any genus $g\geq 0$, any marking data $\eta=\left((\Gamma_1,\chi_1), (\Gamma_2,\chi_2),\cdots, (\Gamma_s, \chi_s)\right)$ and any set of dominant wights $\vec{\lambda}=(\lambda_1, \dots, \lambda_s)$ with $\lambda_i\in D_{c,i}$ , 
the sheaf of conformal blocks $\mathscr{V}_{g, \Gamma, \phi}(\eta,  \vec{\lambda})$ is locally free over $\overline{\mathscr{H}{M}}_{g, \Gamma,  \eta}$.
\end{theorem}
\begin{proof}
It suffices to show that the coherent sheaf $\mathscr{V}_{\Sigma_X,  \Gamma,\phi }(  \vec{q}_X , \vec{\lambda})$ is locally free, where $X$ is an atlas of $ \overline{\mathscr{H}{M}}_{g, \Gamma,   \eta}$.   Since $X$ is a disjoint union of smooth irreducible schemes, we can work with a fixed component $X_\alpha$ of $X$, and show that the associated sheaf of conformal blocks restricted to $X_\alpha$ is locally free.

 We introduce a filtration on $\overline{\mathscr{H}{M}}_{g, \Gamma,  \eta}$:
\[  \overline{\mathscr{H}{M}}_{g, \Gamma,  \eta}^0\subset  \overline{\mathscr{H}{M}}_{g, \Gamma,  \eta}^1  \subset \cdots \subset \overline{\mathscr{H}{M}}_{g, \Gamma,  \eta}^k=\overline{\mathscr{H}{M}}_{g, \Gamma,  \eta},  \]
where $\overline{\mathscr{H}{M}}_{g, \Gamma,  \eta}^i$ is the open substack of $ \overline{\mathscr{H}{M}}_{g, \Gamma,  \eta}$ with each geometric fiber consisting of at most $i$ many $\Gamma$-orbits of nodal points.  Note that $\overline{\mathscr{H}{M}}_{g, \Gamma,  \eta}^0$  consists of stable smooth $s$-pointed $\Gamma$-curves.  With the restriction on the genus to be fixed $g$, there exists $k\geq 0$ such that  the number of orbits of nodal points is bounded by $k$.    This filtration induces an open filtration on $X_\alpha$ via $\delta$, 
\[  X_\alpha^0\subset X_\alpha^1\subset \cdots  \subset  X_\alpha^k=X_\alpha . \]

We now prove inductively that the coherent sheaf $\mathscr{V}_{\Sigma_{X_\alpha^i},  \Gamma,\phi }(  \vec{q}_{X_\alpha^i} , \vec{\lambda})$ is locally free, where $\vec{q}_{X_\alpha^i}$ is the restriction of $\vec{q}$ to  $X_\alpha^i$. When $i=0$, in view of Theorem \ref{Locally_free_smooth_base},  $\mathscr{V}_{\Sigma_{X_\alpha^0},  \Gamma,\phi }(  \vec{q}_{X_\alpha^0} , \vec{\lambda})$ is locally free.  (Observe that by \cite[Chap. III, Theorem 10.2]{H}, $\Sigma_{X_\alpha^0}\to X_\alpha^0$ is a smooth morphism.) Assume that $\mathscr{V}_{\Sigma_{X_\alpha^{i-1}},  \Gamma,\phi }(  \vec{q}_{X_\alpha^{i-1}} , \vec{\lambda})$ is locally free where $i\geq 1$. 
By the smoothing construction in Lemma \ref{smoothing_construction},  for any  $x\in X_\alpha^i\backslash X_\alpha^{i-1}$, there exists a morphism $\beta_x: \mathbb{D}_\tau  \to \overline{\mathscr{H}{M}}_{g, \Gamma,  \eta}^{i}  $ such that $\beta_x(o)=\delta(x)$ and $\beta_x ( g_\tau )\in \overline{\mathscr{H}{M}}_{g, \Gamma,  \eta}^{i-1}\backslash \overline{\mathscr{H}{M}}_{g, \Gamma,  \eta}^{i-2}$,  where $g_\tau$ is the generic point of $\mathbb{D}_\tau$. Recall that $\delta: X\to \overline{\mathscr{H}{M}}_{g, \Gamma,  \eta}$ is \'etale and surjective, hence $\beta_x$ can be lifted to  $\beta_x': \mathbb{D}_\tau\to X_\alpha$ such that $\delta\circ \beta_x'=\beta_x$ and $\beta_x'(o)=x$. It follows that $\beta_x'(g_\tau)\in X_\alpha^{i-1}\backslash X_\alpha^{i-2}$. 
  By Theorems \ref{covacua_coherent_basechange} and  \ref{key_theorem},  the rank of $\mathscr{V}_{\Sigma_{X_\alpha^{i-1}},  \Gamma,\phi }(  \vec{q}_{X_\alpha^{i-1}} , \vec{\lambda})$  (which is locally free by induction)  is equal to the dimension of $\mathscr{V}_{\Sigma_x,  \Gamma,\phi }(  \vec{q}_x , \vec{\lambda})$.  It follows that $\mathscr{V}_{\Sigma_{X_\alpha^{i}},  \Gamma,\phi }(  \vec{q}_{X_\alpha^{i}} , \vec{\lambda})$  
is also locally free (cf. \cite[Chap. II, Exercise 5.8(c)]{H}).   It concludes the proof of the theorem.
\end{proof}

The dimension of $\overline{\mathscr{H}{M}}_{g, \Gamma,  \eta}$ is equal to $3\bar{g}-3+s$ (cf. \cite[Theorem 5.1.5]{BR}) if it is not empty, where by Riemann-Hurwitz formula the genus $\bar{g}$ of $\Sigma/\Gamma$ for any $\Sigma\in \overline{\mathscr{H}{M}}_{g, \Gamma,  \eta}$, satisfies the following equation (cf. \cite[Chap. IV, Corollary 2.4]{H}):
\[2g-2= |\Gamma| (2\bar{g}-2) +\sum_{i=1}^s \frac{|\Gamma|}{|\Gamma_i|}( |\Gamma_i| -1) .\]
If $\dim \overline{\mathscr{H}{M}}_{g, \Gamma,  \eta}=0$, then we must have $\bar{g}=0$ and $s=3$.  
\begin{lemma}
\label{Existence_Lemma}
If $\dim \overline{\mathscr{H}{M}}_{g, \Gamma,  \eta}>0$, then any stable $s$-pointed curve $(\bar{\Sigma},\vec{p})$ of genus $\bar{g}$ and consisting of only one node, admits a stable $s$-pointed $\Gamma$-cover $(\Sigma, \vec{q})\in  \overline{\mathscr{H}{M}}_{g, \Gamma,  \eta}$.
\end{lemma}
\begin{proof}
By assumption,  $3\bar{g}-3 +s>0$. It follows that either $\bar{g}\geq 1$ or $s\geq 4$. Hence, we can always find a stable $s$-pointed curve $\bar{C}$ of genus $\bar{g}$ over $\mathbb{C}[[\tau]]$  such that the special fiber $\bar{C}_o$ has only one node, and the generic fiber $\bar{C}_{\mathcal{K}}$ is smooth, where $\mathcal{K}=\mathbb{C}((\tau))$. Let $\overline{\mathscr{M}}_{\bar{g},s }$ be the moduli stack of stable $s$-pointed curves of genus $\bar{g}$.  By \cite[Proposition 6.5.2 iii)]{BR}, the morphism $\overline{\mathscr{H}{M}}_{g, \Gamma,  \eta}\to \overline{\mathscr{M}}_{\bar{g},s }$ given by $(\Sigma, \vec{q})\mapsto (\bar{\Sigma}, {\vec{\bar{q}}})$,  where $\bar{\Sigma}=\Sigma/\Gamma$ and ${\vec{\bar{q}}}$ is the image of $\vec{q}$ in $\bar{\Sigma}$, is surjective. It follows that after a finite base change of $\mathcal{K}$,  $\bar{C}_{\mathcal{K}}$ has a Galois cover ${C}_{\mathcal{K}}$ with Galois group $\Gamma$ and the prescribed marking data.  By semi-stable reduction theorem (cf. \cite[Proposition 5.2.2]{BR}),  after another  finite base change of $\mathcal{K}$,  ${C}_{\mathcal{K}}$ can be uniquely extended to a stable $s$-pointed $\Gamma$-curve ${C}$ over $\mathbb{C}[[\tau]]$ (cf. \cite[Proposition 5.1.2]{BR}). Hence, the special fiber ${C}_o$ is a stable $s$-pointed $\Gamma$-curve, whose quotient ${C}_o/\Gamma$ is exactly the given stable $s$-pointed curve $\bar{C}_o$. It concludes the proof of the lemma.
\end{proof}

\begin{remark}\label{remark8.11}
{\rm (1) In the case $\Gamma$ is cyclic,  $\overline{\mathscr{H}{M}}_{g, \Gamma,  \eta}$ is irreducible (which can be deduced from the irreduciblity of  $\overline{\mathscr{H}M}_{g, \Gamma,[\xi], [\eta]}$ proved in  \cite[Corollary 6.4.3]{BR}).  Then, 
by Lemma \ref{Existence_Lemma}, Theorem \ref{Hurwitz_locally_free} and the Factorization Theorem (Theorem \ref{thm3.1.2}), one can see that to compute the dimension of the space of conformal blocks on smooth stable $s$-pointed $\Gamma$-curve, we are  reduced to considering the case: cyclic covers over $\mathbb{P}^1$ with $s=3$.  

\vskip1ex

(2)  For any $s$-pointed smooth $\Gamma$-curve $(\Sigma, \vec{q})$ and weights $\vec{\lambda}=(\lambda_1, \dots, \lambda_s)$ with $\lambda_i\in D_{c, q_i}$ , as in Definition \ref{def1.2}, we can attach the space of conformal blocks. Here we do not need to assume that $\cup_{i}\Gamma\cdot q_i$ contains all the ramified points of  $\Sigma \to \bar{\Sigma}$. Thanks to the Propagation Theorem (Corollary \ref{coro2.2.3}(a)), the dimension of the space of conformal blocks in this case can be reduced to the case that all the ramified points are contained in $\cup_{i}\Gamma\cdot q_i$ when $0\in D_{c, q}$ for any ramified point $q$ in $\Sigma$.  
\vskip1ex

(3)   The morphism $f: \overline{\mathscr{H}{M}}_{g, \Gamma,  \eta} \to \overline{\mathscr{M}}_{g, s'}$ given by mapping the stable $s$-pointed $\Gamma$-curve  $(\Sigma , \vec{q})$ to the stable $s'$-pointed curve $(\Sigma,  \cup_{i, \gamma\in \Gamma} \,\gamma\cdot q_i )$
(cf. Remark \ref{remark8.2}) (where $s'=\sum_{i=1}^s\frac{|\Gamma|}{|\Gamma_i|}$) is representable and finite (cf. \cite[Proposition 6.5.2]{BR} for the corresponding result for $\overline{\mathscr{H}M}_{g, \Gamma,[\xi], [\eta]}$; now composing this with the representable and finite morphism $\overline{\mathscr{H}M}_{g, \Gamma, \eta} \to \overline{\mathscr{H}M}_{g, \Gamma,[\xi], [\eta]}$, the assertion about $f$ follows). By taking the pushforward,  the locally free sheaf of conformal blocks on $\overline{\mathscr{H}{M}}_{g, \Gamma,  \eta}$ gives rise to many characteristic classes in the cohomology of $\overline{\mathscr{M}}_{g, s'}$.  It could give rise to  interesting applications. }

\end{remark}

\section{Connectedness of $\Mor_\Gamma(\Sigma^{*},G)$}
\label{thetafuctions}

{\it In this section as well as in Sections 10-12, we consider a group homomorphism $\phi: \Gamma\to {\rm Aut }(\fg)$ such that $\Gamma$ stabilizes a Borel subalgebra $\fb$ of $\fg$. Moreover,  $\Sigma$ denotes a smooth irreducible projective curve with a faithful action of $\Gamma$ with the projection $\pi:\Sigma\to \bar\Sigma :=\Sigma/\Gamma$ and $G$ the simply-connected simple algebraic group over $\mathbb{C}$ with Lie algebra $\mathfrak{g}$. Let $B$ be the Borel subgroup of $G$ with Lie algebra $\mathfrak{b}$.  Observe that in earlier sections, we did not require $\Sigma$ to be smooth unless explicitly stated. }

 We prove the connectedness of the ind-group ${\rm Mor}_{\Gamma}(\Sigma^*, G)$. In particular, we show that the twisted Grassmannian $X^q= G(\mathbb{D}^*_q )^{\Gamma_q}/ G(\mathbb{D}_q )^{\Gamma_q} $ is irreducible.

Let $\sAlg$ be the category of commutative algebras with identity over $\mathbb{C}$ (which are not necessarily finitely generated) and all $\mathbb{C}$-algebra homomorphisms between them.
\begin{definition}\label{appB-defi-B1}
A $\mathbb{C}$-{\it space functor} (resp. $\mathbb{C}$-{\it group functor}) is a covariant functor
$$
\mathscr{F}:\sAlg\to \Set\quad \text{(resp. $\mathscr{G}$roup)}
$$
which is a sheaf for the $\fppf$  (faithfully flat of finite presentation - Fid\`element Plat de Pr\'esentation Finie) topology, i.e., for any $R\in \sAlg$ and any  faithfully flat finitely presented $R$-algebra $R'$, the diagram
\begin{equation}\label{eqnb.1.1}
\mathscr{F}(R)\to \mathscr{F}(R')\rightrightarrows \mathscr{F}(\,\dprod{R'}{R}\,{R'})
\end{equation}
in exact, where $\Set$ (resp. $\mathscr{G}$roup) is the category of sets (resp. groups). In particular, $\mathscr{F}(R)\to \mathscr{F}(R')$ is one to one. 

{\it From now on we shall abbreviate  faithfully flat finitely presented $R$-algebra by $\fppf$ $R$-algebra.}

By a $\mathbb{C}$-{\it functor morphism} $\varphi: \mathscr{F}\to \mathscr{F}'$ between two  $\mathbb{C}$-space functors, we mean a set map 
 $\varphi_R: \mathscr{F}(R)\to \mathscr{F}' (R)$ for any  $R\in \sAlg$ such that the following diagram is commutative for any algebra homomorphism $R \to S$:
\[
\xymatrix{
\mathscr{F}(R)\ar[r]^-{\varphi_R}\ar[d]& \mathscr{F}' (R)\ar[d]\\
\mathscr{F}(S)\ar[r]^-{\varphi_S}& \mathscr{F}' (S).
}
\]
Direct limits exist in the category of $\mathbb{C}$-space ($\mathbb{C}$-group) functors.
For any ind-scheme $X=(X_{n})_{n\geq 0}$ over $\mathbb{C}$, the functor $\mathfrak{S}_{X}$  is a $\mathbb{C}$-space functor
by virtue of the Faithfully Flat Descent (cf. [G, VIII 5.1, 1.1 and 1.2]), where $\mathfrak{S}_{X} (R)$ is the set of all the morphisms $\Mor(\Spec R, X)$. This allows us to realize the category of ind-schemes over $\mathbb{C}$ as a full subcategory of the category of $\mathbb{C}$-space functors.
\end{definition} 
We recall the following well-known lemma (cf. \cite[Lemma B.2]{Ku2}).
\begin{lemma} \label{lemB.2} Let $\mathscr{F}^{o}:\sAlg\to \Set$ be a covariant functor. Assume that
\begin{equation}
\mathscr{F}^{o}(R)\to \mathscr{F}^{o}(R')\,\,\,\text{is one to one}\label{B.2-eq1}
\end{equation}
 for any $R\in \sAlg$ and any $\fppf$ $R$-algebra $R'$.

Then, there exists a $\mathbb{C}$-space functor $\mathscr{F}$ containing $\mathscr{F}^{o}$ (i.e., $\mathscr{F}^{o}(R)\subset \mathscr{F}(R)$ for any $R$)  such that for any $\mathbb{C}$-space functor $\mathscr{G}$ and  a natural transformation
$\theta^{o}:\mathscr{F}^{o}\to \mathscr{G}$, there exists a unique natural transformation $\theta:\mathscr{F}\to \mathscr{G}$ extending $\theta^{o}$.

Moreover, such a $\mathscr{F}$ is unique up to a unique isomorphism extending the identity map of $\mathscr{F}^{o}$.

We call such a $\mathscr{F}$ the $\fppf$-{\it sheafification} of $\mathscr{F}^{o}$. 

If $\mathscr{F}^{o}$ is a $\mathbb{C}$-group functor, then its $\fppf$-sheafification $\mathscr{F}$ is a $\mathbb{C}$-group functor.
\end{lemma}

We recall the following result communicated by Faltings.  A detailed proof (due to B. Conrad) can be found in \cite[Theorem 1.3.22]{Ku2}.

\begin{theorem}\label{chap1-them1.2.22}
Let $\mathscr{G}=(\mathscr{G}_{n})_{n\geq 0}$ be an ind-affine group scheme filtered by (affine) finite type  schemes over $\mathbb{C}$ and let $\mathscr{G}^{\red}=(\mathscr{G}^{\red}_{n})_{n\geq 0}$ be the associated reduced ind-affine group scheme. Assume that the canonical ind-group morphism $i:\mathscr{G}^{\red}\to \mathscr{G}$ induces an isomorphism $(di)_{e}:\Lie (\mathscr{G}^{\red})\xrightarrow{\sim}\Lie \mathscr{G}$ of the associated Lie algebras (cf. \cite[Corollary B.21]{Ku2}). Then, $i$ is an isomorphism of ind-groups, i.e., $\mathscr{G}$ is a reduced ind-scheme.
\end{theorem}
\begin{definition} \label{defi9.4} (Twisted affine Grassmannian) Recall that for any affine scheme $Y=\Spec S$ with the action of a group $H$, the closed subset $Y^H$ acquires a closed subscheme structure by taking 
$$Y^H:= \Spec \left(S/\langle g\cdot f-f\rangle_{g\in H, f\in S}\right),$$
where $\langle g\cdot f-f\rangle$ denotes the ideal generated by the collection $g\cdot f-f$. With this scheme structure, $Y^H$ represents the functor $R
\rightsquigarrow \Mor_H(\Spec R, Y).$

For any point $q\in \Sigma$, let  $\sigma_q$ be the generator of the stabilizer $\Gamma_q$ such that $\chi_q(\sigma_q) = \epsilon_q$, where
$\epsilon_q:=e^{\frac{2\pi i}{|\Gamma_q|}}$ (cf. Definition \ref{Gamma_curve}). 
Choose a formal parameter $z_q$ at $q$ in $\Sigma$ such that 
\begin{equation}\label{eqn9.3.2} \sigma_q\cdot z_q^{-1}= \epsilon_q z_q^{-1},\,\,\,\text{cf. the identity \eqref{3.3.new}}.
\end{equation}
Consider the functor 
$$R\rightsquigarrow G\left(R((z_q))\right)^{\Gamma_q}/G\left(R[[z_q]]\right)^{\Gamma_q}.$$
Its fppf-sheafification is denoted by  the functor $\mathscr{X}^q=\mathscr{X}(G,q, \Gamma)$. 

Recall that there exists an open subset $\mathbb{V}\subset G((z_q))$ (where $G((z_q)) :=G\left(\mathbb{C}((z_q))\right)$) such that the product map 
$$G\left(R[z_q^{-1}]\right)^-\times G\left(R[[z_q]]\right)\simeq \mathbb{V}(R)\,\,\,\text{is a bijection for any $R$},$$
where $G\left(R[z_q^{-1}]\right)^-$ is the kernel of $G\left(R[z_q^{-1}]\right)\to G(R), z_q^{-1}\mapsto 0$ (cf. \cite[Corollary 3]{Fa2} or \cite[Lemma 1.3.16]{Ku2}).  Moreover, the functor $G\left(R[z_q^{-1}]\right)^-$ is represented by an ind-group variety (in particular, reduced) structure on $G[z_q^{-1}]^-$ (cf. \cite[Corollary 3]{Fa2} and \cite[Corollary 1.3.3 and Theorem 1.3.23]{Ku2}). 
This gives rise to a bijection 
\begin{equation} \label {eqn9.4.1}
\left(G\left(R[z_q^{-1}]\right)^-\right)^{\Gamma_q}\times G\left(R[[z_q]]\right)^{\Gamma_q}\simeq \mathbb{V}(R)^{\Gamma_q}.
\end{equation}
Declare $\{g\mathbb{V}^{\Gamma_q}/G[[z_q]]^{\Gamma_q}\}_{g\in G((z_q))^{\Gamma_q}}$ as an open cover of 
$$X^q=X(G, q, \Gamma) :=G((z_q))^{\Gamma_q}/G[[z_q]]^{\Gamma_q}$$
and put the ind-scheme structure on $g\mathbb{V}^{\Gamma_q}/G[[z_q]]^{\Gamma_q}$ via its bijection
$$g\mathbb{V}^{\Gamma_q}/G[[z_q]]^{\Gamma_q}\simeq \left(G[z_q^{-1}]^-\right)^{\Gamma_q}\,\,\,\text{induced from the identification \eqref{eqn9.4.1} }$$
with the closed ind-subgroup scheme structure on $\left(G[z_q^{-1}]^-\right)^{\Gamma_q}$ coming from $G[z_q^{-1}]^-$. In particular, $\left(G[z_q^{-1}]^-\right)^{\Gamma_q}$ represents the functor $\left(G(R[z_q^{-1}])^-\right)^{\Gamma_q}$. Thus, we get an ind-scheme structure on $X^q$ such that the projection $ G((z_q))^{\Gamma_q}\to X^q$ admits local sections in the Zariski topology. Moreover, the injection 
$X^q\hookrightarrow G((z_q))/G[[z_q]]$ is a closed embedding. Further, $X^q$ represents the functor $\mathscr{X}^q$ since the ind-projective variety
 $G((z_q))/G[[z_q]]$ represents the fppf-sheafification of the functor $ G\left(R((z_q))\right)/G\left(R[[z_q]]\right)$. In particular,
\begin{equation} \label {eqn9.4.2} \mathscr{X}^q(\mathbb{C})=X^q.
\end{equation}
Let  $U$ (resp. $U^-$) be the unipotent radical of $B$ (resp. of the opposite Borel subgroup $B^-$). By considering the ind-subgroup schemes 
 $\left(U[z_q^{-1}]^-\right)^{\Gamma_q}$ and  $\left(U^-[z_q^{-1}]^-\right)^{\Gamma_q}$ of  $\left(G[z_q^{-1}]^-\right)^{\Gamma_q}$, it is easy to see that the Lie algebras
$$ \Lie \left(\left(G[z_q^{-1}]^-\right)^{\Gamma_q}\right)= \Lie \left(\left(\left(G[z_q^{-1}]^-\right)^{\Gamma_q}\right)_{\red}\right),$$
since the Lie subalgebras $ \left(\mathfrak{u}\otimes \bc[z_q^{-1}]^-\right)^{\Gamma_q}$ and  $ \left(\mathfrak{u}^-\otimes \bc[z_q^{-1}]^-\right)^{\Gamma_q}$ generate the Lie algebra  $ \left(\mathfrak{g}\otimes \bc[z_q^{-1}]^-\right)^{\Gamma_q}$ (cf. Section \ref{Kac_Moody_Section}), where $Y_{\red}$ denotes the corresponding reduced ind-subscheme. Thus,  $\left(G[z_q^{-1}]^-\right)^{\Gamma_q}$ is a (reduced) ind-group variety (cf. Theorem \ref{chap1-them1.2.22}). Hence, $X^q$ also is a (reduced) ind-projective variety. 

Observe that $X^q$ being reduced, it  is the {\it (twisted) affine Grassmannian}  considered in [Ku, $\S$7.1] based at $q$ corresponding to the twisted affine Lie algebra $\hat{\mathfrak{g}}_{\pi(q)}\simeq \hat{L}(\mathfrak{g}, \Gamma_q)$ (cf. Section \ref{Kac_Moody_Section} and Lemma \ref{lemma2.3}) and its parabolic subalgebra  
$\hat{L}(\mathfrak{g}, \Gamma_q)^{\geq 0}$. To prove this, follow the same argument as in [LS, Proof of Proposition 4.7] and the construction of the projective representation of $G((z_q))^{\Gamma_q}$  given by (subsequent) Theorem \ref{thm1.3.4}.
\end{definition}

Let $\bar{q} =\{q_{1},\ldots,q_{s}\}$ be a set of points of $\Sigma$ (for $s\geq 1$) with distinct $\Gamma$-orbits and let $\Sigma^{*}:=\Sigma\backslash \Gamma\cdot \bar{q}$. Recall that $\Xi=\Xi_{\bar{q}}:=\Mor_\Gamma(\Sigma^{*},G)$ is an  ind-affine group scheme, which is a closed ind-subgroup scheme of  $\Mor(\Sigma^{*},G)$ as $\Gamma$-fixed points. {\it We abbreviate $\Mor_\Gamma(\Sigma^{*},G)= G(\mathbb{C}[\Sigma^*])^\Gamma $ by $G(\Sigma^*)^\Gamma$.} Then, $\Xi$ represents the functor:
$$R\in \Alg\rightsquigarrow \Xi(R):=\Mor_\Gamma(\Sigma^{*}_{R},G),$$
 where $\Sigma^{*}_{R}:=\Sigma^{*}\times \Spec R$, with the trivial action of $\Gamma$ on $R$. This follows from the corresponding result for the functor $R \rightsquigarrow \Mor(\Sigma^{*}_{R},G)$ (without the $\Gamma$-action) which is represented by $G(\Sigma^{*})$ (cf. [Ku2, Lemma 5.2.10]).
 
 Let $\Xi^{\text{an}}$ denote the group $\Xi$ with the analytic topology. The following result in the non-equivariant case (i.e. $\Gamma =(1)$) is due to Drinfeld. We adapt his arguments (cf. \cite[Proof of Theorem 8.1.1]{Ku2}) along with the construction of the projective representation of $G((z_q))^{\Gamma_q}$ as in (subsequent) Theorem \ref{thm1.3.4}.
\begin{theorem}\label{thm8.1.1}
The group $\Xi^{\text{an}}$ is path-connected and hence $\Xi$ is irreducible.
\end{theorem}

\begin{proof}
Take any  points $q_{1}',\ldots, q'_n, q'_{n+1}\in \Sigma\setminus \Gamma\cdot \bar{q}$ with distinct $\Gamma$-orbits and set (for any $0\leq i\leq n+1$)
$$\Xi_{i}=\Xi_{\bar{q}\cup \{q_{1}',\ldots, q_{i}'\}}=G(\Sigma^{*}_{i})^\Gamma,\text{~~ where~~ } \Sigma^{*}_{i}:=\Sigma^{*}\backslash 
\Gamma\cdot \{q'_{1},\ldots, q'_{i}\}.
$$
Consider the functor
$$
\mathscr{F}^{\circ}:R\rightsquigarrow \Xi_{n+1}(R)/\Xi_{n}(R).
$$
It is easy to see that
\begin{equation}
\mathscr{F}^{\circ}(R)\hookrightarrow \mathscr{F}^{\circ}(R'),\quad\text{for any\ $\mathbb{C}$-algebras~} R\subset R'. \label{chap8-eq1}
\end{equation}
Let $\widehat{\Xi_{n+1}/\Xi_{n}}$ be the fppf-sheafification of $\mathscr{F}^{\circ}$ (cf. Lemma \ref{lemB.2}).

We claim that as the $\mathbb{C}$-space functors
\begin{equation}
\widehat{\Xi_{n+1}/\Xi_{n}}\simeq \mathscr{X}^{q},\label{chap8-eq2}
\end{equation}
where $q=q'_{n+1}$. Define the morphism
$$\Xi_{n+1}(R) \to  \mathscr{X}^{q}(R), \,\,\gamma \mapsto \gamma_q,$$
where $\gamma_q$ is the power series expansion of $\gamma$ at $q$ in the parameter $z_q$.
The above morphism clearly factors through 
$$\Xi_{n+1}(R)/\Xi_n(R) \to  \mathscr{X}^{q}(R)$$ and hence we get a morphism of $\mathbb{C}$-space functors
$$\hat{\theta}: \widehat{\Xi_{n+1}/\Xi_{n}}\to \mathscr{X}^{q}.$$
Conversely, we define a map $\hat{\psi}:\mathscr{X}^q\to \widehat{\Xi_{n+1}/\Xi_{n}}$ as follows. Fix $ R\in\sAlg$. Take $\gamma_R\in  G\left(R((z_q))\right)^{\Gamma_q}$. Let $\mathscr{G}=\mathscr{G}(\Sigma, \Gamma, \phi)$ be the parahoric Bruhat-Tits group scheme (cf. Definition \ref{defi11.1}). Then, by \cite[Proposition 4]{He}, $\gamma_R$ corresponds to a $\mathscr{G}$-torsor over $\bar{\Sigma}\times \Spec R$  together with a section  $\sigma_{R}$ over $(\bar{\Sigma}\backslash \pi(q))_{R}$ and  a section $\mu_R$ over $(\mathbb{D}_{\pi(q)})_R$ such that 
\begin{equation} \label{neweqn8.3.1} \mu_R= \sigma_R\cdot  \gamma_R,\,\,\,\text{over $(\mathbb{D}^\times_{q})_R$}.
\end{equation}
This is possible since  $\gamma_R$ extends uniquely to an element of $\left(G(\pi^{-1}\mathbb{D}^\times_{\pi(q)})^\Gamma\right)_R$ (cf. Definition 
\ref{Gamma_curve}). There exists an $R$-algebra $R'$ with $\Spec R'\to \Spec R$ an \'etale cover (in particular, $R'$ is a fppf $R$-algebra) such that the pull-back $\mathscr{G}$-torsor   $E_{\gamma_{R'}}$ over $\bar{\Sigma}_{R'}$  admits a  section $\theta_{R'}$ over $(\pi(\Sigma^*_n))_{R'}$ (cf. [He, Theorem 4]). (Observe that $G$ being simply-connected, generic $\mathscr{G}_{\bc(\Sigma)}$ is simply-connected.)
Define 
$$\theta_{R'}=\sigma_{R'}\cdot {\psi_{\theta_{R'}}}(\gamma_{R'}),\,\,\,\text{over $(\Sigma^*_{n+1})_{R'}$}, $$
where $\sigma_{R'}$ is the pull-back of the section $\sigma_R$ to $(\bar{\Sigma}\setminus \pi(q))_{R'}$ and $\gamma_{R'}$ denotes the image of $\gamma_R$ in $G\left(R'((z_q))\right)^{\Gamma_q}$. 
Now,  set $\hat{\psi}(\gamma_{R'})={\psi}_{\theta_{R'}}(\gamma_{R'})$ mod $\Xi_n(R').$  It is easy to see that $\hat{\psi}(\gamma_{R'})$ does not depend upon the choices of 
$\sigma_R, \mu_R$ and $\theta_{R'}$ satisfying the equation \eqref{neweqn8.3.1}. Moreover,  $\hat{\psi}$ factors through $G\left(R'[[z_q]]\right)^{\Gamma_q}$. Thus, we get a 
$\mathbb{C}$-functor morphism (still denoted by) $\hat{\psi}:  \mathscr{X}^{q} \to  \widehat{\Xi_{n+1}/\Xi_{n}}.$ Further, it is easy to see that   $\hat{\theta}$ and $\hat{\psi}$ are inverses of each other. This proves the assertion \eqref{chap8-eq2}. In particular, the functor  $\widehat{\Xi_{n+1}/\Xi_{n}}$ is also representable represented by its $\mathbb{C}$-points $ \widehat{\Xi_{n+1}/\Xi_{n}}(\mathbb{C})$. 
We abbreviate $\Xi_{i}(\mathbb{C})$ by $\Xi_{i}$. From the equation \eqref{chap8-eq2}, we see that
$$
\Xi_{n+1}/\Xi_{n}\hookrightarrow \widehat{\Xi_{n+1}/\Xi_{n}}(\mathbb{C})\simeq \mathscr{X}^q(\mathbb{C}).
$$
Moreover, from the above definition of $\hat{\psi}$ and the identity \eqref{eqn9.4.2}, 
$
\hat{\psi}:\mathscr{X}^q(\mathbb{C})\xrightarrow{\sim}\widehat{\Xi_{n+1}/\Xi_{n}}(\mathbb{C})$ lands inside $\Xi_{n+1}/\Xi_{n}.$
Thus, we get
\begin{equation}
\Xi_{n+1}/\Xi_{n}=\widehat{\Xi_{n+1}/\Xi_{n}}(\mathbb{C})\simeq \mathscr{X}^q(\mathbb{C}).\label{chap8-eq5}
\end{equation}
This identification gives rise to an ind-variety structure on $\Xi_{n+1}/\Xi_n$ transported from that of $X^q=\mathscr{X}^q(\mathbb{C})$. Moreover, with this ind-variety structure, $\Xi_{n+1}/\Xi_n$ represents the functor $\widehat{\Xi_{n+1}/\Xi_{n}}$. It is easy to see (by considering the corresponding map at $R$-points) that with this ind-variety structure on  $\Xi_{n+1}/\Xi_n$, the action map:
$$\Xi_{n+1}\times (\Xi_{n+1}/\Xi_n)\to  \Xi_{n+1}/\Xi_n $$ 
is a morphism of ind-schemes. 

For any morphism $f: \Spec R \to  \Xi_{n+1}/\Xi_n$, there exists an \'etale cover $\Spec S \to \Spec R$ such that the projection $\Xi_{n+1} \to 
 \Xi_{n+1}/\Xi_n$ splits over $\Spec S$.  From this it is easy to see that $\left(\Xi_{n+1}/\Xi_{n}\right)^{\an}$ has the quotient topology induced from $\Xi_{n+1}^{\an}$. Moreover, for any ind-variety $Y=(Y_n)_{n\geq 0}$, any compact subset of $Y^{\an}$ lies in some $Y_N$ (which is easy to verify). Thus, $\Xi_{n+1}^{\an}\to \left(\Xi_{n+1}/\Xi_{n}\right)^{\an}$ is a Serre fibration.
This gives rise to an exact sequence (cf. \cite[Chap.~7, \S2, Theorem 10]{Sp})
\begin{equation}
\pi_{1}((X^q)^{\text{an}})\to \pi_{0}(\Xi_n^{\text{an}})\to \pi_{0}(\Xi^{\text{an}}_{n+1})\to \pi_{0}((X^q)^{\text{an}}).\label{chap8-eq6}
\end{equation}
But,
\begin{equation} \label{neweqn8.1.1.8} \pi_{1}((X^q)^{\text{an}})=\pi_{0}((X^q)^{\text{an}})=0,
\end{equation}
 from the Bruhat decomposition (cf. [Ku, Proposition 7.4.16]). Thus, we get
\begin{equation}
\pi_{0}(\Xi_{n}^{\text{an}})\simeq \pi_{0}(\Xi_{n+1}^{\text{an}}).\label{chap8-eq7}
\end{equation}
Now, we are ready to prove the theorem. Take
$$
\sigma \in \Xi_{\bar{q}}:=\Mor_\Gamma (\Sigma^{*},G)=G(\mathbb{C}[\Sigma^{*}])^\Gamma\subset G(K)^\Gamma,
$$
where $K$ is the quotient field of $\mathbb{C}[\Sigma^{*}]$. Since $G$ is simply-connected, by the following lemma,  $G(K)^\Gamma$ is generated by subgroups  $U(K)^\Gamma$ and  $U^-(K)^\Gamma$, where (as before) $U$ (resp. $U^-$) is the unipotent radical of $B$ (resp. of the opposite Borel subgroup $B^-$).
Moreover, $U$ and $U^-$ being  unipotent groups and $K\supset \mathbb{C}$,  $U(K)^\Gamma\simeq \mathfrak{u}(K)^\Gamma$ under the exponential map
(and similarly for $U^-$). Thus, we can write 
$$
\sigma=\Exp (x_1)\ldots \Exp (x_{d}),\,\,\,\text{for some  $x_{i}\in  \mathfrak{u}(K)^\Gamma \cup  \mathfrak{u}^-(K)^\Gamma$}. $$
Thus, there exists a finite set $\bar{q}'=\{q'_{1},\ldots,q'_{n+1}\}\subset \Sigma^*$ with disjoint $\Gamma$-orbits such that  all  the poles of any $x_{i}$ (which means the poles of $f_i^j$ writing $x_i=\sum_j e^j\otimes f^j_i$ for  a basis $e^j$ of $\mathfrak{u}$ or $\mathfrak{u}^-$) are contained in $\Gamma\cdot \bar{q}'$. Thus, 
$
\sigma\in\Xi_{n+1}.
$
Consider the curve
$$
\hat{\sigma}:[0,1]\to \Xi^{\text{an}}_{n+1}, \,\,\, t\mapsto \Exp (tx_{1})\ldots \Exp (tx_{d})\,\,\,\text{joining $e$ to $\sigma$}.
$$
 Since
$$
\pi_{0}(\Xi_n^{\text{an}})\simeq \pi_{0}(\Xi^{\text{an}}_{n+1}),\quad \text{by \eqref{chap8-eq7}},
$$
we get that $e$ and $\sigma$ lie in the same path component of $\Xi^{\text{an}}$, thus $\Xi^{\text{an}}$ is path-connected. Using [Ku, Lemma 4.2.5] we get that $\Xi$ is irreducible.
\end{proof}

Our original proof of the following lemma was more direct (and involved). We thank Philippe Gille for pointing out the  following argument relying on results of Borel-Tits and Steinberg.
\begin{lemma}\label{uni} Let $G, \Sigma, \Gamma$ be as in the beginning of this section and let $K$ be the function field of $\Sigma$. Let  $U$ (resp. $U^-$) be the unipotent radical of $B$ (resp. of the opposite Borel subgroup $B^-$). Then, $G(K)^\Gamma $ is generated (as an abstract group)
by  $U(K)^\Gamma$ and  $U^-(K)^\Gamma$. 
\end{lemma}
\begin{proof} Denote $K_o=K^\Gamma$. Then, $G(K)^\Gamma $ can be considered as  a group scheme over $K_o$. Moreover, since $\Gamma$ stabilizes the Borel subgroup $B$ of $G(\mathbb{C})$, $G(K)^\Gamma $ is a quasi-split group scheme (over $K_o$). Moreover, $G(K\otimes_{K_o} \bar{K}_o)^\Gamma$ with the trivial action of $\Gamma$ on $\bar{K}_o$ can be identified with $G(\bar{K}_o)$ since $\Gamma$ acts faithfully on $K$, where $\bar{K}_o$ is the algebraic closure of $K$. Now, the  lemma follows from combining the results \cite[Proposition 6.2 and Remark 6.6]{BT} and  \cite[Lemma 64]{St2}.
\end{proof}
\begin{remark} {\rm The above lemma is also true (by the same proof) for $K$ replaced by $\mathbb{C}((z_q))$ and $\Gamma$ replaced by $\Gamma_q$. 
In particular, this gives another proof of $\pi_0\left((X^q)^{\text{an}}\right)=0$. }
\end{remark}

As a special case  of Theorem \ref{thm8.1.1}, we get the following. The connectedness of $X^q$ in a more general setting is obtained by Pappas-Rapoport [PR1, Theorem 0.1]. 

\begin{coro}\label{newcoro8.1.4}
With the notation as in   Definition  \ref{defi9.4},  the (twisted)  affine  Grassmannian $X^q$ is an irreducible ind-projective (reduced) variety.
\end{coro}
\begin{proof} Let $\Sigma = \mathbb{P}^1, \bar{q}=\{\infty, 0\}$, and the action of $\Gamma = \Gamma_q$ given as follows: Let  $\sigma_q$ be any generator of $\Gamma_q$ (of order $e_q :=|\Gamma_q|$). Define the action of $\Gamma_q$ on $ \mathbb{P}^1$ by setting
$$\sigma_q\cdot z=e^{2\pi i/e_q}z,\,\,\,\text{for any $z\in  \mathbb{P}^1$}.$$
  Consider the natural transformation between the functors 
$$ G\left(R[z, z^{-1}]\right)^{\Gamma_q} \to G\left(R((z))\right)^{\Gamma_q}/ G\left(R[[z]]\right)^{\Gamma_q}.$$
This gives rise to the morphism between the corresponding ind-schemes:
$$\theta:  G\left(\mathbb{C}[z, z^{-1}]\right)^{\Gamma_q} \to X^q=G((z))^{\Gamma_q}/G[[z]]^{\Gamma_q}.$$
From the isomorphism (cf. equation \eqref{chap8-eq5} of Theorem \ref{thm8.1.1}):
$$\Xi_{n+1}/\Xi_n\simeq X^q$$
applied to the above example of  $\Sigma = \mathbb{P}^1, \bar{q}=\{\infty, 0\}$ and the action of $\Gamma$ as above, we get that $\theta$ is surjective. 
 Since $ G\left(\mathbb{C}[z, z^{-1}]\right)^{\Gamma_q} $ is irreducible (by Theorem \ref{thm8.1.1}) and hence so is 
$X^q$. Observe that in the proof of Theorem \ref{thm8.1.1} we used the connectedness and simply-connectedness of $ (X^q)^{\text{an}}$; in particular, this corollary builds upon the  connectedness of  $ (X^q)^{\text{an}}$ to prove the stronger result. 
\end{proof}

\section{Central extension of twisted loop group and its splitting over $\Xi$}\label{sec8.2}
\label{section10}

We construct the central extensions of the twisted loop group $G(\mathbb{D}^*_q )^{\Gamma_q}$. We introduce the notion of `canonical' splitting and prove the existence of its canonical splitting over $\Xi :={\rm Mor}_{\Gamma}(\Sigma\setminus \Gamma\cdot q, G)$ when $c$ is divisible by $|\Gamma|$. The treatment in this section is parallel to the one in \cite[\S1.4]{Ku2}, where the corresponding theory is explained in the untwisted case.

We continue to have the same assumptions on $G, \Gamma$ and $\Sigma$ as in the beginning of Section \ref{thetafuctions}.
Fix any base point  $q\in \Sigma$ and let $\Sigma^{*}:=\Sigma\backslash \Gamma\cdot q$ and  $\Xi=\Xi_q:=\Mor_\Gamma (\Sigma^{*},G)$. Then, $\Xi$ is an irreducible ind-affine group scheme (cf. Theorem \ref{thm8.1.1}).
Let $z_q$ be a formal parameter on $\Sigma$ around $q$ satisfying the condition \eqref{eqn9.3.2}. This gives rise to a morphism 
$$
\Xi\hookrightarrow \mathscr{L}^q_G,
$$
obtained by taking the Laurent series expansion at $q$ (with respect to the parameter $z_q$ at $q$), where $\mathscr{L}^q_G:=G((z_q))^{\Gamma_q}$.

\begin{definition}[Adjoint action of $\mathscr{L}^q_G$]\label{defi1.3.2}
Define the $R$-linear {\em Adjoint action} of the group functor $\mathscr{L}^q_G(R):=G\left(R((z_q))\right)^{\Gamma_q}$ on the Lie-algebra functor $\hat{L}(\mathfrak{g}, \Gamma_q)(R):=\left(\mathfrak{g}\otimes R((z_q))\right)^{\Gamma_q}\oplus R.C$ (extending  $R$-linearly the bracket    in $\hat{L}(\mathfrak{g}, \Gamma_q)(R)$)  by:
$$
(\scrAd \gamma)(x\oplus sC)=\gamma x\gamma^{-1}+\left(s+\frac{1}{|\Gamma_q|} \,\displaystyle\mathop{\Res}\limits_{z_q=0}\langle \gamma^{-1}d\gamma, x\rangle\right) C, 
$$
for $\gamma\in \mathscr{L}^q_G(R)$, $x\in \left(\mathfrak{g}\otimes R((z_q))\right)^{\Gamma_q}$ and $s\in R$, where $\langle,\rangle$ is the $R((z_q))$-bilinear extension of the normalized invariant form on $\mathfrak{g}$ (normalized as in Section \ref{Kac_Moody_Section}) and taking an embedding $i:G\hookrightarrow \SL_{N}$ we view $G(R((z_q)))$ as a subgroup of $N\times N$ invertible matrices over the ring $R((z_q))$. 
From the functoriality of the conjugation, $\gamma x\gamma^{-1}\in \left(\mathfrak{g}\otimes R((z_q))\right)^{\Gamma_q}$ and it does not depend upon the choice of the embedding $i$. A similar remark applies to $\gamma^{-1}d\gamma$. Here $d\gamma$ for $\gamma=(\gamma_{ij})\in M_{N}(R((z_q)))$ denotes $d\gamma:=\left(\dfrac{d\gamma_{ij}}{dz_q}\right)$.

It is easy to check  that for any $\gamma\in  \mathscr{L}^q_G(R)$, $\scrAd\gamma:\hat{L}(\mathfrak{g}, \Gamma_q)(R)\to  \hat{L}(\mathfrak{g}, \Gamma_q)(R)$
 is a $R$-linear Lie algebra homomorphism. Moreover, for $\gamma_1, \gamma_2\in \mathscr{L}^q_G(R)$,
\begin{equation}
\scrAd(\gamma_{1}\gamma_{2})=\scrAd(\gamma_{1})\scrAd(\gamma_{2}).\label{defi-1.3.2-eq1}
\end{equation}

One easily sees that for any $\mathbb{C}$-algebra $R$ and $x\in \left(\mathfrak{g}\otimes R((z_q))\right)^{\Gamma_q}$, the derivative
\begin{equation}\label{eqn1.3.2.2}
\dot\scrAd (x)(y)=[x,y],\quad\text{for any $ y\in \hat{L}(\mathfrak{g}, \Gamma_q)(R)$}.
\end{equation}
\end{definition}

Let $\mathscr{H}(\lambda)$ be an integrable highest weight (irreducible) representation of $\hat{L}(\mathfrak{g}, \Gamma_q)$ (with central charge $c$). It clearly extends to a $R$-linear representation $\bar{\rho}_{R}$ of $\hat{L}(\mathfrak{g}, \Gamma_q) (R)$ in  $\mathscr{H}(\lambda)_{R}:=\mathscr{H}(\lambda)\otimes_{\mathbb{C}} R$.  A proof of the following result is parallel to the proof due to  Faltings in the untwisted case (cf. \,[BL, Lemma A.3]). 

\begin{proposition}\label{newprop10.10}
 {\it For any $R\in \sAlg$ and $\gamma\in G\left(R((z_q))\right)^{\Gamma_q}$, locally over $\Spec R$, there exists an $R$-linear automorphism $\hat{\rho}_R(\gamma)$ of $\mathscr{H}(\lambda)_{R}$ uniquely determined up to an invertible element of $R$ satisfying}
\begin{equation}\label{eqn1.3.3.1}
\hat{\rho}_R(\gamma)\bar{\rho}_{R}(x)\hat{\rho}_R(\gamma)^{-1}=\bar{\rho}_{R}(\scrAd (\gamma)\cdot x),\quad\text{for any $x\in \hat{L}(\mathfrak{g}, \Gamma_q)(R) $}.
\end{equation}
\end{proposition}
As a corollary of the above Proposition, we get the following.

\begin{theorem}\label{thm1.3.4}
There exists a homomorphism $\rho_R:G\left(R((z_q))\right)^{\Gamma_q}\to \mathscr{P}GL_{\mathscr{H}(\lambda)}(R)$ of group functors such that
\begin{equation}
\dot{\rho}=\dot{\rho}(\mathbb{C}):T_{1}(\mathscr{L}^q_G(\mathbb{C}))=\left(\mathfrak{g}\otimes \mathbb{C}((z_q))\right)^{\Gamma_q}\to \End_{\mathbb{C}}(\mathscr{H}(\lambda))/\mathbb{C}\cdot \Iid_{\mathscr{H}(\lambda)}\label{thm1.3.4-eq1}
\end{equation}
coincides with the projective representation $\mathscr{H}(\lambda)$ of $\left(\mathfrak{g}\otimes \mathbb{C}((z_q))\right)^{\Gamma_q}$.
\end{theorem}

\begin{definition}[Central extension]\label{defi1.3.5}
Let $0\in D_{c, q}$, where $D_{c, q}$ denotes $D_c$ for the twisted affine Lie algebra $\hat{L}(\mathfrak{g}, \Gamma_q)$ (cf. Lemma \ref{finite_set_weight_lem} and Corollary \ref{newcoroweight0}). By the above theorem, we have a homomorphism of group functors:
$$
\rho_R : \mathscr{L}^q_G(R)\to \mathscr{P}GL_{\mathscr{H}_c} (R),
$$
where $\mathscr{H}_c:=\mathscr{H}(0)$ with central charge $c$ for the twisted affine Lie algebra  $\hat{L}(\mathfrak{g}, \Gamma_q)$.
Also, there is a canonical homomorphism of group functors
$$
\pi_R:\mathscr{G}L_{\mathscr{H}_c}(R)\to \mathscr{P}GL_{\mathscr{H}_c} (R).
$$
From this we get the fiber product group functor $\hat{\mathscr{G}}^q_c$:
$$
\hat{\mathscr{G}}^q_c(R):=\fprod{\mathscr{L}^q_G(R)}{\mathscr{P}GL_{\mathscr{H}_c}(R)}{\mathscr{G}L_{\mathscr{H}_c}(R)}.
$$
By  definition, we get homomorphisms of group functors
$$
p_R:\hat{\mathscr{G}}^q_c (R)\to \mathscr{L}^q_G(R)\quad\text{and}\quad \hat{\rho}_R:\hat{\mathscr{G}}^q_c(R)\to \mathscr{G}L_{\mathscr{H}_c}(R)
$$
making the following diagram commutative:
\[
\xymatrix{
\hat{\mathscr{G}}^q_c(R)\ar[d]_{p_R}\ar[r]^-{\hat{\rho}_R} & \mathscr{G}L_{\mathscr{H}_c}(R)\ar[d]^{\pi_R}\\
\mathscr{L}^q_G(R)\ar[r]_-{\rho_R} & \mathscr{P}GL_{\mathscr{H}_c}(R)\,.
}
\]
The following is the {\it central extension} we are seeking:
\begin{equation}\label{eqn1.3.5.1}
1\to \mathbb{C}^*\to \hat{\mathscr{G}}^q_c\xrightarrow{p} \mathscr{L}^q_G\to 1,\,\,\,\text{where $\hat{\mathscr{G}}^q_c:=\hat{\mathscr{G}}^q_c (\mathbb{C})$}.
\end{equation}
It is easy to see that the Lie algebra  $\Lie (\hat{\mathscr{G}}^q_c(R)):=T_{1}(\hat{\mathscr{G}}^q_c)_{R}$ is identified with the fiber product Lie algebra:
$$
\hat{\mathfrak{g}}^q(R) = \fprod{\left(\mathfrak{g}\otimes R((z_q))\right)^{\Gamma_q}}{\End_{R}((\mathscr{H}_c)_{R})/R.\Iid}{\End_{R}((\mathscr{H}_c)_{R})},
$$
for any commutative $\mathbb{C}$-algebra $R$.
\end{definition}

\begin{lemma}\label{lem1.3.6}
The Lie algebra $\hat{\mathfrak{g}}^q:=\Lie \hat{\mathscr{G}}^q_c (\mathbb{C})$ can canonically be identified with the twisted affine Lie algebra $\hat{L}(\mathfrak{g}, \Gamma_q)$.
\end{lemma}

\begin{proof} Define
$$
\psi:\hat{L}(\mathfrak{g}, \Gamma_q)\to \hat{\mathfrak{g}}^q,\quad x+zC\mapsto (x,\bar{\rho}(x)+zc\Iid),\,\,\,\text{
for $ x\in \mathfrak{g}((z_q))^{\Gamma_q}$ and $z\in \mathbb{C}$}.
$$
From the definition of the bracket in $\hat{L}(\mathfrak{g}, \Gamma_q)$ and Theorem \ref{thm1.3.4}, $\psi$ is an isomorphism of Lie algebras.
\end{proof}

Combining Theorem \ref{thm1.3.4}, Definition \ref{defi1.3.5} and Lemma \ref{lem1.3.6}, we get the following.

\begin{coro}\label{coro1.3.7}
We have a homomorphism of group functors:
$$
\hat{\rho}:\hat{\mathscr{G}}^q_c\to \mathscr{G}L_{\mathscr{H}_c}
$$
such that its derivative at $R=\mathbb{C}$:
$$
\dot{\hat{\rho}} : \hat{\mathfrak{g}}^q\to \End_{\mathbb{C}}(\mathscr{H}_c)
$$
under the identification of Lemma \ref{lem1.3.6} coincides with the Lie algebra representation
$$
\bar{\rho}:\hat{L}(\mathfrak{g}, \Gamma_q)\to \End_{\mathbb{C}}(\mathscr{H}_c).
$$
Moreover, for any $\hat{\gamma}\in \hat{\mathscr{G}}^q_c(R)$ and $x\in \hat{\mathfrak{g}}^q(R)$,
\begin{equation}
\hat{\rho}_R(\hat{\gamma})\bar{\rho}_{R}(x)  \hat{\rho}_R(\hat{\gamma})^{-1} =\bar{\rho}_{R}(\scrAd(p_R(\hat{\gamma})) x),\,\,\,\text{as operators on $(\mathscr{H}_c)_R$}.  \label{thm1.3-proof-eq1}
\end{equation}
\end{coro}

\begin{theorem}
\label{Thm10.7}
(1) The central extension $p: \hat{\mathscr{G}}^q_c\to \mathscr{L}_G^q$ (as in Definition \ref{defi1.3.5}) splits over $G[[z_q]]^{\Gamma_q}$ for any $c\geq 1$ such that $0\in D_c =D_{c, q}$. Moreover, we can choose the splitting so that the corresponding tangent map is the identity via Lemma \ref{lem1.3.6}. 
\vskip1ex

(2) The above central extension splits over $\Xi$ if $c$ is a multiple of $|\Gamma|$. Moreover, we can choose the splitting so that the corresponding tangent map is the identity via Lemma \ref{lem1.3.6}.

(By Corollary \ref{newcoroweight0}, if $|\Gamma|$ divides $c$ then $0\in D_c$.)
\vskip1ex

We call the unique splitting satisfying the above property {\it canonical}. 
\end{theorem}
\begin{proof}
We first prove part (1) of the theorem.
By Proposition  \ref{newprop10.10} (using the fact, as in Section 2, that the annihilator of $ \fg[[z_q]]^{\Gamma_q}$ in $\mathscr{H}_c$ is exactly $\mathbb{C}v_+$), the map 
\[\rho: G((z_q))^{\Gamma_q}\to {\rm PGL}_{\mathscr{H}_c} \]
restricted to $G[[z_q]]^{\Gamma_q}$ lands inside ${\rm PGL}^+_{\mathscr{H}_c}$ consisting of those (projective) automorphisms which take the highest weight vector $v_+$ of $\mathscr{H}_c$ to $\mathbb{C}^\times  v_+$. Take the subgroup ${\rm GL}^+_{\mathscr{H}_c}$ consisting of 
those automorphisms which take $v_+\mapsto v_+$. Then, the map ${\rm GL}^+_{\mathscr{H}_c}\to {\rm PGL}^+_{\mathscr{H}_c}={\rm Im} ({\rm GL}^+_{\mathscr{H}_c})$ is an isomorphism providing the splitting of ${\rm GL}_{\mathscr{H}_c}\to {\rm PGL}_{\mathscr{H}_c}$ over ${\rm PGL}^+_{\mathscr{H}_c}$. Thus, the central extension $p: \hat{\mathscr{G}}^q_c\to \mathscr{L}_G^q$ splits over $G[[z_a]]^{\Gamma_q}$. Denote this splitting by $\sigma$.  

We next prove that $\dot{\sigma}$ (via Lemma \ref{lem1.3.6}) is the identity map:  Let
\[ \dot{\sigma}(x)=x+\lambda(x)C, \quad \text{ for } x\in \fg[[z_q]]^{\Gamma_q},    \]
where $\lambda: \fg[[z_q]]^{\Gamma_q}\to \mathbb{C}$ is a $\mathbb{C}$-linear map. Thus, for any $x\in \fg[[z_q]]^{\Gamma_q} $, 
\begin{equation}
\label{eq10.7.1}
\dot{ \hat{\rho} }\circ \dot{\sigma}(x)(v_+)=x\cdot v_++ \lambda(x)cv_+=\lambda(x)cv_+.
\end{equation}
 But, since $\hat{\rho}\left({\rm Im}(\sigma)\right)\subset {\rm GL}^+_{\mathscr{H} _c}$,  
\begin{equation}
\label{eq10.7.2}
\dot{ \hat{\rho} }\circ \dot{\sigma}(x)(v_+)=0,\,\,\, \text{for all $x\in  \fg[[z_q]]^{\Gamma_q}$} .
\end{equation}
Combining (\ref{eq10.7.1}) and (\ref{eq10.7.2}), we get $\lambda\equiv 0$. This proves that $\dot{\sigma}$ is the identity map. \\

We now prove part (2) of the theorem. Consider the embedding obtained via the restriction: 
\[   i_q:  \Xi =G(\Sigma\backslash \Gamma\cdot q)^\Gamma  \hookrightarrow G(\mathbb{D}_q^*)^{\Gamma_q} . \]
Also, consider the embedding 
\[j_q=\prod j_q^\gamma: G(\mathbb{D}_q^*) \hookrightarrow  \prod_{\gamma\in  \widehat{ \Gamma/\Gamma_q}} G(\mathbb{D}_{\gamma\cdot q} ^*), \]
where $j_q^\gamma: G(\mathbb{D}_q^*) \xrightarrow{\sim} G(\mathbb{D}_{\gamma\cdot q}^*)$ is defined by 
\[ j^\gamma_q(f)(\gamma z)=\gamma\cdot f(z), \quad \text{ for } \gamma\in \widehat{\Gamma/\Gamma_q}, z\in \mathbb{D}^*_q, \text{ and } f\in G(\mathbb{D}_q^*).    \]
Here $\widehat{ \Gamma/\Gamma_q}$ denotes a (fixed) set of coset representatives of the cosets $\Gamma/\Gamma_q$.  

Let $\bar{\mathscr{H}}_1$ denote the integrable highest weight module of highest weight 0 and central charge 1 of the untwisted affine Lie algebra $\hat{L}(\fg)$ based at $q$, i.e.,  the central extension of $\fg((z_q))$, where $z_q$ is a formal parameter for $\Sigma$ at $q$.  Identifying $G(\mathbb{D}_{\gamma\cdot q}^*)$ with $G(\mathbb{D}_q^*)$ via $j^\gamma_q$, we get a projective representation $\rho$ and the following commutative diagram: 

\begin{equation}
\label{diag10.7}
\xymatrix{
\tilde{\mathscr{G}}^q  \ar[r]^{\hat{j}_q }  \ar[d]^{p_q} & {\hat{\mathscr{G}}}^{\Gamma\cdot q}  \ar[r]^<<<<<<{\hat{\rho}} \ar[d]^{p_{\Gamma\cdot q}}  &   {\rm GL}(\bar{\mathscr{H}_1 } \otimes \cdots  \otimes   \bar{\mathscr{H} }_1 )  \ar[d]^{\pi}\\
G(\mathbb{D}_q^*)  \ar[r]^<<<<{j_q} &    \prod_{\gamma\in  \widehat{ \Gamma/\Gamma_q}} G(\mathbb{D}_{\gamma\cdot q} ^*)  \ar[r]^<<<{\rho} &    {\rm  PGL}(\bar{\mathscr{H}_1 } \otimes \cdots  \otimes   \bar{\mathscr{H} }_1 ),
}
\end{equation}
where we take $|\Gamma/\Gamma_q|$ copies of $\bar{\mathscr{H}}_1$ and $\hat{\rho}$ (resp. $\hat{j}_q$) is the pull-back of $ {\rm GL}(\bar{\mathscr{H}_1 } \otimes \cdots  \otimes   \bar{\mathscr{H} }_1 )$ (resp. ${\hat{\mathscr{G}}}^{\Gamma\cdot q}$)  induced from $\rho$ (resp. $j_q$). Since $\bar{\mathscr{H}}_1$ has central charge 1, it is easy to see that $\tilde{\mathscr{G}}^q$ is the central extension of $G(\mathbb{D}_q^*)$ corresponding to the central charge $= |\Gamma/\Gamma_q|$.   

Since the $\hat{L}(\fg, \Gamma_q)$-submodule of $\bar{\mathscr{H}}_1$ generated by the highest weight vector is of central charge $|\Gamma_q|$
(cf. equation \eqref{eq1.1.1.4}), we get that the restriction $p_{\Gamma_q}$ of the central extension $p_q: \tilde{\mathscr{G}}^q \to  G(\mathbb{D}_q^*) $ to $ G(\mathbb{D}_q^*)^{\Gamma_q} $ is the central extension corresponding to the central charge $|\Gamma|$. By the same proof as of \cite[Proposition 3.3]{So2} (see also \cite[Theorem 8.2.1]{Ku2}), the central extension $p_{\Gamma\cdot q}$ splits over $G(\Sigma\backslash  \Gamma\cdot q)$. 

 Now, any splitting $\sigma$ of $p_{\Gamma\cdot q}$ over $G(\Sigma\backslash \Gamma\cdot q)$ clearly induces a splitting $\hat{\sigma}$ of the central extension $p_q:  p_q^{-1}( G(\mathbb{D}_q^*)^{\Gamma_q}  )\to G(\mathbb{D}_q^*) ^{\Gamma_q}$ over $G(\Sigma\backslash \Gamma\cdot q)^\Gamma$.

Observe next that any splitting $\sigma$ of  $p_{\Gamma\cdot q}:  {\hat{\mathscr{G}}}^{\Gamma\cdot q} \to      \prod_{\gamma\in  \widehat{ \Gamma/\Gamma_q}} G(\mathbb{D}_{\gamma\cdot q} ^*) $
over $G(\Sigma\backslash \Gamma\cdot q)$ satisfies    $\dot{\sigma} ={\rm Id}$.
This follows trivially from the fact that 
\[  [\fg(\Sigma\backslash \Gamma\cdot q),  \fg(\Sigma \backslash  \Gamma\cdot q) ]= \fg(\Sigma \backslash  \Gamma\cdot q) . \]
Now, the induced splitting $\hat{\sigma}$ of the central extension $p_q:  p_q^{-1}( G(\mathbb{D}_q^*)^{\Gamma_q}  )\to G(\mathbb{D}_q^*) ^{\Gamma_q}$ over $G(\Sigma\backslash \Gamma\cdot q)^\Gamma$ clearly satisfies  $\dot{\hat{\sigma}}={\rm Id}$.
This proves the theorem. 
\end{proof}

\begin{remark}
\label{remark10.8}
{\rm Because of the possible existence of nontrivial characters of $G^{\Gamma_q}$ (resp. $G^{\Gamma_{q'}}$ for $q'\in \Sigma\backslash \Gamma\cdot q$), the splittings of $\hat{\mathscr{G}}^q_c\to \mathscr{L}^q_G$ over $G[[z_q]]^{\Gamma_q}$ (resp. $G(\Sigma\backslash \Gamma\cdot q)^\Gamma $) may not be unique.}
\end{remark}
By a theorem of Steinberg (cf.\,\cite{St1}), the fixed subgroup $G^{\sigma}$ is connected for any finite order automorphism $\sigma$ of $G$. 
\begin{proposition}
\label{Prop10.8}
{\it Let $\sigma$ be a finite order automorphism of $\fg$ of order $m$ and let $m$ divide $\bar{s}c$, where $\bar{s}$ is defined above Corollary \ref{newcoroweight0}. For any $\lambda\in D_c$, the irreducible $\fg^\sigma$-module $V(\lambda)$ integrates to a representation of $G^\sigma$. }
\end{proposition}
\begin{proof}
Decompose $\sigma=\tau \epsilon^{{\rm ad} h}$ as in (\ref{eq1.1.1.0}). To prove that $V(\lambda)$ integrates to a $G^\sigma$-module,  it suffices to show that the torus $H^\sigma=H^\tau$ acts on $V(\lambda)$, where $H$ is the maximal torus of $G$ with Lie algebra $\fh$ ($\fh$ being a $\sigma$-stable  Cartan subalgebra). By Lemma \ref{finite_set_weight_lem}, since $m$ divides $\bar{s} c$ (by assumption), $\lambda(\alpha^\vee_i)\in \mathbb{Z}$ for all the simple coroots $\alpha^\vee_i$ of $\fg^\tau$, i.e., $\lambda$ belongs to the weight lattice of $\fg^\tau$. So, if $G^\tau$ is simply-connected, $\lambda$ gives rise to a character of $H^\tau$. Thus, in this case $H^\tau$ acts on $V(\lambda)$. Recall that for a diagram automorphism $\tau$, $G^\tau$ is simply-connected unless $(\fg,r)=(A_{2n}, 2)$, where $r$ is the order of $\tau$. In this case $G^\tau =$ SO$(2n+1)$ and following the notation of the identity \eqref{integer_n_mu},
\[ [x_o,y_o]=-(\alpha_1^\vee+ \cdots + \alpha_{n-1}^\vee+ \alpha^\vee_n/2)  \]
with the Bourbaki convention \cite[Planche II]{Bo}. Since $n_{\lambda,i}$ is required to lie in $\mathbb{Z}_{\geq 0}$, for all $i\in \hat{I}(\fg, \sigma)$, and $m$ divides $\bar{s} c$, we get 
\[ \lambda(\alpha^\vee_n)/2 \in \mathbb{Z}. \]
Thus, $\lambda$ belongs to the root lattice of $\fg^\tau$ and hence $\lambda$ gives rise to a character of $H^\tau$. This proves that $H^\tau$ acts on $V(\lambda)$, proving the proposition.
\end{proof}

\section{Uniformization theorem (A review)}
\label{section11}
We continue to have the same assumptions on $G, \Gamma, \Sigma$ as in the beginning of Section \ref{thetafuctions}.  We recall some results due to Heinloth \cite{He} (conjectured by Pappas-Rapoport [PR1, PR2]) only in the generality we need and in the form suitable for our purposes. In particular, we recall the uniformization theorem due to Heinloth for the  parahoric Bruhat-Tits group schemes $\mathscr{G}$ in our setting. We introduce the moduli stack $\mathscr{P}arbun_{\mathscr{G}}$ of quasi-parabolic $\mathscr{G}$-torsors over $\bar{\Sigma}$ and construct the line bundles over 
  $\mathscr{P}arbun_{\mathscr{G}}$.

\begin{definition}[Parahoric Bruhat-Tits group scheme]\label{defi11.1}
Consider the $\Gamma$-invariant Weil restriction $\mathscr{G}=\mathscr{G}(\Sigma, \Gamma, \phi)$ via $\pi:\Sigma \to \bar{\Sigma}:=\Sigma/\Gamma$ of the constant group scheme $\Sigma\times G \to \Sigma$ over $\Sigma$. More precisely, $\mathscr{G}$ is given by the following group functor over $\bar{\Sigma}$:
\[ U \leadsto  G(U\times_{\bar{\Sigma}} \Sigma )^\Gamma,    \]
for any scheme $U$ over $\bar{\Sigma}$, where $U\times_{\bar{\Sigma}} \Sigma $ is the fiber product of $U$ and $\Sigma$ over $\bar{\Sigma}$.   Then, $\mathscr{G}\to \bar{\Sigma}$ is a smooth affine group scheme over $\bar{\Sigma}$. 

 This provides a class of examples of {\it parahoric Bruhat-Tits group schemes.}
\end{definition}
  For any point $p\in \bar{\Sigma}$, the fiber $\mathscr{G}_p\simeq G$ if $p$ is an unramified point. However, if $p$ is a ramified point, the group $\mathscr{G}_p$ has unipotent radical $U_p$ and 
\begin{equation}
\label{def11.1.1}
  \mathscr{G}_p/U_p \simeq G^{\Gamma_q}, \text{ for any } q\in \pi^{-1}(p). 
\end{equation}
Take any point $q\in \pi^{-1}(p)$ and let $\mathbb{D}_p\subset \bar{\Sigma}$ (resp. $\mathbb{D}_q\subset \Sigma$) be the formal disc around $p$ in $\bar{\Sigma}$ (resp. around $q$ in $\Sigma$). Then,
\begin{equation}
\label{def11.1.2}
\mathscr{G}(\mathbb{D}_p )\simeq G(\mathbb{D}_q)^{\Gamma_q}.
\end{equation}
Similarly, for the punctured discs $\mathbb{D}^\times_p$ and $\mathbb{D}^\times_q$, 
\begin{equation}
\label{def11.1.3}
\mathscr{G}(\mathbb{D}^\times_p)\simeq G(\mathbb{D}^\times_q)^{\Gamma_q}.
\end{equation}
Thus, 
\[ \mathscr{G}(\mathbb{D}^\times_p)/\mathscr{G}(\mathbb{D}_p) \simeq X^q, \quad  (\text{cf. Definition 9.4} ). \]
In particular, it is also an {\it irreducible} (reduced) ind-projective variety (cf.\,Corollary \ref{newcoro8.1.4}).

\begin{definition}[Moduli stack of $\mathscr{G}$-torsors] 
\label{def11.2}
Consider the stack 
$ \mathscr{B}un_{\mathscr{G}}$  assigning to a commutative $\bc$-algebra  $R$  the category of $\mathscr{G}_R$-torsors over $\bar{\Sigma}_R:=\bar{\Sigma}\times {\rm Spec} R$, where $\mathscr{G}_R$ is the pull-back of $\mathscr{G}$ via the projection from $\bar{\Sigma}_R$ to $\bar{\Sigma}$.   Then, as proved by Heinloth \cite[Proposition 1]{He}, $\mathscr{B}un_{\mathscr{G}}$ is a smooth algebraic stack, which is locally of finite type.  

We need the following parabolic generalization of $\mathscr{B}un_{\mathscr{G}}$.   Let $\vec{p}=(p_1,\cdots, p_s)$ ($s\geq 1$) be a set of distinct points in $\bar{\Sigma}$. Label the points $\vec{p}$ by  parabolic subgroups $\vec{P}=(P_1,\cdots, P_s)$, where $P_i$ is a parabolic subgroup of $\mathscr{G}_{p_i}$. Via the isomorphism (\ref{def11.1.1}), we can think of $P_i$ as a parabolic subgroup $P_i^{q_i}$ of $G^{\Gamma_{q_i}}$ for any $q_i\in \pi^{-1}(p_i)$.  

A quasi-parabolic $\mathscr{G}$-torsor of type $\vec{P}$ over $(\bar{\Sigma}, \vec{p})$ is, by definition, a $\mathscr{G}$-torsor $\mathscr{E}$ over $\bar{\Sigma}$ together with points $\sigma_i$ in $\mathscr{E}_{p_i}/P_i$. This gives rise to the stack:   $R\leadsto $ the category of $\mathscr{G}_R$-torsors $\mathscr{E}_R$ over $\bar{\Sigma}_R$ together with sections $\sigma_i$ of $({\mathscr{E}_R}_{|_{\{p_i\}\times {\rm Spec} R }})/P_i\to {\rm Spec} R$.  We denote this stack by $\mathscr{P}arbun_{\mathscr{G}}=\mathscr{P}arbun_{\mathscr{G}}(\vec{P})$.
\end{definition}

We recall the following uniformization theorem.  It was proved by Heinloth \cite[Theorem 4, Proposition 4, and Theorem 5]{He} (and conjectured by Pappas-Rapoport [PR2])  for $\mathscr{B}un_{\mathscr{G}}$ (in fact, he proved a more general result). Its extension to $\mathscr{P}arbun_{\mathscr{G}}$ follows by the same proof. (Since $G$ is simply-connected, it is easy to see that so is the generic $\mathscr{G}_{\bc(\Sigma)}$.)

\begin{theorem}
\label{Thm11.3}
Take any $q_i\in \pi^{-1}(p_i)$, $q\in \Sigma\backslash \pi^{-1}\{p_1, \cdots, p_s\}$ and any parabolic type $\vec{P}$ at the points $\vec{p}$. Then, as stacks, 
\begin{equation}
\mathscr{P}arbun_{\mathscr{G}}(\vec{P})\simeq  \big [ G(\Sigma\backslash  \Gamma\cdot q)^\Gamma \big\backslash  \left(X^q \times  \prod_{i=1}^s  (G^{\Gamma_{q_i}} /P_i^{q_i} ) \right) \big],
\end{equation}
where $G(\Sigma\backslash  \Gamma\cdot q)^\Gamma$ acts on $X^q$ via its restriction to $\mathbb{D}_q^*$ and it acts on $G^{\Gamma_{q_i}} /P_i^{q_i} $ via its evaluation at $q_i$. Here   $ [ G(\Sigma \backslash  \Gamma\cdot q)^\Gamma \big \backslash  \left(X^q\times \prod_{i=1}^s  (G^{\Gamma_{q_i}} /P_i^{q_i} ) \right) ]$ denotes the quotient stack 
(cf. \cite[Example C.18(b)]{Ku2}) obtained by taking the quotient of the projective ind-variety  $X^q\times \prod_{i=1}^s  (G^{\Gamma_{q_i}} /P_i^{q_i} )$ by the ind-group $G(\Sigma\backslash  \Gamma\cdot q)^\Gamma$.

Moreover, the projection $X^q\times \prod_{i=1}^s  (G^{\Gamma_{q_i}} /P_i^{q_i} )\to   \mathscr{P}arbun_{\mathscr{G}}(\vec{P})$ is locally trivial in the smooth topology.    
\end{theorem}

\begin{remark}
{\rm Even though we will not use, there is also an isomorphism of stacks:
\[   \mathscr{P}arbun_{\mathscr{G}}(\vec{P})\simeq   \big [ G(\Sigma\backslash  \Gamma\cdot \vec{q})^\Gamma \big \backslash \left(\prod_{i=1}^s  ( X^{q_i}(P_i^{q_i}) ) \right) \big]  ,   \]
where $\Gamma\cdot \vec{q}:= \cup_{i=1}^s  \Gamma\cdot q_i$ and $X^{q_i}(P_i^{q_i}) $ is the partial twisted affine flag variety which is by definition $G(\mathbb{D}^\times_{q_i} )^{\Gamma_{q_i}} /\mathscr{P}_i $ and $\mathscr{P}_i $ is the inverse image of $P_i^{q_i}$ under the surjective evaluation map  $G(\mathbb{D}_{q_i} )^{\Gamma_{q_i}} \to  G^{\Gamma_{q_i}}$.}
\end{remark}
\begin{coro}
\label{coro11.5}
The ind-group scheme $G(\Sigma\backslash \Gamma\cdot q)^{\Gamma}$ is reduced. Moreover, it is irreducible by Theorem \ref{thm8.1.1}. 
\end{coro}
\begin{proof} We follow the same argument as in [LS, $\S$5]. 
Consider the projection 
\[ \beta:   X^q\times \prod_{i=1}^s  (G^{\Gamma_{q_i}} /P_i^{q_i} )\to   \mathscr{P}arbun_{\mathscr{G}}(\vec{P}) .  \]
By Theorem \ref{Thm11.3}, there exists a neighborhood $U\to   \mathscr{P}arbun_{\mathscr{G}}(\vec{P})   $  in the smooth topology such that $\beta^*(U)\simeq U\times G(\Sigma\backslash \Gamma\cdot q)^{\Gamma} $. Since $X^q$ is reduced by Corollary \ref{newcoro8.1.4} and of course $G^{\Gamma_{q_i}} /P_i^{q_i}$ are reduced, we get that $G(\Sigma\backslash \Gamma\cdot q)^{\Gamma} $ is reduced.
\end{proof}

\begin{definition}[Line bundles on $ \mathscr{P}arbun_{\mathscr{G}}$] 
\label{def11.6}
Let $\Xi:=G(\Sigma\backslash \Gamma\cdot q)^{\Gamma} $ and assume that $0\in D_{c, q}$. Recall (cf., e.g., \cite[\S 7.1]{BL}) that, by virtue of Theorem \ref{Thm11.3}, 
\begin{equation}
\label{def11.6.1}
{\rm Pic} \left(\mathscr{P}arbun_{\mathscr{G}}(\vec{P}) \right)  \simeq  {\rm Pic}^{\Xi} \left(X^q\times \prod_{i=1}^s  (G^{\Gamma_{q_i}} /P_i^{q_i} ) \right).
\end{equation}
Moreover, 
 from the {\it see-saw principle} (also see \cite[Chap. III, Exercise 12.6]{H}) , since $X^q$ is ind-projective and $\rm Pic$ of each factor is discrete,
 $$ {\rm Pic} \left( X^q  \times   \prod_{i=1}^s  (G^{\Gamma_{q_i}} /P_i^{q_i} )\right) \simeq {\rm Pic} ( X^q) \times   \prod_{i=1}^s {\rm Pic} (G^{\Gamma_{q_i}} /P_i^{q_i} ).
 $$
 Let us consider  the following canonical homomorphism:
 $$
 {\rm Pic}^{\Xi}( X^q) \times   \prod_{i=1}^s {\rm Pic}^{\Xi}  (G^{\Gamma_{q_i}} /P_i^{q_i} ) \to {\rm Pic}^{\Xi} \left( X^q\times \prod_{i=1}^s  (G^{\Gamma_{q_i}} /P_i^{q_i} ) \right).$$
Consider the morphism 
\[ \hat{\mathscr{G} }^q_c\to \mathscr{H}_c\backslash \{0\}, \quad  g\mapsto gv_+, \]
where $v_+$ is a highest weight vector of $\mathscr{H}_c$. This factors through a morphism (via Theorem \ref{Thm10.7} (1)):
\[  X^q=\hat{\mathscr{G}^q_c }/ (G(\mathbb{D}_q )^{\Gamma_q} \times \mathbb{C}^\times) \to  \mathbb{P}(\mathscr{H}_c ) .\]

Pulling back the dual of the tautological line bundle on $\mathbb{P}(\mathscr{H}_c )$, we get a $\hat{\mathscr{G} }^q_c$-equivariant line bundle $\mathfrak{L}^q_c $ on $X^q$ given by the character 
\[   G(\mathbb{D}_q )^{\Gamma_q} \times \mathbb{C}^\times\to  \mathbb{C}^\times, \quad  (g,z)\mapsto z.   \]
Observe that the canonical splitting of $G(\mathbb{D}_q)^{\Gamma_q}$ is taken for the central extension $\hat{\mathscr{G} }^q_c$ corresponding to the central charge $c$.  

Now, if $c$ is a multiple of $|\Gamma|$, the central extension $\hat{\mathscr{G} }^q_c\to G(\mathbb{D}_q^\times)^\Gamma$ splits over $\Xi$. As in Theorem
\ref{Thm10.7} (2), take the canonical splitting. This provides a $\Xi$-equivariant structure on the line bundle $\mathfrak{L}^q_c$ over $X^q$.   

 Similarly, for any $\lambda_i\in D_{c,q_i}$, the $\fg^{\Gamma_{q_i}}$-module $V(\lambda_i)$ with highest weight $\lambda_i$ integrates to a $G^{\Gamma_{q_i}}$-module $V(\lambda_i)$ if $|\Gamma|$ divides $c$ (cf.\,Proposition \ref{Prop10.8}). Take the highest weight vector $v_+\in \mathscr{H}(\lambda_i) $ which is an (irreducible) integrable highest weight $\hat{L}(\fg, \Gamma_{q_i})=\hat{\fg}_{q_i}$-module with highest weight  $\lambda_i$ and central charge $c$. Then, $V(\lambda_i)$ is the $G^{\Gamma_{q_i}}$-submodule of $\mathscr{H}(\lambda_i)$ generated by $v_+$. Let $P_i^{q_i}$ be the parabolic subgroup of $G^{\Gamma_{q_i}}$ which stabilizes the line $\mathbb{C}v_+$. Define the $G^{\Gamma_{q_i}}$-equivariant ample line bundle 
\[  \mathscr{L}^{q_i}(\lambda_i):=G^{\Gamma_{q_i}} \times_{P_i^{q_i}}   (\mathbb{C} v_+ )^* \to  G^{\Gamma_{q_i}}/ P_i^{q_i} .    \]

Then,  $\mathscr{L}^{q_i}(\lambda_i) $ is $\Xi$-equivariant line bundle by virtue of the following evaluation map at $q_i$:
\[ e_i:   \Xi:= G(\Sigma\backslash \Gamma \cdot q) ^\Gamma\to   G^{\Gamma_{q_i}} .  \]
Thus, we obtain the $\Xi$-equivariant line bundle 
\[  \mathfrak{L}^q_c \boxtimes \mathscr{L}^{q_1}(\lambda_1) \boxtimes \cdots  \boxtimes  \mathscr{L}^{q_s}(\lambda_s)   \]
over $X^q\times  \prod_{i=1}^s   (G^{\Gamma_{q_i}}/P_i^{q_i})$, for any $c$ divisible by $|\Gamma|$ and $\lambda_i\in D_{c, q_i}$.  

Thus, under the isomorphism (\ref{def11.6.1}), we get the corresponding  line bundle $\mathfrak{L}(c; \vec{\lambda})$ over the stack $ \mathscr{P}arbun_{\mathscr{G}}(\vec{P}) $,   where $\vec{P}=(P^{q_1}_1,\cdots, P^{q_i}_s)$ and  $P_i^{q_i}$ is the stabilizer in $G^{\Gamma_{q_i}}$ of the line $\mathbb{C}\cdot v_+\subset \mathscr{H}(\lambda_i)$.
\end{definition}

\section{Identification of  twisted conformal blocks with the space of global sections of line bundles on moduli stack}
\label{section12}
In this final section, we establish the identification of twisted conformal blocks and generalized theta functions on the moduli stack $\mathscr{P}arbun_{\mathscr{G}}$.

We continue to have the same assumptions on $G, \Gamma, \Sigma$ and $\phi: \Gamma\to {\rm Aut}(\fg)$ as in the beginning of Section \ref{thetafuctions}. Let $\vec{q}=(q_1,\cdots, q_s)$  ($s \geq 1$) be marked points on $\Sigma$ with distinct $\Gamma$-orbits and let $\vec{\lambda}=(\lambda_1,\cdots, \lambda_s)$ be weights with $\lambda_i\in D_{c,q_i}$ attached to the points $q_i$. Let $P^{q_i}_i$ be the stabilizer  of the line $\mathbb{C}v_+\subset \mathscr{H}(\lambda_i)$ in $ G^{\Gamma_{q_i}}$. 

Recall the definition of the moduli stack $\mathscr{P}arbun_{\mathscr{G}}(\vec{P})$ of quasi-parabolic $\mathscr{G}$-torsors over $(\bar{\Sigma}, \vec{p})$ of type $\vec{P}=(P_1^{q_1}, \cdots, P_s^{q_s})$ from Definition \ref{def11.2}, where $\vec{p}=(\pi(q_1), \cdots, \pi(q_s))$. Also, recall from  Definition \ref{def11.6} the definition of the line bundle $\mathfrak{L}^q_c$ over $X^q$ for any $c$ such that  $0\in D_{c,q}$, and any $q\in \Sigma \backslash  \cup_{i=1}^{s}\Gamma\cdot q_i$ and the definition of the  ample homogeneous  line bundle $\mathscr{L}^{q_i}(\lambda_i)$ over the flag variety $G^{\Gamma_{q_i}}/P_i^{q_i}$. When $|\Gamma|$ divides $c$, these line bundles give rise to a line bundle $\mathfrak{L}(c;\vec{\lambda})$ over the stack $\mathscr{P}arbun_{\mathscr{G}}(\vec{P})$ (cf. Definition \ref{def11.6}).  
\vskip1ex

The following result confirms a conjecture by Pappas-Rapoport [PR2, Conjecture 3.7] in the case of the parahoric Bruhat-Tits group schemes considered in our paper. 
\begin{theorem}
\label{thm12.1}
Assume that $|\Gamma|$ divides $c$ and $\Gamma$ stabilizes a Borel subgroup of $G$. Then, there is a canonical isomorphism: 
\[  H^0( \mathscr{P}arbun_{\mathscr{G}}(\vec{P}), \mathfrak{L}(c, \vec{\lambda})   ) \simeq   \mathscr{V}_{\Sigma, \Gamma, \phi}(\vec{p}, \vec{\lambda} )^{\dagger},       \]
where $\mathscr{V}_{\Sigma, \Gamma, \phi}(\vec{p}, \vec{\lambda})^\dagger $ is the space of (twisted) vacua (cf. Identity (\ref{eq7})).  
\end{theorem}
\begin{proof}
From the uniformization theorem (Theorem \ref{Thm11.3}), there is an isomorphism of stacks: 
\[  \mathscr{P}arbun_{\mathscr{G}}(\vec{P})\simeq   \big[ G(\Sigma\backslash  \Gamma\cdot q)^\Gamma \big \backslash  \left(X^q \times  \prod_{i=1}^s  (G^{\Gamma_{q_i}} /P_i^{q_i} ) \right) \big]  .\]
Moreover, by definition, the line bundle $\mathfrak{L}(c,\vec{\lambda})$ over $ \mathscr{P}arbun_{\mathscr{G}}(\vec{P})$ is the descent of the line bundle 
 \[  \mathfrak{L}^q_c \boxtimes \mathscr{L}^{q_1}(\lambda_1) \boxtimes \cdots  \boxtimes  \mathscr{L}^{q_s}(\lambda_s)   \]
over $X^q\times \prod_{i=1}^s  (G^{\Gamma_{q_i}} /P_i^{q_i} )$ (Definition \ref{def11.6}). Thus, we have the following isomorphisms:
\begin{align*}
H^0( \mathscr{P}arbun_{\mathscr{G}}(\vec{P}), \mathfrak{L}(c, \vec{\lambda})   )   &\simeq     H^0\left( X^q\times \prod_{i=1}^s  (G^{\Gamma_{q_i}} /P_i^{q_i} ),   \mathfrak{L}^q_c \boxtimes \mathscr{L}^{q_1}(\lambda_1) \boxtimes \cdots  \boxtimes  \mathscr{L}^{q_s}(\lambda_s)   \right)^{G(\Sigma\backslash \Gamma\cdot q)^\Gamma}\\
    &\simeq  \left(\mathscr{H}_c^*\otimes V(\lambda_1)^*\otimes \cdots   \otimes    V(\lambda_s)^*  \right)^{  G(\Sigma\backslash \Gamma\cdot q)^\Gamma }\\
    &\simeq   \left(\mathscr{H}_c^*\otimes V(\lambda_1)^*\otimes \cdots   \otimes    V(\lambda_s)^*  \right)^{  \fg(\Sigma\backslash \Gamma\cdot q)^\Gamma }\\
    &\simeq     \mathscr{V}_{\Sigma, \Gamma, \phi}(\vec{p}, \vec{\lambda} )^{\dagger},
\end{align*}  
where the first isomorphism follows from \cite[Lemma 7.2]{BL} (also see \cite[Proposition C.23]{Ku2});  the second isomorphism follows from the standard Borel-Weil theorem and its generalization for the Kac-Moody case due to Kumar as well as Mathieu \cite[Corollary 8.3.12]{Kbook};  the third isomorphism follows from \cite[Proposition 7.4]{BL} since $G(\Sigma\backslash \Gamma\cdot q)^\Gamma$ is reduced and irreducible (Corollary \ref{coro11.5}) and $X^q$ is reduced and irreducible by Corollary \ref{newcoro8.1.4}; and the last isomorphism follows from Propagation of Vacua (Corollary \ref{coro2.2.3} (b)). It finishes the proof of the theorem. 
\end{proof}

\begin{remark} \label{remark12.2} {\rm (a) If we drop the assumption that
$$(*)\,\,\,\,\,\,\,\text{ $\Gamma$ stabilizes a Borel subgroup of $G$,}$$
we still have the isomorphism:
\begin{equation} \label{eqn12.2.1} 
 H^0( \mathscr{P}arbun_{\mathscr{G}}(\vec{P}), \mathfrak{L}(c, \vec{\lambda})   ) \simeq   \left(\mathscr{H}_c^*\otimes V(\lambda_1)^*\otimes \cdots   \otimes    V(\lambda_s)^*  \right)^{  \fg(\Sigma\backslash \Gamma\cdot q)^\Gamma }
\end{equation}
since Theorem \ref{Thm11.3} remains valid without the assumption  ($*$). Since our Propagation Theorem (Corollary \ref{coro2.2.3}) requires the  assmption ($*$), the space on the right side of the equation \eqref{eqn12.2.1} is not known to be isomorphic with 
$ \mathscr{V}_{\Sigma, \Gamma, \phi}(\vec{p}, \vec{\lambda} )^{\dagger}$ in general.   
\vskip1ex

(b) The condition `$|\Gamma|$ divides $c$' can not, in general, be dropped since for $\lambda_i\in D_{c, q_i}$ to be  a dominant integral weight of $\mathfrak{g}^{\Gamma_{q_i}}$ imposes some divisibility condition on $c$ with respect to $\Gamma_{q_i}$ (cf. Lemma \ref{finite_set_weight_lem} and Proposition \ref{Prop10.8}). 

Also, Heinloth's example [He, Remark 19 (4)] shows that the line bundle 
$$ \mathfrak{L}^q_c \boxtimes \mathscr{L}^{q_1}(\lambda_1) \boxtimes \cdots  \boxtimes  \mathscr{L}^{q_s}(\lambda_s)  $$
does not, in general, descend to the moduli stack $\mathscr{P}arbun_{\mathscr{G}}(\vec{P})$ for an arbitrary $c$. }
\end{remark}

\end{document}